\documentclass[12pt]{amsart}

\usepackage{comment}
\usepackage[normalem]{ulem}
\usepackage[colorlinks,linkcolor=red,anchorcolor=blue,citecolor=blue]{hyperref} 
\usepackage{amscd} 
\usepackage{amsmath}
\usepackage{mathtools}
\usepackage{amssymb} 
\usepackage{amsthm}
\usepackage{dsfont}
\usepackage{bigdelim}
\usepackage{color} 
\usepackage{enumerate}
\usepackage{graphicx}
\usepackage{mathrsfs}
\usepackage{multirow}
\usepackage[all]{xy} 
\usepackage{color} 
\usepackage{enumitem}
\usepackage[dvipsnames]{xcolor}
\usepackage{bbm}
\usepackage{tikz}
\usepackage{tikz-cd}

\newtheorem{thm}{Theorem}[section] 

\newtheorem{cor}[thm]{Corollary}
\newtheorem{prop}[thm]{Proposition}

\newtheorem{deflem}[thm]{Definition-Lemma}
\newtheorem{lem}[thm]{Lemma}
\theoremstyle{definition} 
\newtheorem{defn}[thm]{Definition}

\newtheorem{ex}[thm]{Example}

\theoremstyle{remark}
\newtheorem{rem}[thm]{Remark}

\newtheorem{claim}[thm]{Claim}

\numberwithin{equation}{section}
\newtheorem{case}{Case}

\newcommand{\rk}[0]{\operatorname{rk}}

\newcommand{\id}{{\rm id}}

\newcommand{\Coker}[0]{\operatorname{Coker}}

\newcommand{\Hom}[0]{\mathscr{H}\!\textit{om}}

\newcommand{\Tor}{{\rm Tor}}

\newcommand{\reg}{{\rm{reg}}}

\newcommand{\underalign}[2]{\quad \underset{\mathclap{\strut #1}}{#2}\quad}

\newcommand{\orb}{{\rm orb}}
\newcommand{\tor}{{\rm Tor}}
  
\newcommand{\R}{\mathbb{R}}

\newcommand{\Q}{\mathbb{Q}}

\newcommand{\mO}{\mathcal{O}}

\title[Miyaoka--Yau inequality with big (anti-)canonical divisors]
{The Miyaoka--Yau inequality for singular varieties with big canonical or anticanonical divisors}

\author{Masataka IWAI}
\address{Department of Mathematics, Graduate School of Science, The University of Osaka,
1-1, Machikaneyama-cho, Toyonaka, Osaka 560-0043, Japan.}
\email{{\tt masataka@math.sci.osaka-u.ac.jp}}
\email{{\tt masataka.math@gmail.com}}

\author{Satoshi Jinnouchi}
\address{Department of Mathematics, Graduate School of Science, The University of Osaka,
1-1, Machikaneyama-cho, Toyonaka, Osaka 560-0043, Japan.}
\email{{\tt u122988d@ecs.osaka-u.ac.jp}}
\email{{\tt 20160312sti@gmail.com}}

\author{Shiyu Zhang}
\address{School of Science, Institute for Theoretical Sciences, Westlake University, Hangzhou 310030, China}
\email{{\tt zhangshiyu@westlake.edu.cn}}

\date{\today}

\begin{document}

\begin{abstract}
We establish the Miyaoka–Yau inequality for $n$-dimensional projective klt varieties with big canonical divisor $K_X$:
\[
\left(2(n+1)\widehat{c}_2(X) - n \widehat{c}_1(X)^2\right) \cdot \langle c_1(K_X)^{n-2} \rangle \ge 0.
\]
We also prove the Miyaoka--Yau inequality for K-semistable projective klt varieties with big anticanonical divisor $-K_X$.
As part of our approach, we define the non-pluripolar product $\langle \alpha_1 \cdots \alpha_p \rangle$ on singular varieties, and establish the Bogomolov--Gieseker type inequality for $\langle \alpha^{n-1} \rangle$-semistable Higgs sheaves with respect to a big class $\alpha$.
%In addition, we investigate second Chern class inequalities in the cases where $K_X$ or $-K_X$ is nef.
\end{abstract}
\maketitle
\tableofcontents
\section{Introduction}

\subsection{History and Main results}
One of the most important applications of Yau's solution of the Calabi conjecture in \cite{Yau78} is a Chern number inequality for compact K\"ahler manifolds with ample canonical bundle. More precisely, let \( X \) be an \( n \)-dimensional compact K\"ahler manifold with ample canonical divisor \( K_X \). By Yau's theorem, \( X \) admits a K\"ahler--Einstein metric with negative Ricci curvature. Using this metric, Yau \cite{Yau77} proved the following inequality:
\begin{equation}\label{equa-intro-yau}
    \big(2(n+1)c_2(X)-n c_1(X)^2\big) \cdot K_X^{n-2} \geq 0.
\end{equation}
When \( n=2 \), inequality \eqref{equa-intro-yau} becomes
\[
    3c_2(X)-c_1(X)^2 \geq 0,
\]
which is the classical \textit{Bogomolov--Miyaoka--Yau inequality} for surfaces. For surfaces, this inequality is known to hold in a more general setting. Indeed, Miyaoka~\cite{Miy77} proved that every smooth projective surface of general type satisfies $3c_2(X)-c_1(X)^2 \geq 0$. In other words, for surfaces, the inequality remains valid when \( K_X \) is only assumed to be big, without assuming that it is ample.

It is therefore natural to ask whether, in higher dimensions, the inequality \eqref{equa-intro-yau} remains valid when \(K_X\) is only assumed to be big and \(X\) is allowed to have mild singularities. When \( K_X \) is nef and big, the Miyaoka--Yau inequality has already been established in \cite{GKPT19b, GT22}. However, when \( K_X \) is only big, the class $K_X^{n-2}$ may be  negative, and hence the classical form of the Miyaoka--Yau inequality cannot be expected in general (for explicit examples, see Lemma \ref{lem-counterexample-MY}).

To address this issue, we use non-pluripolar products, which were introduced in \cite{BEGZ10} in the study of Monge--Amp\`ere equations for big cohomology classes. Roughly speaking, for a class $\alpha$ on a compact K\"ahler manifold, the non-pluripolar product $\langle \alpha^{n-2} \rangle$ is defined by taking the product after removing the negative part of $\alpha$. This product agrees with the movable intersection product introduced in \cite{BDPP13}.

We extend non-pluripolar products to singular varieties and use them to formulate the Miyaoka--Yau inequality when $K_X$ is only big. This gives the following new form of the Miyaoka--Yau inequality.

\begin{thm}
\label{thm-main-big-canonical}
Let $X$ be an $n$-dimensional compact normal complex analytic variety. Assume one of the following conditions holds:
\begin{enumerate}[label=$(\arabic*)$]
    \item $X$ is a projective klt variety with big canonical divisor $K_X$;
    \item $X$ has at most quotient singularities in codimension $2$ and rational singularities. Furthermore, $K_X$ is a $\mathbb{Q}$-Cartier divisor, and there exists a resolution $\pi: \widetilde{X} \to X$ such that $K_{\widetilde{X}}$ is big.
\end{enumerate}
Then the following Miyaoka--Yau inequality holds:
$$
\left(2(n+1)\widehat{c}_2(X) - n \widehat{c}_1(X)^2\right) \cdot \langle c_1(K_X)^{n-2} \rangle \ge 0.
$$
\end{thm}

Here, $\widehat{c}_2(X)$ and $\widehat{c}_1(X)^2$ denote the orbifold Chern classes introduced in \cite{Kaw92, GKPT19b, GK20}. When $X$ is smooth, these classes coincide with the usual Chern classes. At present, orbifold Chern classes can be defined for klt varieties, more precisely for varieties with quotient singularities in codimension two. Therefore, the assumptions on the singularities in Theorems \ref{thm-main-big-canonical} and \ref{thm-main-big-anticanonical} cannot currently be weakened any further. Note that these theorems recover the Miyaoka--Yau inequality proved in \cite{GKPT19b, GT22}, because if $K_X$ is nef, then $\langle c_1(K_X)^{n-2} \rangle = c_1(K_X)^{n-2}$. Moreover, when $X$ is a smooth threefold, the non-pluripolar product $\langle c_1(K_X)^{n-2} \rangle$ coincides with the positive part $P(K_X)$ of the divisorial Zariski decomposition of $K_X$, as defined in \cite{Bou04} and \cite{Nak04}.

Next, we consider the case where \(-K_X\) is big. In this case, it is known that the Miyaoka--Yau inequality does not hold in general, even when \(-K_X\) is ample (for example, see \cite{GKP22}). However, under the assumption of K-semistability, the Miyaoka--Yau inequality holds by \cite{GKP22, DGP24}. In this paper, we extend this result to the case where $-K_X$ is only big.

\begin{thm}
\label{thm-main-big-anticanonical}
Let $X$ be an $n$-dimensional compact normal complex analytic variety. Assume one of the following conditions holds:
\begin{enumerate}[label=$(\arabic*)$]
    \item $X$ is a K-semistable projective klt variety with big anticanonical divisor $-K_X$;
    \item $X$ has at most quotient singularities in codimension $2$. Furthermore, $-K_X$ is a $\mathbb{Q}$-Cartier divisor, and there exists a resolution $\pi: \widetilde{X} \to X$ such that $\widetilde{X}$ is a smooth $K$-semistable projective variety with $-K_{\widetilde{X}}$ big.
\end{enumerate}
Then the following Miyaoka--Yau inequality holds:
$$
\left(2(n+1)\widehat{c}_2(X) - n \widehat{c}_1(X)^2\right) \cdot \langle c_1(-K_X)^{n-2} \rangle \ge 0.
$$
\end{thm}

\subsection{Strategy of the proof}
The strategy for Theorems \ref{thm-main-big-canonical} and \ref{thm-main-big-anticanonical} is to establish the following Bogomolov--Gieseker  inequality. 
\begin{thm}[$\subset$ Theorem \ref{thm-BGinequality-nonpluripolar}]
\label{thm-BGinequality-main}
Let $X$ be an $n$-dimensional compact normal Moishezon variety with at most quotient singularities in codimension 2 and rational singularities. $($For example, $X$ is a projective klt variety.$)$

If a rank $r$ reflexive Higgs sheaf $(\mathcal{E}, \theta)$ is $\langle \alpha^{n-1} \rangle$-semistable for some big class $\alpha \in H^{1,1}_{BC}(X)$,  
then the following Bogomolov--Gieseker inequality holds:
$$ 
\left( 2r \widehat{c}_2(\mathcal{E}) - (r-1)\widehat{c}_1(\mathcal{E})^2\right)\cdot\langle\alpha^{n-2}\rangle \ge 0.
$$
\end{thm}

We now explain the approach to this inequality.
When $X$ is smooth, we can apply Demailly's approximation theorem to obtain a sequence of closed positive currents $T_k$ with analytic singularities such that $\{ \langle T_k \rangle\} \to \langle \alpha \rangle$ as $k \to \infty$.  
For sufficiently large $k$, we can take a suitable blow-up $\pi: X_k \to X$ and a semipositive form $\tau$ on $X_k$ such that $\langle \pi^* T_k^{n-2} \rangle = \tau^{n-2}$, and the  pullback of $\mathcal{E}$ is $\tau^{n-1}$-stable. 
By applying the usual Bogomolov--Gieseker inequality and taking the limit as $k \to \infty$, we obtain the desired inequality in Theorem \ref{thm-BGinequality-main}.

However, when $X$ is singular, the situation becomes more subtle. 
First, Demailly's approximation theorem is not available on singular analytic varieties, since it is not known whether the Ohsawa–Takegoshi extension theorem holds in this setting. 
Second, the intersection number between $\widehat{c}_2(\mathcal{E})$ and $\langle \alpha^{n-2} \rangle$ cannot be defined in the usual sense on $X$. While the non-pluripolar product $\langle \alpha^{n-2} \rangle$ is still well-defined via a resolution of singularities (cf.~Section~\ref{sec-nonpluripolar-singular}), it can only be interpreted as a homology class in $H_{2n - 4}(X, \mathbb{R})$. On the other hand, the orbifold Chern class $\widehat{c}_2(\mathcal{E})$, as defined in \cite{GK20}, lies in $H_4(X, \mathbb{R})$. Hence, their intersection number is not well-defined in the usual sense. For this reason, the singular case requires a more delicate analysis.

To address this issue, we use the orbifold modification  by \cite{Ou24} and \cite{KO25}. 
According to \cite{Ou24} and \cite{KO25}, if $X$ has quotient singularities in codimension 2, then there exists a bimeromorphic map  
$q: Z \to X$ from a normal analytic variety $Z$ that admits a complex orbifold structure $Z_{\mathrm{orb}}$.  
Using this modification $q: Z \to X$, we define the intersection number between $\widehat{c}_2(\mathcal{E})$ and $\langle \alpha^{n-2} \rangle$ (cf.~ Section~\ref{sec-orbifold}).  
This reduces the problem to the orbifold setting, where analytic tools such as Demailly's approximation theorem, as extended in \cite{Wu23}, and the Bogomolov--Gieseker inequality from \cite{Ou25b, ZZZ25} are available.

Once the Bogomolov--Gieseker inequality in Theorem~\ref{thm-BGinequality-main} is established, we can prove the Miyaoka–Yau inequality in Theorem~\ref{thm-main-big-canonical} by applying  $\langle c_1(K_X)^{n-1} \rangle$-stability introduced by the second author in \cite{Jin25}.  
Under the assumptions of Theorem~\ref{thm-BGinequality-main}, the method in \cite{Jin25} shows that the reflexive cotangent sheaf $\Omega_X^{[1]}$ is $\langle c_1(K_X)^{n-1} \rangle$-semistable.  
Therefore, the sheaf $\Omega_X^{[1]} \oplus \mathcal{O}_X$ admits a $\langle c_1(K_X)^{n-1} \rangle$-stable Higgs structure, which yields the desired Miyaoka–Yau inequality.  
For Theorem~\ref{thm-main-big-anticanonical}, we use the canonical extension sheaf instead of a Higgs sheaf.

The structure of this paper is as follows.  
In Section~\ref{sec-pre}, we review basic concepts related to normal complex analytic varieties.  
In Section~\ref{sec-nonpluripolar-singular}, we define the non-pluripolar product and the associated slope on singular varieties in Fujiki's class.  
Section~\ref{sec-orbifold} recalls some properties of orbifolds and defines the intersection number between $\widehat{c}_2(\mathcal{E})$ and $\langle \alpha^{n-2} \rangle$, leading to the proof of Theorem~\ref{thm-BGinequality-main}.  
In Section~\ref{sec-MY-inequality}, we prove Theorems~\ref{thm-main-big-canonical} and \ref{thm-main-big-anticanonical}.

\subsection{Acknowledgments}
The authors would like to express their gratitude to the anonymous referee for their valuable comments and suggestions.
M.\,I.\ and S.\,J.\ express their gratitude to Prof.\ Ryushi Goto for valuable discussions and for answering their questions.  
They also thank the organizers of the ``What is ...?" seminar held at Osaka University, where discussions served as the starting point for this work.  
M.\,I.\ is especially grateful to Prof.\ Wenhao Ou, Prof.\ Mihai P\u{a}un, and Prof.\ David Witt Nyström for kindly answering his questions.  
S.\,Z.\ would like to express his gratitude to Prof. Xi Zhang for constant encouragement.
M.\,I.\ was supported by the Grant-in-Aid for Early Career Scientists, No.\ 22K13907.
\section{Preliminary results}\label{sec-pre}

\subsection{Notation and conventions}\label{subsec-notation}

In this paper, we denote by $\mathbb{N}$ the set of non-negative integers. For any real vector space $M$, we write $M^\vee$ for its dual.

All complex spaces are assumed to be \textit{second countable}.
A \textit{complex analytic variety} or simply an \textit{analytic variety} is a reduced and irreducible complex space. Unless otherwise stated, an analytic variety $X$ has complex dimension $n$. 
We follow the standard notations and conventions used in \cite{KM98} about Minimal Model Program. Additionally, we say that a normal analytic variety $X$ is a \emph{klt variety} if the pair $(X, 0)$ is klt; equivalently, $X$ is log terminal.

Unless stated otherwise, all sheaves considered in this paper are assumed to be coherent.
Let $\mathcal{E}$ be a torsion-free (coherent) sheaf on a normal analytic variety $X$.
We define the dual reflexive sheaf $\mathcal{E}^\vee$ by
\(
\mathcal{E}^\vee \coloneqq \Hom(\mathcal{E}, \mathcal{O}_X).
\)
For any positive integer $m \in \mathbb{N}$, we define the reflexive tensor power by
\(
\mathcal{E}^{[\otimes m]} \coloneqq (\mathcal{E}^{\otimes m})^{\vee\vee}.
\)
Given a morphism $f \colon Y \rightarrow X$ between analytic varieties, the reflexive pullback of $\mathcal{E}$ is defined as
\(
f^{[*]}\mathcal{E} \coloneqq (f^*\mathcal{E})^{\vee\vee}.
\)
A reflexive rank-one sheaf $\mathcal{E}$ is called a \textit{$\mathbb{Q}$-line bundle} if there exists an integer $m \in \mathbb{N}$ such that $\mathcal{E}^{[m]}$ is locally free. We denote the torsion part of a sheaf by $\tor$.

\subsection{Bott-Chern cohomology group and notions of positivity}
\subsubsection{Definition of differential forms and currents on complex spaces}

Let $X$ be an $n$-dimensional reduced equidimensional complex space.
According to \cite[Definition 1.1]{Dem85}, a \emph{$(p,q)$-form} $\omega$ on $X$ is defined as a $(p,q)$-form on $X_{\reg}$ such that, for every point $x \in X$, there exists an open neighborhood $U \subset X$ of $x$, a closed embedding $i_U \colon U \hookrightarrow V \subset \mathbb{C}^N$ (where $V$ is an open subset of $\mathbb{C}^N$), and a $(p,q)$-form $\omega_V$ on $V$ satisfying
\[
\omega|_{U \cap X_{\mathrm{reg}}} = \omega_V|_{U \cap X_{\reg}}.
\]
Similarly, differential $d$-forms are defined in the same manner.
As in the smooth case, we can define bidegree $(p,q)$-currents and degree $d$-currents on complex spaces. Furthermore, the differential operators $d$, $\partial$, and $\overline{\partial}$ can be defined in the same way, along with the notions of closed and positive currents.
We also introduce the differential operator $d^c$ such that
\[
dd^c = \frac{\sqrt{-1}}{2\pi} \partial \overline{\partial}.
\]
Throughout this paper, we will use the notation $dd^c$ instead of $\frac{\sqrt{-1}}{2\pi} \partial \overline{\partial}$.

Let $f \colon Y \to X$ be a morphism between reduced equidimensional complex spaces.  
According to \cite[Lemma 1.3]{Dem85}, for any $(p,q)$-form $\omega$ on $X$, we can define the pullback $f^*\omega$, which is a $(p,q)$-form on $Y$.
Moreover, if \( f \) is proper (i.e., the preimage of any compact set under \( f \) is compact), then for any current \( T \) on \( Y \), the push-forward \( f_* T \) is well-defined as a current on \( X \).
%then for any $(p,q)$-current $T$ on $Y$, the push-forward\(f_*T\) is a $(p,q)$-current on $X$, defined by\[f_*T (\omega) := T(f^*\omega),\]for any $(n-p,n-q)$-form $\omega$ with compact support on $X$. 

According to \cite[Definition 1.5]{Dem85}, a function $u \colon X \to [-\infty, +\infty)$ is called \textit{plurisubharmonic} (in short \textit{psh}) if $u \not\equiv -\infty$ and, for every open subset $U \subset X$ admitting a closed embedding $i_U \colon U \hookrightarrow V \subset \mathbb{C}^N$ (where $V$ is an open subset of $\mathbb{C}^N$), there exists a psh function $\varphi$ on $V$ such that
\(
u = \varphi|_U.
\)
As in the smooth case, a function $u \colon X \to [-\infty, +\infty)$ is said to be \textit{quasi-plurisubharmonic} (in short \textit{quasi-psh}) if it can be locally written as the sum of a smooth function and a psh function.

\subsubsection{de Rham cohomology and singular cohomology groups}

We briefly recall the de Rham cohomology and singular cohomology groups following \cite[Section 1]{Wu21}.
Let $X$ be an $n$-dimensional compact analytic variety. Then we have
\[
H_{0}(X, \R) \cong \R 
\quad \text{and} \quad
H_{2n}(X, \R) \cong \R.
\]
(See also \cite[Subsection 19.1]{Ful93} and \cite[Subsection 4.3]{DO23}.) 
As shown in \cite[Section 1]{Wu21}, $X$ admits a triangulation, so the sheaf cohomology groups $H^p(X, \R)$ are naturally isomorphic to the singular (or simplicial) cohomology groups.

We next examine the relationship between homology and de Rham cohomology by currents. By the universal coefficient theorem, we obtain the isomorphism
\[
H_p(X, \R) \cong H^p(X, \R)^\vee.
\]
We denote the de Rham cohomology of differential forms by $H^p_{dR}(X, \R)$ and that of currents by $H^p_{dR, \mathcal{D}'}(X, \R)$. According to \cite{Ser55}, there is a duality
\[
H^{2n-p}_{dR, \mathcal{D}'}(X, \R) \cong H^{p}_{dR}(X, \R)^\vee.
\]
Furthermore, from \cite{Her67}, there exists a natural surjective map
\[
H^p_{dR}(X, \R) \twoheadrightarrow H^p(X, \R),
\]
which, however, may fail to be an isomorphism in general (see \cite[Remark 2]{Wu21}).
Combining these facts, we obtain the following injective map:
\begin{equation}
\label{eq-current-homology}
\Psi \colon H_p(X, \R) \hookrightarrow H^{2n - p}_{dR, \mathcal{D}'}(X, \R).
\end{equation}

\subsubsection{Bott-Chern cohomology group $H^{1,1}_{BC}(X)$ on complex analytic varieties}

Let \(X\) be a normal analytic variety. Denote by
\(\mathcal C_X^\infty\) (resp. \(\mathcal D'_X, \mathcal H_X\))
the sheaves of smooth functions (resp. distributions, real-valued
pluriharmonic functions) on \(X\). Note that \(\mathcal H_X\) is a sheaf of Abelian groups under addition. According to \cite[Definition 4.6.2]{BG13}, a
\((1,1)\)-form (resp. current) with local potentials is defined as a section
of \(H^0(X,\mathcal C_X^\infty/\mathcal H_X)\)
(resp. \(H^0(X,\mathcal D'_X/\mathcal H_X)\)). By definition, any such
\((1,1)\)-form (resp. current) with local potentials can be locally written
as \(\alpha=dd^c u\) for some smooth function (resp. distribution) \(u\).
If the current is positive, the local potential may be chosen to be psh.
Following \cite[Definition 4.6.2]{BG13} and \cite[Definition 3.1]{HP16}, the \textit{Bott-Chern cohomology} is defined by
\[
H^{1,1}_{BC}(X) := H^1(X, \mathcal{H}_X).
\]
Consider the following short exact sequence:
\[
0 \to \mathcal{H}_X \to \mathcal{C}^{\infty}_{X} \to \mathcal{C}^{\infty}_X / \mathcal{H}_X \to 0,
\]
which induces a surjective map
\[
H^0(X, \mathcal{C}^{\infty}_X / \mathcal{H}_X) \twoheadrightarrow H^{1,1}_{BC}(X).
\]
Hence, any closed $(1,1)$-form $\alpha$ (resp.~ closed $(1,1)$-current $T$) with local potentials defines a class $\{ \alpha \}$ (resp.~ $\{T\}$) in $H^{1,1}_{BC}(X)$. Conversely, every element of $H^{1,1}_{BC}(X)$ can be represented in this way.

We now define the positivity of Bott-Chern cohomology classes.
Before proceeding, we introduce several definitions and notations.
\begin{defn}
\label{def-BC-positivity}
(\cite[Definition 2.2]{HP16}, \cite[Definitions 2.34]{DH20}) Let $X$ be a normal analytic variety.
\begin{enumerate}[label=$(\arabic*)$]
\item A smooth $(1,1)$-form $\omega$ is said to be \emph{Hermitian} if $\omega$ is smooth and positive definite.
\item A closed $(1,1)$-form $\omega$ with local potentials is called \emph{K\"ahler} if it is locally of the form $\omega = dd^c u$ for some smooth strictly psh function $u$. The variety $X$ is called \emph{K\"ahler} if it admits a K\"ahler form.
\item A closed positive $(1,1)$-current $T$ with local potentials is called a \emph{K\"ahler current} if there exists a smooth Hermitian form $\omega$ on $X$ such that $T \ge \omega$.
\item Let $E$ be a prime divisor and $T$ a closed positive $(1,1)$-current with local potentials. The \emph{Lelong number} of $T$ along $E$ is defined by
\[
\nu(T, E) := \inf_{x \in E} \nu(T, x).
\]
(For the precise definition of the Lelong number, see \cite[Section 2]{Dem12} or \cite[Definition 4.4]{HP24}.)
Since $X$ is normal, it follows that
\[
\nu(T, E) = \inf_{x \in E \cap X_{\mathrm{reg}}} \nu(T, x).
\]
That is, for a very general point $x \in E \cap X_{\reg}$, we have $ \nu(T, E) = \nu(T,x)$.
\end{enumerate}
\end{defn}
\begin{defn}
(\cite[Definition 2.2]{Bou04}, \cite[Chapter 6]{Dem12}, \cite[Definition 2.2 and 2.34]{DH20})
Let $X$ be a normal analytic variety and $\omega$ be a Hermitian $(1,1)$-form.  
For a class $\alpha \in H^{1,1}_{BC}(X)$, we define the following notions:
\begin{enumerate}[label=$(\arabic*)$]
\item $\alpha$ is called \emph{K\"ahler} if it contains a K\"ahler form.
\item $\alpha$  is called \emph{nef} if it can be represented by a $(1,1)$-form $\eta$ with local potentials such that, for every $\varepsilon > 0$, there exists a smooth function $f_{\varepsilon}$ satisfying
\(\eta + dd^c f_{\varepsilon} \geq -\varepsilon \omega.
\)
\item $\alpha$  is called \emph{big} if it contains a K\"ahler current.
\item $\alpha$ is called \emph{pseudo-effective} (or \emph{psef} for short) if it contains a closed positive $(1,1)$-current with local potentials.

\item $\alpha$  is called \emph{modified K\"ahler} if it contains a K\"ahler current $T$ such that the Lelong number satisfies
\(
\nu(T, E) = 0
\)
for all prime divisors $E$ on $X$.
\item $\alpha$  is called \emph{modified nef} if, for any $\varepsilon > 0$, it contains a closed positive $(1,1)$-current $T_{\varepsilon}$ with local potentials such that
\[
T_{\varepsilon} \geq -\varepsilon \omega
\quad \text{and} \quad
\nu(T_{\varepsilon}, E) = 0
\]
for all prime divisors $E$ on $X$.
\end{enumerate}
\end{defn}

\subsubsection{Pullback of the Bott-Chern class}
\label{subsec-intersection-BC}

Let $f \colon Y \to X$ be a proper bimeromorphic morphism between normal analytic varieties. As shown in the argument of \cite[Lemma 3.3]{HP16}, by the five-term exact sequence of the Leray spectral sequence, we can define a pullback map 
$$f^{*} \colon H^{1,1}_{BC}(X) \hookrightarrow H^{1,1}_{BC}(Y),$$
even if $X$ and $Y$ do not necessarily have rational singularities.
Note that, although a morphism $f\colon X \to Y$ may be proper, we cannot generally define a pushforward $f_{*} \colon H^{1,1}_{BC}(Y) \to H^{1,1}_{BC}(X)$, since it is unclear whether the pushforward of a $(1,1)$-current with local potentials on $Y$ still admits local potentials on $X$. 

Let $\mathbb{R}_X$ denote the constant sheaf with values in $\mathbb{R}$. According to \cite[Subsection 3.1]{GK20}, we have the following short exact sequence:
\begin{equation*}
0 \to \mathbb{R}_X \xrightarrow{\sqrt{-1}} \mathscr{O}_X \xrightarrow{\mathrm{Re}} \mathcal{H}_{X} \to 0.
\end{equation*}
This induces the following connecting homomorphism:
\begin{equation}
\label{eq-connect-h2}
\delta^1 \colon H^{1,1}_{BC}(X)=H^1(X, \mathcal{H}_{X}) \to H^2(X, \mathbb{R}_X)
\end{equation}
The map $\delta^1$ is not necessarily injective, but it is injective if $X$ has rational singularities.
Using the connecting homomorphism $\delta^1$ in \eqref{eq-connect-h2}, we define the intersection number between Bott-Chern classes and singular cohomology classes. For classes $\alpha_1, \ldots, \alpha_p \in H^{1,1}_{BC}(X)$ and $\sigma \in H^{2n - 2p}(X, \mathbb{R})$, we define:
$$
\alpha_1 \cdots \alpha_p \cdot \sigma 
:= 
\delta^1(\alpha_1) \cdots \delta^1(\alpha_p) \cdot \sigma.
$$
Then, for any resolution $\pi \colon \widetilde{X} \to X$, we obtain:
\begin{equation}
\label{eq-pullback-intersection}
\alpha_1 \cdots \alpha_p \cdot \sigma 
=
\pi^{*}\alpha_1 \cdots \pi^{*}\alpha_p \cdot \pi^{*}\sigma,
\end{equation}
since the following diagram commutes (even if $X$ does not necessarily have rational singularities):
$$
\xymatrix@C=25pt@R=20pt{
H^{1,1}_{BC}(X) \ar[r]^{\pi^{*}} \ar[d]_{\delta^1}
& H^{1,1}_{BC}(\widetilde{X}) \ar[d]^{\delta^1} \\
H^2(X, \mathbb{R}) \ar[r]^{\pi^{*}} & H^2(\widetilde{X}, \mathbb{R}) \\
}
$$

\subsubsection{Fujiki's class and Moishezon varieties}
\label{subsubsec-Fujiki}

\begin{defn}(\cite[Definition 2.2]{DH20})
Let $X$ be a compact normal analytic variety.  
We say that $X$ belongs to \textit{Fujiki's class} (resp.~ $X$ is \textit{Moishezon}) if $X$ is bimeromorphic to a compact K\"ahler manifold (resp.~ a complex projective manifold).
\end{defn}

By definition, every Moishezon variety belongs to Fujiki's class.  
If $X$ admits a big class (resp.~ a big Cartier divisor), then $X$ belongs to Fujiki's class (resp.~ $X$ is Moishezon).  
However, it remains unknown whether the converse holds when $X$ is singular (see \cite[Remark 2.3]{DH23}).

Certain assumptions (such as being in Fujiki's class, being Moishezon, and others) are required for the following reasons:

\begin{itemize}
  \item The \textit{Fujiki} condition is necessary to define the non-pluripolar product (cf.~ Section \ref{sec-nonpluripolar-singular}).
  %and to establish the generic nefness theorem of (co-)tangent sheaves (cf.~ Section \ref{sec-Miyaoka's-inequality}). 
  At present, this condition cannot be removed.
  
  \item The \textit{Moishezon} condition is used to establish a \textit{vanishing property} of a big class (see Definition \ref{defn-vanishing-exceptional}). However, it is expected that this assumption may eventually be removed (cf.~ Lemma \ref{lem-VP} and Remark \ref{assumption-remark}).
  
  \item The assumption of having \textit{rational singularities} is important when working with Higgs sheaves, particularly when defining their pullbacks (cf.~ Subsection \ref{subsub-slope-Higgs}).
  
  \item The assumption of \textit{quotient singularities} is required when applying results that are currently only available in the orbifold case (for example, Demailly’s approximation theorem in Subsection \ref{subsec-BG-nonpluripolar}).
  
  \item The assumption of \textit{quotient singularities in codimension two} is essential to define the orbifold second Chern class for reflexive sheaves (cf.~ Subsection \ref{subsec-intersection-orbifold-nonpluripolar} and \ref{subsec-BG-nonpluripolar}). 
  %Since any klt variety always has a quotient singularity in codimension two by \cite[Lemma 5.8]{GK20}, the orbifold second Chern class can be defined on klt varieties.
\end{itemize}

\section{Non-pluripolar product on singular varieties}
\label{sec-nonpluripolar-singular}

In this section, we discuss how to define the non-pluripolar product for closed positive $(1,1)$-currents and psef classes on compact normal analytic varieties in Fujiki's class.  
When $X$ is an $n$-dimensional compact normal K\"ahler variety, the $n$-th non-pluripolar product $\langle \alpha_1 \cdots \alpha_n \rangle$ has already been defined in \cite[Subsection 1.2]{BBEGZ19}. 
However, in this paper, we do not follow the approach of \cite[Subsection 1.2]{BBEGZ19}; instead, we define the non-pluripolar product via resolution of singularities.

\subsection{Definition of the non-pluripolar product on singular varieties}
\subsubsection{Non-pluripolar product of closed positive currents on singular varieties}

Let $X$ be a compact normal analytic variety in Fujiki's class, and let $\theta \in H^{1,1}_{BC}(X)$ be a psef class.  
A quasi-psh function $\varphi$ is said to be \textit{$\theta$-psh} if $\theta + dd^c \varphi$ is a closed positive $(1,1)$-current with local potentials.

Let $T$ be a $(1,1)$-current with local potentials in the class $\theta$. Then there exists a $\theta$-psh function $\varphi$ such that $T = \theta + dd^c \varphi$ (this function $\varphi$ is often called a \textit{global potential} of $T$).  
Given a resolution $\pi \colon \widetilde{X} \to X$ such that $\widetilde{X}$ is a compact K\"ahler manifold, we define the pullback of $T$ by
$$
\pi^{*}T := \pi^{*}\theta + dd^{c}\pi^{*}\varphi.
$$
Note that the function $\pi^{*}\varphi$ is $\pi^{*}\theta$-psh. Furthermore, if $T$ is closed (resp.~ positive), then so is $\pi^{*}T$ (cf.~ \cite[Subsection 2.4.1]{DHP24}).

Given closed positive $(1,1)$-currents $T_1, \ldots, T_p$ with local potentials on $X$, we define their non-pluripolar product as
$$
\langle T_1 \wedge \cdots \wedge T_p\rangle
:=
\pi_{*} \langle \pi^{*}  T_1 \wedge \cdots \wedge \pi^{*} T_p\rangle,
$$
where $\pi \colon \widetilde{X} \to X$ is a resolution such that $\widetilde{X}$ is a compact K\"ahler manifold.

\begin{lem}
\label{lem-independence-nonpluripolar}
The current $\langle T_1 \wedge \cdots \wedge T_p \rangle$ is a closed positive $(p,p)$-current and is independent of the choice of resolution $\pi$.
\end{lem}

\begin{proof}
Since $\langle \pi^{*}T_1 \wedge \cdots \wedge \pi^{*} T_p \rangle$ is a closed positive $(p,p)$-current, its pushforward by $\pi$ is also closed and positive by \cite[Lemma 1.4 below]{Dem85}.  
The independence of the resolution follows from the fact that for any modification $\rho \colon W \to \widetilde{X}$, it holds by \cite[Remark 1.7]{BEGZ10} that
$$
\langle \pi^{*}  T_1 \wedge \cdots \wedge \pi^{*} T_p\rangle
=
\rho_{*} \langle \rho^{*}\pi^{*}  T_1 \wedge \cdots \wedge \rho^{*}\pi^{*} T_p\rangle.
$$
\end{proof}

Next, we consider when these currents have small unbounded loci.

\begin{defn}(\cite[Definition 1.2]{BEGZ10})
A subset $A \subset X$ is called \textit{locally complete pluripolar} if locally it coincides with the polar set of some psh function.

A current $T$ is said to have \textit{small unbounded locus} if there exists a Euclidean closed subset $A \subset X$, which is locally complete pluripolar, such that the global potential of $T$ is locally bounded on $X \setminus A$.
\end{defn}

Suppose there exist closed positive $(1,1)$-currents $T_1, \ldots, T_p$ with local potentials, and a Euclidean closed locally complete pluripolar set $A \subset X$ such that the unbounded locus of each $T_i$ is contained in $A$.
Moreover, we may assume that $X_{\mathrm{sing}} \subset A$.  
Since $X \setminus A \subset X_{\mathrm{reg}}$, the Bedford-Taylor product $T_1 \wedge \cdots \wedge T_p$ is well-defined on $X \setminus A$.

\begin{lem}
\label{lem-nonpluripolar-unbounded}
$$
\langle T_1 \wedge \cdots \wedge T_p \rangle = 1_{X \setminus A} \, T_1 \wedge \cdots \wedge T_p.
$$
\end{lem}

\begin{proof}
Fix a resolution $\pi \colon \widetilde{X} \to X$.  
Then the pullbacks $\pi^{*}T_i$ have unbounded loci contained in $\pi^{-1}(A)$.  
Hence, by \cite[Definition 1.2 below]{BEGZ10}, we have
$$
\langle \pi^{*} T_1 \wedge \cdots \wedge \pi^{*} T_p \rangle = 1_{\widetilde{X} \setminus \pi^{-1}(A)} \, \pi^{*} T_1 \wedge \cdots \wedge \pi^{*} T_p.
$$
Since $\pi \colon \widetilde{X} \setminus \pi^{-1}(A) \to X \setminus A$ is biholomorphic, we conclude that
$$
\pi_{*} \langle \pi^{*} T_1 \wedge \cdots \wedge \pi^{*} T_p \rangle = T_1 \wedge \cdots \wedge T_p \quad \text{on } X \setminus A.
$$
As the left-hand side equals $\langle T_1 \wedge \cdots \wedge T_p \rangle$, the result follows immediately.
Note that $\langle T_1 \wedge \cdots \wedge  T_p \rangle$ also puts no mass on locally complete pluripolar subsets from \cite[Subsection 1.2]{BEGZ10}.
\end{proof}

\subsubsection{Minimal singular currents on singular varieties}

Let $X$ be a compact normal analytic variety, and let $\alpha \in H^{1,1}_{BC}(X)$ be a psef class.  
Let $\theta$ be a smooth $(1,1)$-form with local potentials such that $\alpha = \{ \theta \}$.  
For any two closed positive $(1,1)$-currents $T_1$ and $T_2$ with local potentials in $\alpha$,
we say that $T_1$ is \textit{less singular} than $T_2$ if their global potentials satisfy
$$
\varphi_2 \le \varphi_1 + O(1).
$$
A closed positive $(1,1)$-current $T_{\min}$ with local potentials in $\alpha$ is said to have \textit{minimal singularities} if it is less singular than any other closed positive $(1,1)$-current with local potentials in $\alpha$.

\begin{lem}[{cf.~ \cite[Chapter 6]{Dem12}}]
If $\alpha$ is psef, then there exists a closed positive $(1,1)$-current with local potentials $T_{\min}$ in the class $\alpha$ that has minimal singularities.
\end{lem}

\begin{proof}
The construction follows that in \cite[Chapter 6]{Dem12} or \cite[Subsection 1.4]{BEGZ10}. Define
$$
V_{\theta} := \sup \{ \varphi \mid \varphi \text{ is $\theta$-psh and } \varphi \le 0 \text{ on } X \}.
$$
We prove that $V_{\theta}$ is $\theta$-psh.  
Once this is established, setting $T_{\min} :=  \theta + dd^c V_{\theta}$ gives the desired current.

The restriction $V_{\theta}|_{X_{\reg}}$ is quasi-psh.   
Since $V_{\theta}$ is locally bounded from above on $X$, we can take its upper semicontinuous regularization $V^{*}$, as given by \cite[Theorem 1.7]{Dem85}:
\[
V^{*}(y) := \limsup_{x \in X_{\reg},\ x \to y} V_{\theta}(x).
\]
Note that $V^{*}$ is a $\theta$-psh function and provides the unique extension of $V_{\theta}|_{X_{\reg}}$.  
Since $V^{*}$ is $\theta$-psh on $X$, we have $V^{*} \le V_{\theta}$. On the other hand, for any $\theta$-psh function $\varphi$ with $\varphi \le 0$ on $X$, we have $\varphi \le V^{*}$ on $X_{\reg}$, hence $\varphi \le V^{*}$ on $X$ by Lemma \ref{lem-psh-Zariski-open} below.  
Therefore, $V_{\theta} \le V^{*}$, and thus $V_{\theta} = V^{*}$. This shows that $V_{\theta}$ is $\theta$-psh.
\end{proof}

\begin{lem}
\label{lem-psh-Zariski-open}
Let $u$ and $v$ be $\theta$-psh functions, and let $Z$ be an analytic set of $X$. If $u \le v$ holds on $X \setminus Z$, then $u \le v$ holds on $X$.
\end{lem}

\begin{proof}
We may assume $X_{\mathrm{sing}} \subset Z$. Let $\pi \colon \widetilde{X} \to X$ be a resolution.  
Then $\pi^{*}u \le \pi^{*}v$ holds on a Zariski open subset of $\widetilde{X}$.  
Since both $u$ and $v$ are $\theta$-psh, their pullbacks are $\pi^{*}\theta$-psh, and the inequality extends to  $\widetilde{X}$ by the upper semicontinuity of psh functions. 
Hence, the inequality $u \le v$ holds on $X$.
\end{proof}

\begin{prop}[{cf.~ \cite[Proposition 1.12]{BEGZ10}}]
\label{prop-BEGZ10-1.12}
Let $\pi \colon Y \to X$ be a surjective morphism between compact normal analytic varieties.  
If $\varphi$ is a $\theta$-psh function with minimal singularities, then $\varphi \circ \pi$ is a $\pi^{*}\theta$-psh function with minimal singularities.
\end{prop}

\begin{proof}
The proof follows that in \cite[Proposition 1.12]{BEGZ10}.  
Let $\psi$ be a $\pi^{*}\theta$-psh function and 
$X^{\circ} \subset X_{\reg}$ be the set of regular values of $\pi$.
For each $x \in X^{\circ}$, define
$$
\tau(x) := \sup_{y \in \pi^{-1}(x)} \psi(y).
$$
Then $\tau$ is a $\theta$-psh function on $X^{\circ}$. Since $Y$ is compact, $\psi$ is bounded above, so $\tau$ is bounded above as well.  
By \cite[Theorem 1.7]{Dem85}, $\tau$ extends to a psh function on $X$.  
By assumption, we have $\tau \circ \pi \le \varphi + O(1)$, which implies that
\[
\psi \le \varphi \circ \pi + O(1) \quad \text{ on } \pi^{-1}(X^{\circ}).
\]
By Lemma~\ref{lem-psh-Zariski-open}, this inequality extends to all of $Y$.  
Therefore, $\varphi \circ \pi$ is a $\pi^*\theta$-psh function with minimal singularities.
\end{proof}

\begin{lem}
\label{lem-nonpluripolar-minimal}
Assume that $X$ is in Fujiki's class.
Let $\alpha_1, \ldots, \alpha_p \in H^{1,1}_{BC}(X)$ be psef classes, and let $T_{i,\min} \in \alpha_i$ be currents with minimal singularities. Then
$$
\langle T_{1, \min} \wedge \cdots \wedge T_{p,\min} \rangle = \pi_{*} \langle \widetilde{T}_{1, \min} \wedge \cdots \wedge \widetilde{T}_{p,\min} \rangle,
$$
where $\pi \colon \widetilde{X} \to X$ is a resolution from a compact K\"ahler manifold $\widetilde{X}$, and $\widetilde{T}_{i,\min} \in \pi^{*}\alpha_i$ are currents with minimal singularities.
\end{lem}

\begin{proof}
By definition,
\[
\langle T_{1, \min} \wedge \cdots \wedge T_{p,\min} \rangle = \pi_{*} \langle \pi^{*} T_{1, \min} \wedge \cdots \wedge \pi^{*} T_{p,\min} \rangle.
\]
By Proposition~\ref{prop-BEGZ10-1.12}, we have $\pi^{*} T_{i, \min} = \widetilde{T}_{i, \min}$ for each $i$, so the result follows.
\end{proof}

\subsubsection{Non-pluripolar product of pseudo-effective classes}

Let $X$ be a compact normal analytic variety in Fujiki's class, and let $\alpha_1, \ldots, \alpha_p \in H^{1,1}_{BC}(X)$ be psef classes.  
We define their non-pluripolar product by
$$
\langle \alpha_1 \cdots \alpha_p \rangle 
:= 
\pi_{*} \langle \pi^{*}\alpha_1 \cdots \pi^{*}\alpha_p \rangle,
$$
where $\pi \colon \widetilde{X} \to X$ is a resolution such that $\widetilde{X}$ is a compact K\"ahler manifold.

\begin{lem}
\label{lem-pluripolar-sing-property}
The non-pluripolar product $\langle \alpha_1 \cdots \alpha_p \rangle$ is independent of the choice of resolution $\pi$ and defines an element in $H^{2p}_{dR, \mathcal{D}'}(X, \mathbb{R})$.

Moreover, under the map $\Psi \colon H_{2n - 2p}(X, \mathbb{R}) \hookrightarrow H^{2p}_{dR, \mathcal{D}'}(X, \mathbb{R})$ introduced in \eqref{eq-current-homology}, there exists a class $\gamma \in H_{2n - 2p}(X, \mathbb{R})$ such that
$$
\Psi(\gamma) = \langle \alpha_1 \cdots \alpha_p \rangle.
$$
Thus, we may regard $\langle \alpha_1 \cdots \alpha_p \rangle$ as an element of $H_{2n - 2p}(X, \mathbb{R})$.
\end{lem}

\begin{proof}
By \cite[Remark 1.7]{BEGZ10}, the independence of the resolution follows as in Lemma~\ref{lem-independence-nonpluripolar}.  
Since $\pi$ is proper and the pushforward of $\pi$ preserves the closedness of currents,  $\langle \alpha_1 \cdots \alpha_p \rangle$ lies in $H^{2p}_{dR, \mathcal{D}'}(X, \mathbb{R})$.

Let us denote by
$$
\Phi \colon H^{2n-2p}_{dR}(X, \mathbb{R}) \twoheadrightarrow H^{2n-2p}(X, \mathbb{R})
$$
the natural surjection defined in \cite{Her67}.  
We may then regard $\Psi = \Phi^{\vee}$, using the identifications
\[
H_{2n - 2p}(X, \mathbb{R}) \cong H^{2n-2p}(X, \mathbb{R})^{\vee} \quad \text{and} \quad
H^{2p}_{dR, \mathcal{D}'}(X, \mathbb{R}) \cong 
H^{2n-2p}_{dR}(X, \mathbb{R})^{\vee}.
\]
Under this identification, it suffices to show that $\langle \alpha_1 \cdots \alpha_p \rangle \cdot \{ \eta \} = 0$ for any $d$-closed $(2n - 2p)$-form $\eta$ such that $\Phi(\{ \eta \}) = 0$.
Let $\eta$ be such a form. Consider the following commutative diagram:
$$
\xymatrix@C=25pt@R=20pt{
H^{2n-2p}_{dR}(X, \mathbb{R}) \ar[r]^{\Phi} \ar[d]^{\pi^{*}} &
H^{2n-2p}(X, \mathbb{R}) \ar[d]^{\pi^{*}} \\
H^{2n-2p}_{dR}(\widetilde{X}, \mathbb{R}) \ar[r] &
H^{2n-2p}(\widetilde{X}, \mathbb{R})
}
$$
Since $\widetilde{X}$ is smooth, the bottom horizontal map is an isomorphism.  
Therefore, we have $\pi^{*} \{ \eta \} = 0$ in $H^{2n-2p}_{dR}(\widetilde{X}, \mathbb{R})$.  
It follows that
\begin{align*}
\langle \alpha_1 \cdots \alpha_p \rangle \cdot \{ \eta \}
&= \langle \pi^{*}\alpha_1 \cdots \pi^{*}\alpha_p \rangle \cdot \pi^{*} \{ \eta \} = 0.
\end{align*}
\end{proof}

\begin{rem}
\label{rem-nonpluripolar-cohomology}
If all $\alpha_i \in H^{1,1}_{BC}(X)$ are nef, then for any $\sigma \in H^{2n - 2p}(X, \mathbb{R})$, it follows that
$$
\langle \alpha_1 \cdots \alpha_p \rangle \cdot \sigma =
\alpha_1 \cdots \alpha_p \cdot \sigma,
$$
where the latter intersection number is defined in Subsection~\ref{subsec-intersection-BC}.
In particular, $\langle \alpha_1 \cdots \alpha_p \rangle$ and $\alpha_1 \cdots \alpha_p$ can be identified as elements of $H^{2n - 2p}(X, \mathbb{R})^{\vee}$.  

If $X$ satisfies the Poincaré duality $H^{2n - 2p}(X, \mathbb{R})^{\vee} \cong H^{2p}(X, \mathbb{R})$ (for example, if $X$ has only quotient singularities by \cite[Proposition 5.10]{GK20}), then $\langle \alpha_1 \cdots \alpha_p \rangle$ can be regarded as an element of $H^{2p}(X, \mathbb{R})$. Furthermore,  if all $\alpha_i \in H^{1,1}_{BC}(X)$ are nef, then $\langle \alpha_1 \cdots \alpha_p \rangle$ and $\alpha_1 \cdots \alpha_p$ can be identified as elements of $H^{2p}(X, \mathbb{R})$.

\end{rem}

\begin{lem}
\label{lem-nonpluripolar-minimlal-singular}
Let $X$ be a compact normal analytic variety, and let $\alpha_1, \ldots, \alpha_p \in H^{1,1}_{BC}(X)$ be big classes.  
Let $T_{i, \min} \in \alpha_i$ be currents with minimal singularities. Then we have
\[
\{ \langle T_{1, \min} \wedge \cdots \wedge T_{p, \min} \rangle \} 
=
\langle \alpha_1 \cdots \alpha_p \rangle
\quad
\mathrm{in} \,\, H^{2p}_{dR, \mathcal{D}'}(X, \mathbb{R}).
\]
\end{lem}

\begin{proof}
By Lemmas~\ref{lem-nonpluripolar-minimal} and~\ref{lem-pluripolar-sing-property}, we may assume that $X$ is smooth.  
In the smooth case, the result follows from \cite[Definition 1.17]{BEGZ10}.
\end{proof}

\begin{rem}
\label{rem-nonopluripolar-fact}
As in the smooth case, the following statements hold. (The proof reduces to the smooth case by taking a resolution).
\begin{enumerate}[label=$(\arabic*)$]

    \item (cf.~ \cite[Theorem 1.16]{BEGZ10}) Let $T_i$ and $S_i$ be closed positive $(1,1)$-currents with small unbounded loci, and assume that $T_i$ is less singular than $S_i$ and cohomologous to $S_i$ for each $i = 1, \ldots, p$. Then their non-pluripolar products satisfy
$$
\{ \langle T_1  \wedge \cdots \wedge T_p \rangle \} \geq \{ \langle S_1  \wedge \cdots \wedge  S_p \rangle \}
%\quad \mathrm{in} \,\, H^{2p}_{dR, \mathcal{D}'}(X, \mathbb{R}),
$$
where the inequality $\geq$ means that the difference is represented by a closed positive $(p,p)$-current.

    \item (cf.~ \cite[Definition 1.17]{BEGZ10}) The non-pluripolar product is homogeneous and non-decreasing in each variable, and it is continuous on big classes.

    \item (cf.~ \cite[Proposition 1.22]{BEGZ10}) A psef class $\alpha$ is big if and only if $\langle \alpha^n \rangle > 0$.

    \item (cf.~ \cite[Theorem D]{WN19}) If $X$ is Moishezon, then for any big class $\alpha \in H^{1,1}_{BC}(X)$, the following orthogonality relation holds:
$$
\langle \alpha^{n-1} \rangle \cdot \alpha = \langle \alpha^n \rangle.
$$

    \item (cf.~ \cite[Lemma 3.1]{Xiao18}) For any psef class $\alpha_1, \ldots, \alpha_{n-1}, \beta \in H^{1,1}_{BC}(X)$, we have
    $$
    \langle \alpha_1 \cdots \alpha_{n-1} \rangle \cdot \beta 
    \ge 
    \langle \alpha_1 \cdots \alpha_{n-1} \cdot \beta \rangle.
    $$
    In particular, $\langle \alpha_1 \cdots \alpha_{n-1} \rangle \cdot \beta$ is non-negative, and thus $\langle \alpha_1 \cdots \alpha_{n-1} \rangle$ defines a movable class in the sense of \cite[Definition 4.2]{Ou25}.
    
\end{enumerate}
\end{rem}

\subsubsection{Comparison with the divisorial Zariski decomposition}
\label{subsubsec-DZD-nonpluripolar}

Let $X$ be a compact normal analytic variety in Fujiki's class.  
For any psef class $\alpha \in H^{1,1}_{BC}(X)$, we define the Lelong number of $\alpha$ along a prime divisor $E$ by
$$
\nu(\alpha, E) := \nu(\pi^{*}\alpha, \widetilde{E}),
$$
where $\pi \colon \widetilde{X} \to X$ is a resolution with $\widetilde{X}$ a compact K\"ahler manifold and $\widetilde{E}$ is the strict transform of $E$.  
This definition is independent of the choice of resolution.  
Moreover, if $\alpha$ is big, then by Proposition~\ref{prop-BEGZ10-1.12} and \cite[Proposition 3.6 (ii)]{Bou04}, the Lelong number is given by the current $T_{\min}$ with minimal singularities in $\alpha$:
$$
\nu(\alpha, E) = \nu(T_{\min}, E).
$$
Thus, by \cite[Proposition 3.2 (ii)]{Bou04}, a big class $\alpha$ is modified nef if and only if $\nu(\alpha, E) = 0$ for all prime divisors $E$.  
%Note that if $\alpha$ is modified nef, then so is $\pi^{*}\alpha$ for any resolution $\pi \colon \widetilde{X} \to X$.

\begin{defn}[{cf.~ \cite[Definition 3.7]{Bou04}, \cite[Appendix A]{DHY23}}]
\label{def-DZD}
Let $X$ be a compact normal $\mathbb{Q}$-factorial analytic variety in Fujiki's class and let $\alpha \in H^{1,1}_{BC}(X)$ be a psef class.  
We define the Weil divisor $N(\alpha)$ by
$$
N(\alpha) := \sum_{E\ \text{prime divisor}} \nu(\alpha, E)\, E.
$$
Since $X$ is $\mathbb{Q}$-factorial, we have $\{N(\alpha)\} \in H^{1,1}_{BC}(X)$. Thus we can define
$$
P(\alpha) := \alpha - \{N(\alpha)\} \in H^{1,1}_{BC}(X).
$$
The decomposition $\alpha = P(\alpha) + N(\alpha)$ is called the \textit{divisorial Zariski decomposition}, also known as the \textit{Boucksom-Zariski decomposition}.
We refer to $P(\alpha)$ as the \emph{positive part} and $N(\alpha)$ as the \emph{negative part} of $\alpha$.
\end{defn}

This definition agrees with that in \cite[Appendix A.6]{DHY23}, since for a resolution $\pi \colon \widetilde{X} \to X$, we have
$\pi_{*} N(\pi^{*} \alpha) = N(\alpha)$.  
In particular, $P(\alpha)$ is also psef.  
Moreover, if $X$ admits a big class, then by \cite[Lemma 2.6]{DH23}, $P(\alpha)$ is modified nef.

Unlike in the smooth case, $\langle \alpha \rangle$ and $P(\alpha)$ lie in different (co-)homology groups; specifically, $\langle \alpha \rangle \in H_{2n-2}(X, \mathbb{R})$ while $P(\alpha) \in H^{1,1}_{BC}(X)$.  
However, the following lemma shows that they agree when viewed as elements of the dual space $H^{2n-2}(X, \mathbb{R})^{\vee}$.

\begin{lem}
\label{lem-nonpluri-DZD}
Let $X$ be a compact normal $\mathbb{Q}$-factorial analytic variety in Fujiki's class and let $\alpha \in H^{1,1}_{BC}(X)$ be a psef class. 
Then, for any $\sigma \in H^{2n-2}(X, \mathbb{R})$, we have
$$
\delta^1(P(\alpha)) \cdot \sigma = \langle \alpha \rangle \cdot \sigma,
$$
where $\delta^1 \colon H^{1,1}_{BC}(X) \to H^2(X,\mathbb{R})$ is the connecting morphism as in \eqref{eq-connect-h2}.
\end{lem}

\begin{proof}
Take a resolution $\pi \colon \widetilde{X} \to X$. Then,
\begin{equation}
\label{eq-positivepart}
\delta^1(P(\alpha)) \cdot \sigma 
\underalign{\text{(\ref{eq-pullback-intersection})}}{=} \pi^{*} P(\alpha) \cdot \pi^{*} \sigma
\underalign{\text{(Def. \ref{def-DZD})}}{=}  
(\pi^{*} \alpha - \pi^{*} \{N(\alpha)\}) \cdot \pi^{*} \sigma.
\end{equation}
On the other hand, by Lemma \ref{lem-pluripolar-sing-property},
\begin{equation}
\label{eq-positivepart-2}
\langle \alpha \rangle \cdot \sigma
\underalign{\text{(Lem. \ref{lem-pluripolar-sing-property})}}{=} 
\pi_{*} \langle \pi^{*} \alpha \rangle \cdot \sigma
= \langle \pi^{*} \alpha \rangle \cdot \pi^{*} \sigma.
\end{equation}
Since $\widetilde{X}$ is smooth, we have $\langle \pi^{*} \alpha \rangle = \pi^{*} \alpha - \{N(\pi^{*} \alpha)\}$.  
Comparing \eqref{eq-positivepart} and \eqref{eq-positivepart-2}, it suffices to show
$$
(\pi^{*} \{N(\alpha)\} - \{N(\pi^{*} \alpha)\}) \cdot \pi^{*} \sigma = 0.
$$
As the support of $\pi^{*} N(\alpha) - N(\pi^{*} \alpha)$ is contained in the $\pi$-exceptional locus, the claim follows from \cite[Lemma 2.2]{Ou25b}.
\end{proof}

\begin{lem}[{cf.~ \cite[Proposition 3.8]{Bou04}}]
\label{lem-bigness-positivepart}
Let $X$ be a compact normal $\mathbb{Q}$-factorial analytic variety in Fujiki's class, and let $\alpha \in H^{1,1}_{BC}(X)$ be a psef class.  
Then $\alpha$ is big if and only if $P(\alpha)$ is big.
\end{lem}

\begin{proof}
The proof is similar to \cite[Proposition 3.8]{Bou04}.  
The "if" part is clear. We prove the "only if" direction. 
Assume that $\alpha$ is big.
Take a K\"ahler current $T \in \alpha$, and define a positive current
$$
S := \sum_{E\ \text{prime divisor}} \nu(T, E) [E],
$$
where $[E]$ denotes the integration current associated to a prime divisor $E$.  
Then the restriction $(T - S)|_{X_{\reg}}$ is a closed positive $(1,1)$-current in the class $(\alpha - \{S\})|_{X_{\reg}}$.  
By \cite[Proposition 4.6.3]{BG13}, the current $T - S$ extends as a closed positive $(1,1)$-current with local potentials in $\alpha - \{S\}$.  
Moreover, since $T$ is a K\"ahler current, so is $T - S$. 
Let $T_{\min} \in \alpha$ be a current with minimal singularities, and consider
$$
(T - S) + \sum_{E\ \text{prime divisor}} (\nu(T, E) - \nu(T_{\min}, E)) [E].
$$
This is a K\"ahler current in the class $P(\alpha)$, hence $P(\alpha)$ is big.
\end{proof}

\begin{prop}[{cf.~ \cite[Proposition 3.2.10]{Bou02}}]
\label{prop-nonpluripolar-modifiednef}
Let $X$ be a compact normal $\mathbb{Q}$-factorial analytic variety in Fujiki's class.  
For any big classes $\alpha_1, \ldots, \alpha_p \in H^{1,1}_{BC}(X)$, we have
$$
\langle \alpha_1 \cdots \alpha_p \rangle = \langle P(\alpha_1) \cdots P(\alpha_p) \rangle.
$$
\end{prop}

\begin{proof}
The inequality $\langle \alpha_1 \cdots \alpha_p \rangle \ge \langle P(\alpha_1) \cdots P(\alpha_p) \rangle$ follows from Remark~\ref{rem-nonopluripolar-fact} (2).  
We prove the reverse inequality.
Let $T_{i, \min} \in \alpha_i$ be a current with minimal singularities. Then
$$
S_i:=T_{i, \min} - \sum_{E\ \text{prime divisor}} \nu(T_{i, \min}, E)\, [E]
$$
is a closed positive $(1,1)$-current with local potentials in the class $P(\alpha_i)$ by Definition \ref{def-DZD}.  
Thus, since both $T_{i, \min}$ and $S_i$ have small unbounded loci, we conclude
\begin{align*}
\langle \alpha_1 \cdots \alpha_p \rangle 
&\underalign{\text{(Lem. \ref{lem-nonpluripolar-minimlal-singular})}}{=}
\{\langle T_{1,\min} \wedge \cdots \wedge T_{p,\min} \rangle \} \\
&\underalign{\text{(Lem. \ref{lem-nonpluripolar-unbounded})}}{=}
\{\langle S_1 \wedge \cdots \wedge S_p \rangle \} \\
&\underalign{\text{(Rem. \ref{rem-nonopluripolar-fact} (1))}}{\le}
\langle P(\alpha_1) \cdots P(\alpha_p) \rangle.
\end{align*}
For the last inequality, we use the fact that each $P(\alpha_i)$ is big by Lemma~\ref{lem-bigness-positivepart}, together with Lemma~\ref{lem-nonpluripolar-minimlal-singular}.
\end{proof}

\subsubsection{Hodge index theorem for non-pluripolar product}

\begin{lem}
\label{lem-Hodge-index}
Let $X$ be a compact normal analytic variety in Fujiki's class, and let $\beta, \alpha_1, \ldots, \alpha_{n-2} \in H^{1,1}_{\mathrm{BC}}(X)$ be psef classes.
If 
\(
\langle \beta^2 \cdot \alpha_1 \cdots \alpha_{n-2} \rangle > 0,
\)
then for any $\gamma \in H^{1,1}_{\mathrm{BC}}(X)$, the following inequality holds:
\[
(\gamma^2 \cdot \langle \alpha_1 \cdots \alpha_{n-2} \rangle) \cdot \langle \beta^2 \cdot \alpha_1 \cdots \alpha_{n-2} \rangle
\leq
(\gamma \cdot \langle \beta \cdot \alpha_1 \cdots \alpha_{n-2} \rangle)^2.
\]
\end{lem}

\begin{proof}
For simplicity, assume that $\alpha_1 = \cdots = \alpha_{n-2}$ and set $\alpha := \alpha_1$.
By taking a resolution, we may reduce to the case where $X$ is a compact K\"ahler manifold.
%We follow the strategy used in the proof of \cite[Corollary 5.8]{LX16}.
As in \cite[Corollary 5.8]{LX16}, we can take a sequence of modifications $\mu_m \colon X_m \to X$ such that 
$$\mu_m^* \alpha = \alpha_{(m)} + [E_m]
\quad \text{and} \quad \mu_m^* \beta = \beta_{(m)} + [F_m],$$ 
where $\alpha_{(m)}, \beta_{(m)}$ are smooth K\"ahler forms on $X_m$, and $E_m, F_m$ are effective divisors satisfying
\[
\langle \beta^k \cdot \alpha^{n-2} \rangle 
= 
\lim_{m \to \infty} (\mu_m)_* \left( \beta_{(m)}^k \cdot \alpha_{(m)}^{n-2} \right)
\quad \text{for } k = 0,1,2.
\]
Since both $\alpha_{(m)}$ and $\beta_{(m)}$ are K\"ahler, we can apply the Hodge index theorem to conclude that
\[
(\mu_m^*\gamma^2 \cdot \alpha_{(m)}^{n-2}) \cdot (\beta_{(m)}^2 \cdot \alpha_{(m)}^{n-2}) 
\le 
(\mu_m^*\gamma \cdot \beta_{(m)} \cdot \alpha_{(m)}^{n-2})^2.
\]
Taking the limit as $m \to \infty$, we obtain the desired inequality.
\end{proof}

\subsection{Slope stability with respect to non-pluripolar products}

In this subsection, we define slope stability with respect to the non-pluripolar product $\langle \alpha^{n-1} \rangle$, where $\alpha \in H^{1,1}_{BC}(X)$ is a big class on a compact normal analytic variety $X$, and examine the notion of stability introduced in \cite[Definition 4.6]{Jin25}.
Unlike the product $\alpha^{n-1}$, the non-pluripolar product $\langle \alpha^{n-1} \rangle$ is an element of $H_2(X,\R)$, whereas $c_1(\mathcal{E})$ lies in $H_{2n-2}(X, \mathbb{R})$. Hence, their intersection number cannot be defined directly on $X$. To resolve this, we consider a resolution $\pi: \widetilde{X} \to X$.

\subsubsection{Slope stability and its properties}
\label{subsec-slope}

As introduced in \cite[Definition 4.1]{GKP16}, for any coherent sheaf $\mathcal{E}$ on a compact normal analytic variety $X$, a resolution $\pi: \widetilde{X} \to X$ is called a \textit{strong resolution} of $\mathcal{E}$ if the pullback $\pi^* \mathcal{E} / \tor$ is locally free.

\begin{deflem}
\label{defn-slope}
Let $X$ be a compact normal analytic variety in Fujiki's class.
Let $\mathcal{E}$ be a torsion-free sheaf on $X$, and let $\alpha_1, \ldots, \alpha_{n-1} \in H^{1,1}_{BC}(X)$ be psef classes.
Then, for any strong resolution $\pi: \widetilde{X} \to X$, the intersection number
\[
c_1(\pi^* \mathcal{E} / \tor) \cdot \langle \pi^* \alpha_1 \cdots \pi^* \alpha_{n-1} \rangle
\]
is independent of the choice of the strong resolution.
Hence, we define
\[
c_1(\mathcal{E}) \cdot \langle \alpha_1 \cdots \alpha_{n-1} \rangle
:= 
c_1(\pi^* \mathcal{E} / \tor) \cdot \langle \pi^* \alpha_1 \cdots \pi^* \alpha_{n-1} \rangle,
\]
where $\pi: \widetilde{X} \to X$ is any $($some$)$ strong resolution.
\end{deflem}

If all $\alpha_1, \ldots, \alpha_{n-1}$ are nef, then the intersection number $c_1(\mathcal{E}) \cdot \langle \alpha_1 \cdots \alpha_{n-1} \rangle$ agrees with the usual intersection number $c_1(\mathcal{E}) \cdot \alpha_1 \cdots \alpha_{n-1}$ as defined in \cite[Definition~4.3]{GKP16} and \cite[Section~3]{Wu21}.

\begin{proof}
Fix a strong resolution $\pi: \widetilde{X} \to X$. Then it suffices to show that for any modification $\rho: W \to \widetilde{X}$, we have
\[
c_1(\pi^* \mathcal{E} / \tor) \cdot \langle \pi^* \alpha_1 \cdots \pi^* \alpha_{n-1} \rangle
=
c_1(\rho^* \pi^* \mathcal{E} / \tor) \cdot \langle \rho^* \pi^* \alpha_1 \cdots \rho^* \pi^* \alpha_{n-1} \rangle.
\]
This follows from the projection formula for non-pluripolar products \cite[Remark 1.7]{BEGZ10}:
\[
\rho_* \langle \rho^* \pi^* \alpha_1 \cdots \rho^* \pi^* \alpha_{n-1} \rangle = \langle \pi^* \alpha_1 \cdots \pi^* \alpha_{n-1} \rangle
\]
and from the isomorphism
\((\rho^* \pi^* \mathcal{E}) / \tor \cong \rho^*(\pi^* \mathcal{E} / \tor)\)
(see \cite[Definition 5 below]{Wu22}).
\end{proof}

\begin{rem}
\label{rem-mixed-slope}
For a $\mathbb{Q}$-line bundle $\mathcal{F}$ on a compact normal analytic variety $X$, the pullback $\pi^* c_1(\mathcal{F})$ does not necessarily coincide with $c_1(\pi^{[*]} \mathcal{F})$, where $\pi: \widetilde{X} \to X$ is a resolution and we define
\(
\pi^* c_1(\mathcal{F}) := \frac{1}{m} \pi^* c_1(\mathcal{F}^{[m]}),
\)
with $\mathcal{F}^{[m]}$ being locally free for some integer $m \in \mathbb{N}$.
%(See \cite[Theorem 1.2]{Graf12} for details.)
As an example, let $X$ be the cone over a smooth quadric curve (i.e., $X = \mathbb{P}(1,1,2)$), and let $\pi: \widetilde{X} \to X$ be the blow-up at the vertex. In this case, $\widetilde{X}$ is the Hirzebruch surface $\mathbb{F}_2$, and $\pi$ is the contraction of a $(-2)$-curve $C$.
Let $\mathcal{A}$ be the rank 1 reflexive  sheaf on $X$ corresponding to the ruling of $X$. Then $\mathcal{A}$ is not locally free, but its reflexive square $\mathcal{A}^{[2]}$ is. In this setting, we have
\[
\pi^* c_1(\mathcal{A}) = c_1(\pi^{[*]} \mathcal{A}) + \frac{1}{2} c_1(C),
\]
so in particular, $\pi^* c_1(\mathcal{A}) \ne c_1(\pi^{[*]} \mathcal{A})$.

Let $\pi : \widetilde{X} \to X$ be the resolution mentioned above.
Since the difference $\pi^* c_1(\mathcal{F}) - c_1(\pi^{[*]} \mathcal{F})$ is $\pi$-exceptional, we would like to expect that
\[
(\pi^* c_1(\mathcal{F}) - c_1(\pi^{[*]} \mathcal{F})) \cdot \langle \pi^* \alpha_1 \cdots \pi^* \alpha_{n-1} \rangle = 0.
\]
However, this is a rather delicate issue, because it is not known whether
\[
\langle \pi^* \alpha_1 \cdots \pi^* \alpha_{n-1} \rangle \cdot \{E\} = 0
\]
holds in general for any $\pi$-exceptional effective divisor $E$.
(Here, $\{E\}$ represents the class in $H^{1,1}(\widetilde{X}, \mathbb{R})$ induced by the positive current associated with $E$.) In fact, \cite[Example 3.6]{LX16} shows that
\(
\langle \pi^* \alpha_1 \cdots \pi^* \alpha_{n-1} \rangle \ne \pi^* \langle \alpha_1 \cdots \alpha_{n-1} \rangle
\)
in general. Therefore, to proceed with arguments involving these classes, we need to impose additional assumptions on $\alpha_1, \ldots, \alpha_{n-1}$. (See Definition \ref{defn-vanishing-exceptional}.)
\end{rem}

\begin{defn}[Slope and Stability]
\label{defn-slope-nonpluri}
Let $X$ be a compact normal analytic variety in Fujiki's class and let $\alpha_1, \ldots, \alpha_{n-1} \in H^{1,1}_{BC}(X)$ be psef classes.
For any rank $r$ torsion-free sheaf $\mathcal{E}$, we define the \emph{slope} by
\[
\mu_{\langle \alpha_1 \cdots \alpha_{n-1} \rangle}(\mathcal{E})
:=
\frac{c_1(\mathcal{E}) \cdot \langle \alpha_1 \cdots \alpha_{n-1} \rangle}{r}.
\]
We say that $\mathcal{E}$ is \emph{$\langle \alpha_1 \cdots \alpha_{n-1} \rangle$-stable} (resp.~ \emph{$\langle \alpha_1 \cdots \alpha_{n-1} \rangle$-semistable}) if, for every non-zero torsion-free subsheaf $0 \subsetneq \mathcal{F} \subset \mathcal{E}$ with $\rk \mathcal{F} \neq \rk \mathcal{E}$, the following inequality holds:
\[
\mu_{\langle \alpha_1 \cdots \alpha_{n-1} \rangle}(\mathcal{F}) < \mu_{\langle \alpha_1 \cdots \alpha_{n-1} \rangle}(\mathcal{E})
\quad
\text{(resp.~ } \mu_{\langle \alpha_1 \cdots \alpha_{n-1} \rangle}(\mathcal{F}) \le \mu_{\langle \alpha_1 \cdots \alpha_{n-1} \rangle}(\mathcal{E}) \text{)}.
\]

\end{defn}

We now investigate the relation between Definition \ref{defn-slope-nonpluri} and the notion of stability introduced in \cite[Definition 4.6]{Jin25}. To this end, we need the following assumption:

\begin{defn}[{cf.~ \cite[Assumption 3.1]{Jin25}}]
\label{defn-vanishing-exceptional}
Let $X$ be a compact normal analytic variety in Fujiki's class, and let $\alpha$ be a big class on $X$.
We say that $\alpha$ satisfies the \emph{vanishing property}  if, for any bimeromorphic morphism $\pi: \widetilde{X} \to X$ from a compact complex manifold $\widetilde{X}$ and any $\pi$-exceptional divisor $E$ on $\widetilde{X}$, the following holds:
\[
\langle (\pi^*\alpha)^{n-1} \rangle \cdot \{E\} = 0.
\]
\end{defn}

If a big class $\alpha \in H^{1,1}_{BC}(X)$ satisfies the vanishing property, then its pullback $f^*\alpha$ also satisfies the vanishing property for any bimeromorphic morphism $f \colon Y \to X$ from a compact normal analytic variety.
It is conjectured that every big class satisfies the vanishing property. Currently, however, this has only been verified in the following situations:

\begin{lem}
\label{lem-VP}
Let $X$ be a compact normal analytic variety in Fujiki's class. Then any big class $\alpha$ on $X$ satisfies the vanishing property in each of the following cases:
\begin{enumerate}[label=$(\arabic*)$]
\item $X$ is Moishezon;
\item $\dim X = 2$;
\item $\alpha$ is nef.
\end{enumerate}
\end{lem}

These are mentioned in \cite[Section 3]{Jin25}, but we include a proof here for completeness.

\begin{proof}
(1) Let $\pi: \widetilde{X} \to X$ be a bimeromorphic morphism from a compact complex manifold $\widetilde{X}$. We may assume that $\widetilde{X}$ is projective. Let $E$ be a $\pi$-exceptional divisor, assumed irreducible and effective. By \cite[Lemma~A.5]{DHY23}, we have
$\langle \pi^* \alpha + \{E\} \rangle = \langle \pi^* \alpha \rangle.$
Hence, by Proposition~\ref{prop-nonpluripolar-modifiednef}, we obtain
\[
\langle (\pi^* \alpha + t\{E\})^n \rangle = \langle (\pi^* \alpha)^n \rangle
\]
for any $t > 0$.
Since $\widetilde{X}$ is projective, the volume function is differentiable by \cite[Theorem C]{WN19}, so
$$
\left.\frac{d}{dt}\right|_{t=0} \langle (\pi^* \alpha + t\{E\})^n \rangle 
= n \langle (\pi^* \alpha)^{n-1} \rangle \cdot \{E\}.
$$
The left-hand side is zero because the volume is constant in $t$. Therefore, we conclude that $\langle (\pi^* \alpha)^{n-1} \rangle \cdot \{E\} = 0$.

(2) When $\dim X = 2$, the differentiability of the volume also holds by \cite[Theorem 1.5]{Deng17}, so the same argument as in case (1) applies.

(3) If $\alpha$ is nef, then $\langle (\pi^* \alpha)^{n-1} \rangle = (\pi^* \alpha)^{n-1}$, so the claim is immediate.
\end{proof}

\begin{rem}
\label{assumption-remark}
A crucial step in proving that every big class satisfies the vanishing property is to establish the differentiability of the volume function on the big cone, particularly in the setting of compact K\"ahler manifolds. In the projective case, this was proved by Witt-Nyström in \cite[Theorem C]{WN19}. However, in the K\"ahler case, the differentiability is still open (see also \cite{Vu23}).
\end{rem}

\begin{ex}
Let $X = \mathbb{C}^2$ and consider the psh function $\varphi = \log(|z|^2 + |w|^2)$.  
Let $\pi: \widetilde{X} = \mathrm{Bl}_0(\mathbb{C}^2) \to \mathbb{C}^2$ be the blow-up at the origin.  
Then we have
$$
dd^c(\pi^*\varphi) = [C] + \omega_{-C},
$$
where $C$ is the exceptional divisor and $\omega_{-C} \in c_1(-C)$ is a smooth semi-positive form such that $\omega_{-C}|_C = \omega_{FS}$, the Fubini–Study metric on $C \simeq \mathbb{C}\mathbb{P}^1$.  
Hence,
$$
\langle dd^c(\pi^*\varphi) \rangle = \omega_{-C}, \quad \text{and} \quad \langle dd^c(\pi^*\varphi) \rangle \cdot [C] = (-C) \cdot C = 1 > 0.
$$

This example shows that even when a closed positive $(1,1)$-current $T$ has analytic singularities, it may still happen that
$
 \langle (\pi^*T)^{n-1} \rangle  \cdot [D] \ne 0,
$
where $\pi: \widetilde{X} \to X$ is a bimeromorphic morphism.
\end{ex}

\begin{lem}
\label{lem-slope-invariance-VP}
Let $X$ be a compact normal analytic variety in Fujiki's class, and let $\alpha \in H^{1,1}_{\mathrm{BC}}(X)$ be a big class satisfying the vanishing property.  
Then for any torsion-free sheaf $\mathcal{E}$ on $X$ and any bimeromorphic morphism $f: Y \to X$, we have
$$
c_1(f^{[*]} \mathcal{E}) \cdot \langle (f^* \alpha)^{n-1} \rangle
=
c_1(\mathcal{E}) \cdot \langle \alpha^{n-1} \rangle.
$$
\end{lem}
\begin{proof}
Let $\pi_X: \widetilde{X} \to X$ (resp.~ $\pi_Y: \widetilde{Y} \to Y$) be strong resolutions such that  
$\pi_X^* \mathcal{E}/\tor$ (resp.~ $(\pi_Y^* f^{[*]} \mathcal{E})/\tor$) is locally free.
Let $\widetilde{f}: \widetilde{Y} \to \widetilde{X}$ be the induced morphism.  
Then   $(\pi_Y^* f^{[*]} \mathcal{E})/\tor$ and $\widetilde{f}^* (\pi_X^* \mathcal{E}/\tor)$ agree outside a $\pi_X \circ \widetilde{f}$-exceptional set.  
Hence, by the vanishing property,  we obtain:
$$
 c_1\big( (\pi_Y^* f^{[*]} \mathcal{E})/\tor \big) \cdot \langle ( \pi_{Y}^* f^* \alpha)^{n-1} \rangle 
 =
 c_1\big(\widetilde{f}^* (\pi_X^* \mathcal{E}/\tor) \big) \cdot \langle (\widetilde{f}^* \pi_X^* \alpha)^{n-1} \rangle 
$$
Lemma \ref{defn-slope} implies that the left-hand side equals  \( c_1(f^{[*]} \mathcal{E}) \cdot \langle (f^* \alpha)^{n-1}\rangle \) and the right-hand side equals \( c_1(\mathcal{E}) \cdot \langle \alpha^{n-1} \rangle \), which completes the proof.
%\begin{align*}c_1(\mathcal{E}) \cdot \langle \alpha^{n-1} \rangle&= c_1\big( \widetilde{f}^* \pi_X^* \mathcal{E}/\tor \big) \cdot \langle (\widetilde{f}^* \pi_X^* \alpha)^{n-1} \rangle \\&= c_1\big( \pi_Y^* f^{[*]} \mathcal{E}/\tor \big) \cdot \langle (\widetilde{f}^* \pi_X^* \alpha)^{n-1} \rangle \\&= c_1(f^{[*]} \mathcal{E}) \cdot \langle (f^* \alpha)^{n-1} \rangle.\end{align*}
\end{proof}
Let us add one remark. Under the setting of the above Lemma, when $c_1(\mathcal{E}) \in H^2(X, \mathbb{R})$ (for instance, when $X$ has an orbifold structure or when $\det(\mathcal{E})$ is a $\mathbb{Q}$-line bundle), one can define the intersection number  with $c_1(\mathcal{E}) \in H^2(X, \mathbb{R})$ and $\langle \alpha^{n-1} \rangle \in H_2(X, \mathbb{R})$. By the above Lemma, we know that this intersection number coincides with the one defined in Definition-Lemma \ref{defn-slope} (see also Remark~\ref{rem-mixed-slope}).

The following proposition shows that the stability condition in Definition~\ref{defn-slope-nonpluri} is equivalent to the one introduced by the second author in \cite[Definition 4.6]{Jin25}.

\begin{prop}
\label{prop-comparison-stability}
Let $X$ be a compact normal analytic variety in Fujiki's class and let $\alpha \in H^{1,1}_{BC}(X)$ be a big class satisfying the vanishing property.  
Then, for any torsion-free sheaf $\mathcal{E}$ on $X$, the following conditions are equivalent:
\begin{enumerate}[label=$(\arabic*)$]
    \item $\mathcal{E}$ is $\langle \alpha^{n-1} \rangle$-stable $($resp.~ semistable$)$.
    \item For some resolution $\pi : \widetilde{X} \to X$, the reflexive pullback $\pi^{[*]}\mathcal{E}$ is $\langle (\pi^{*}\alpha)^{n-1} \rangle$-stable $($resp.~ semistable$)$.
    \item For every resolution $\pi : \widetilde{X} \to X$, the reflexive pullback $\pi^{[*]}\mathcal{E}$ is $\langle (\pi^{*}\alpha)^{n-1} \rangle$-stable $($resp.~ semistable$)$.
\end{enumerate}
\end{prop}

The equivalence between (2) and (3) was already established in \cite[Lemma 4.7]{Jin25}. (Although \cite[Lemma 4.7]{Jin25} addresses only the stable case, the same argument also applies to the semistable case.)
This proposition shows that, under the vanishing property, one may freely apply the results of \cite{Jin25}, especially in situations where $X$ is Moishezon.

\begin{proof}
We prove the equivalence in the stable case; the semistable case is analogous.

\noindent
$(1) \Rightarrow (2)$:  
Let $\widetilde{\mathcal{S}}$ be a torsion-free subsheaf of $\pi^{[*]}\mathcal{E}$.  
Note that $(\pi_*\pi^*\mathcal{E})^{\vee\vee} = \mathcal{E}^{\vee\vee}$, since they coincide in codimension one.  
Moreover, we have
\[
(\pi_*\pi^{[*]}\mathcal{E})^{\vee\vee} = (\pi_*\pi^*\mathcal{E})^{\vee\vee} = \mathcal{E}^{\vee\vee}.
\]
Since both $(\pi_*\widetilde{\mathcal{S}})^{\vee\vee}$ and $\mathcal{E}$ are subsheaves of $\mathcal{E}^{\vee\vee}$, we may define the subsheaf of $\mathcal{E}$ by
\[
\mathcal{S} := (\pi_*\widetilde{\mathcal{S}})^{\vee\vee} \cap \mathcal{E} \subset \mathcal{E}.
\]
Observe that $\pi^{[*]}\mathcal{S}$ agrees with $\widetilde{\mathcal{S}}$ outside a $\pi$-exceptional locus.
Now, since $\alpha$ satisfies the vanishing property, we obtain
\[
\mu_{\langle \alpha^{n-1} \rangle}(\mathcal{S}) 
\underalign{\text{(Lem. \ref{lem-slope-invariance-VP})}}{=}
\mu_{\langle (\pi^*\alpha)^{n-1} \rangle}(\pi^{[*]}\mathcal{S}) 
\underalign{\text{(Def. \ref{defn-vanishing-exceptional})}}{=}
\mu_{\langle (\pi^*\alpha)^{n-1} \rangle}(\widetilde{\mathcal{S}}).
\]
%(If necessary, we may take a further bimeromorphic morphism $\rho: W \to \widetilde{X}$ so that $(\rho \circ \pi)^*\mathcal{S}/\tor$ becomes locally free.)
Since $\mathcal{E}$ is $\langle \alpha^{n-1} \rangle$-stable, it follows that $\pi^{[*]}\mathcal{E}$ is $\langle (\pi^*\alpha)^{n-1} \rangle$-stable.

\noindent
$(2) \Rightarrow (1)$:  
Let $\mathcal{S} \subset \mathcal{E}$ be a torsion-free subsheaf.  
Then we define $\widetilde{\mathcal{S}}$ as the image of $\pi^*\mathcal{S}$ under the natural morphism $\pi^*\mathcal{E} \to \pi^{[*]}\mathcal{E}$.  
Note that $\widetilde{\mathcal{S}}$ agrees with $\pi^{[*]}\mathcal{S}$ away from the $\pi$-exceptional locus.  
By the same argument as above, and using the vanishing property of $\alpha$, we obtain
\(
\mu_{\langle \alpha^{n-1} \rangle}(\mathcal{S}) = 
\mu_{\langle (\pi^*\alpha)^{n-1} \rangle}(\widetilde{\mathcal{S}}),
\)
which implies that $\mathcal{E}$ is $\langle \alpha^{n-1} \rangle$-stable.
\end{proof}
We conclude by stating a lemma that will be used in Section~\ref{sec-MY-inequality}. 
% and~\ref{sec-Miyaoka's-inequality}.

\begin{lem}
\label{lem-slope-same}
Let $X$ be a compact normal analytic variety in Fujiki's class, and let $\alpha \in H^{1,1}_{BC}(X)$ be a big class satisfying the vanishing property.  
Then for any torsion-free sheaves $\mathcal{F} \subset \mathcal{G}$ of the same rank, we have
\[
c_1(\mathcal{F}) \cdot \langle \alpha^{n-1} \rangle \le c_1(\mathcal{G}) \cdot \langle \alpha^{n-1} \rangle.
\]
\end{lem}

\begin{proof}
Let $\pi \colon \widetilde{X} \to X$ be a resolution such that $\widetilde{X}$ is a compact K\"ahler manifold.  
Define $\widetilde{\mathcal{F}}$ as the image of $\pi^*\mathcal{F}$ under the natural morphism $\pi^*\mathcal{F} \to \pi^{[*]}\mathcal{G}$.  
Since $\alpha$ satisfies the vanishing property, the argument in Proposition~\ref{prop-comparison-stability} shows that
\begin{equation}
\label{eq-slope-samerank}
\mu_{\langle \alpha^{n-1} \rangle}(\mathcal{F}) = 
\mu_{\langle (\pi^*\alpha)^{n-1} \rangle}(\widetilde{\mathcal{F}}).
\end{equation}
On the other hand, since $\widetilde{\mathcal{F}} \subset \pi^{[*]}\mathcal{G}$ and both sheaves have the same rank, the quotient $\pi^{[*]}\mathcal{G}/\widetilde{\mathcal{F}}$ is a torsion sheaf.  
By \cite[Proposition~5.6.14]{Kob14}, the first Chern class of this quotient is represented by an effective divisor.  
Hence, by Remark~\ref{rem-nonopluripolar-fact}~(5), we have
\[
\left( c_1(\pi^{[*]}\mathcal{G}) - c_1(\widetilde{\mathcal{F}}) \right) \cdot \langle (\pi^*\alpha)^{n-1} \rangle
= c_1(\pi^{[*]}\mathcal{G} / \widetilde{\mathcal{F}}) \cdot \langle (\pi^*\alpha)^{n-1} \rangle\underset{\text{(Rem. \ref{rem-nonopluripolar-fact} (5))}}{\ge}0.
\]
Moreover, by Lemma~\ref{lem-slope-invariance-VP}, we have
\(
c_1(\pi^{[*]}\mathcal{G}) \cdot \langle (\pi^*\alpha)^{n-1} \rangle = c_1(\mathcal{G}) \cdot \langle \alpha^{n-1} \rangle.
\)
Combining this with \eqref{eq-slope-samerank}, we obtain the desired inequality.
\end{proof}

\subsubsection{Stability of Higgs sheaves}
\label{subsub-slope-Higgs}

We review the notion of reflexive Higgs sheaves on normal analytic varieties and define the slope stability with respect to a big class.

\begin{defn}(\cite[Definition 4.8, 5.1]{GKPT19b}, \cite[Subsection 2.2.1, 2.2.2]{ZZZ25})
Let $X$ be a normal analytic variety. 
\begin{enumerate}[label=$(\arabic*)$]
\item A \emph{Higgs sheaf} is a pair $(\mathcal{E}, \theta)$ consisting of a coherent $\mathcal{O}_X$-module $\mathcal{E}$ together with an $\mathcal{O}_X$-linear morphism $\theta:\mathcal{E} \to \mathcal{E} \otimes \Omega_X^{[1]}$, called the \emph{Higgs field}, such that the following composition vanishes:
\[
\mathcal{E} \overset{\theta}{\longrightarrow} \mathcal{E} \otimes \Omega_X^{[1]} \overset{\theta \otimes \id}{\longrightarrow} \mathcal{E} \otimes \Omega_X^{[1]} \otimes \Omega_X^{[1]} \overset{\id \otimes [\wedge]}{\longrightarrow} \mathcal{E} \otimes \Omega_X^{[2]}
\]
%vanishes. This condition is commonly denoted by $\theta \wedge \theta = 0$.
\item Let $(\mathcal{E}, \theta)$ be a Higgs sheaf. A \emph{generically $\theta$-invariant subsheaf} is a subsheaf $\mathcal{G} \subset \mathcal{E}$ such that there exists a Zariski open subset $U \subset X$ where $\Omega_X^{[1]}$ is locally free and $\mathcal{G}|_U$ is $\theta|_{U}$-invariant. (More precisely, the restriction $\theta(\mathcal{G})|_U$ is contained in the image of the natural map $\mathcal{G} \otimes \Omega_X^{[1]} |_U\to \mathcal{E} \otimes \Omega_X^{[1]}|_U$ on $U$.)
\end{enumerate}

\end{defn}

We now define stability for Higgs sheaves analogously to Subsection~\ref{subsec-slope}.

\begin{defn}[{cf.~ \cite[Definition 4.14]{GKPT19b}}]
Let $X$ be a compact normal analytic variety in Fujiki's class and let $\alpha$ be a big class on $X$. A torsion-free Higgs sheaf $(\mathcal{E}, \theta)$ is said to be \emph{$\langle \alpha^{n-1} \rangle$-stable} (resp.~ \emph{$\langle \alpha^{n-1} \rangle$-semistable}) if, for any nontrivial generically $\theta$-invariant subsheaf $\mathcal{F} \subset \mathcal{E}$ with $\rk \mathcal{F} \neq \rk \mathcal{E}$, we have
\[
\mu_{\langle \alpha^{n-1} \rangle}(\mathcal{F}) < \mu_{\langle \alpha^{n-1} \rangle}(\mathcal{E})
\quad
\text{(resp.~ } \mu_{\langle \alpha^{n-1} \rangle}(\mathcal{F}) \le \mu_{\langle\alpha^{n-1}\rangle}(\mathcal{E}) \text{)}.
\]
%We define \emph{$\langle \alpha^{n-1} \rangle$-semistability} in the same way by replacing the strict inequality with a weak inequality.
\end{defn}

Next, we present the Higgs sheaf analogue of Proposition~\ref{prop-comparison-stability}.  
To do so, we first recall the notion of the pullback of a Higgs sheaf (cf.~\cite[Subsection 2.2.4]{ZZZ25}).  
Let $f \colon Y \to X$ be a morphism between normal analytic varieties. If both $X$ and $Y$ have at most rational singularities, then by \cite[Theorem 1.11]{KS21}, there exists a functorial pullback for reflexive differentials:
\[
f^{[*]} := d_{\mathrm{refl}}f \colon f^{*} \Omega_X^{[1]} \to \Omega_Y^{[1]}.
\]
Hence, for any Higgs sheaf $(\mathcal{E}, \theta)$ on $X$, the reflexive pullback 
$f^{[*]}(\mathcal{E}, \theta) := (f^{[*]} \mathcal{E}, f^{[*]} \theta)$
is well-defined as a reflexive Higgs sheaf on $Y$, as described in \cite[Subsection~5.3]{GKPT19b} and \cite[Subsection~2.2.4]{ZZZ25}.

As in Proposition~\ref{prop-comparison-stability}, we can now state the equivalence between the stability notion defined here and that in~\cite[Definition 4.6]{Jin25}.

\begin{prop}
\label{prop-comparison-stability-Higgs}
Let $X$ be a compact normal analytic variety in Fujiki's class with rational singularities and let $\alpha \in H^{1,1}_{\mathrm{BC}}(X)$ be a big class satisfying the vanishing property.
Then for any torsion-free Higgs sheaf $(\mathcal{E}, \theta)$ on $X$, the following conditions are equivalent:
\begin{enumerate}[label=$(\arabic*)$]
    \item $(\mathcal{E}, \theta)$ is $\langle \alpha^{n-1} \rangle$-stable $($resp.~ semistable$)$.
    \item For some resolution $\pi \colon \widetilde{X} \to X$, the reflexive pullback $\pi^{[*]}(\mathcal{E}, \theta)$ is $\langle (\pi^* \alpha)^{n-1} \rangle$-stable $($resp.~ semistable$)$.
    \item For any resolution $\pi \colon \widetilde{X} \to X$, the reflexive pullback $\pi^{[*]}(\mathcal{E}, \theta)$ is $\langle (\pi^* \alpha)^{n-1} \rangle$-stable $($resp.~ semistable$)$.
\end{enumerate}
\end{prop}

The proof follows the same arguments as in \cite[Lemma 4.7]{Jin25} and Proposition~\ref{prop-comparison-stability}, and is therefore omitted (strictly speaking, the verification of the $\theta$-generically invariance is required; see Lemma~\ref{lem-DO23-lemma3}).
Moreover, many of the results in \cite{Jin25} for torsion-free sheaves also hold for torsion-free Higgs sheaves, provided that the underlying normal analytic varieties have rational singularities.

\section{Complex orbifolds and Bogomolov-Gieseker inequality}
\label{sec-orbifold}
In this section, we recall the definition of orbifolds and define the intersection number between the orbifold Chern class \( \widehat{c}_2(\mathcal{E}) \) and the non-pluripolar product \( \langle \alpha^{n-2} \rangle \). We then proceed to prove Theorem~\ref{thm-BGinequality-main}.

\subsection{Complex orbifolds}
\subsubsection{The definition of complex orbifold}
Although several closely related conventions for the notion of a complex
orbifold are used in the literature, throughout this paper we follow the
convention of \cite[Subsection~3.1]{DO23}. For the reader's convenience, we
recall the basic definition below.

\begin{defn}[{\cite[Definition 3.1]{DO23}}]
A \emph{complex orbifold} \( X \) of dimension \( n \)  is a connected, second countable Hausdorff space equipped with an orbifold structure \( X_{\orb} = \{ (U_i, G_i, \pi_i) \}_{i \in I} \) satisfying the following conditions:
\begin{enumerate}[label=$(\arabic*)$]
    \item Each \( U_i \) is an open subset of \( \mathbb{C}^n \), and \( G_i \subset \mathrm{GL}_n(\mathbb{C}) \) is a finite group acting holomorphically on \( U_i \). The map \( \pi_i \colon U_i \to U_i / G_i \) is the quotient map, such that \( U_i / G_i \cong X_i \subset X \) for some (analytic) open set \( X_i \), and \( X  =  \bigcup_{i \in I} X_i\).
    
    \item (Compatibility) For any two orbifold charts \( (U_i, G_i, \pi_i) \) and \( (U_j, G_j, \pi_j) \), and any point \( x \in X_i \cap X_j \), there exists an open neighborhood \( X_k \subset X \) of \( x \) with a chart \( (U_k, G_k, \pi_k) \), and there are embeddings of charts from \( (U_k, G_k, \pi_k) \) into both \( (U_i, G_i, \pi_i) \) and \( (U_j, G_j, \pi_j) \).
Here, an \emph{embedding} from \( (U_k, G_k, \pi_k) \) into \( (U_i, G_i, \pi_i) \) consists of an embedding \( \varphi \colon U_k \to U_i \) and a group homomorphism \( \lambda \colon G_k \to G_i \), such that \( \varphi \) is \( \lambda \)-equivariant, i.e.,
\[
    \varphi(g \cdot x) = \lambda(g) \cdot \varphi(x) \quad \text{for all } g \in G_k,
    \]
    and the diagram commutes: \( i \circ \pi_k = \pi_i \circ \varphi \), where \( i \colon X_k \hookrightarrow X_i \) is the inclusion map.
\end{enumerate} 
The orbifold structure \( X_{\orb} = \{ (U_i, G_i, \pi_i) \}_{i \in I} \) is said to be \emph{effective} if for every \( i \in I\), the kernel subgroup
\[
\mathrm{Ker} \,G_i := \bigcap_{x \in U_i} \{ g \in G_i \mid g \cdot x = x \}
\]
is trivial. The orbifold is called \emph{standard} if each \( G_i \) acts freely in codimension one. In this case, each \( \pi_i \) is quasi-\'etale (i.e. finite and \'{e}tale in codimension one).
%(i.e., smooth and unramified) in codimension one.
\end{defn}

Throughout this paper, we always assume that the orbifold structure \( X_{\orb} \) is effective.
According to  \cite[Remark~3.4 (1) and (2)]{DO23}, 
the quotient space \( X \) of a complex orbifold \(X_\orb\) is a normal analytic variety with quotient singularities, and conversely, any complex analytic variety with quotient singularities admits a unique standard orbifold structure.
%Hence, we may regard complex orbifolds as normal analytic varieties with quotient singularities. (This definition is equivalent to that given in \cite{Bla96}.)

\begin{defn}[{\cite[Definition 3.3]{DO23}}]
An \emph{orbi-sheaf} $\mathcal{E}_{\orb}$ on a complex orbifold $X_{\orb} = \{(U_i, G_i, \pi_i)\}_{i \in I}$ is a collection of holomorphic $G_i$-linearized sheaves $\{\mathcal{E}_i\}_{i \in I}$ on $U_i$ satisfying the following compatibility conditions: 
\begin{itemize}
    \item For any embedding $(\varphi, \lambda) : (U_k, G_k, \pi_k) \rightarrow (U_i, G_i, \pi_i)$, there exists an isomorphism $\Phi_\varphi : \mathcal{E}_k \to \varphi^* \mathcal{E}_i$.
    \item These isomorphisms are functorial, i.e., for another embedding $(\psi, \mu) : (U_i, G_i, \pi_i) \rightarrow (U_j, G_j, \pi_j)$, we have $\Phi_{\psi \circ \varphi} = (\varphi^* \Phi_\psi) \circ \Phi_\varphi$.
\end{itemize}
The orbi-sheaf $\{\mathcal{E}_i\}_{i \in I}$ is said to be \emph{torsion-free} (resp.~ \emph{reflexive}, \emph{locally free}, \emph{torsion}) if each $\mathcal{E}_i$ is torsion-free (resp.~ reflexive, locally free, torsion).
\end{defn}
As in the smooth case, when an orbi-sheaf \( \mathcal{E}_{\orb} \) is locally free, we identify it with the corresponding vector orbi-bundle \( E_{\orb} \). 

\begin{defn}
An \emph{orbifold differential form} $\sigma_{\orb}$ on a complex orbifold $X_{\orb} = \{(U_i, G_i, \pi_i)\}_{i \in I}$ is a collection $\{\sigma_i\}_{i \in I}$ of $G_i$-invariant differential forms on the charts $\{U_i\}$, satisfying the compatibility condition that for any embeddings $(\varphi_1, \lambda_1) : (V, H, \rho) \hookrightarrow (U_1, G_1, \pi_1)$ and $(\varphi_2, \lambda_2) : (V, H, \rho) \hookrightarrow (U_2, G_2, \pi_2)$, it holds that $\varphi_1^* \sigma_1 = \varphi_2^* \sigma_2$.
In the same manner, we define \( d \)-forms and \( (p,q) \)-forms on \( X_{\orb} \).
\end{defn}

\begin{defn}\cite{Sat56}
Let $X_{\orb} = \{(U_i, G_i, \pi_i)\}_{i \in I}$ and $Y_{\orb} = \{(V_j, H_j, \rho_j)\}_{j \in J}$ be complex orbifolds, and let \( X \) and \( Y \) denote their respective quotient spaces.
A continuous map \( f : X \to Y \) is called a \emph{holomorphic morphism of orbifolds} if, for any \( x \in X \) and any open set \( Y_j \subset Y \) containing \( f(x) \) with a local orbifold chart \( (V_j, H_j, \rho_j) \), there exists an open set \( X_i \subset X \) containing \( x \) with a local orbifold chart \( (U_i, G_i, \pi_i) \), and a holomorphic map between orbifold charts
$$
(f_i, \lambda_i) : (U_i, G_i, \pi_i) \to (V_j, H_j, \rho_j)
$$
such that \( f \circ \pi_i = \rho_j \circ f_i\) and \( f_i \) is \( \lambda_i \)-equivariant. In this case, we denote the map by \( f_{\mathrm{orb}} : X_{\mathrm{orb}} \to Y_{\mathrm{orb}} \).
We say that \( f_{\mathrm{orb}} \) is \emph{proper} (resp.~ \emph{bimeromorphic}) if each local lift \( f_i \) is proper (resp.~ bimeromorphic).
\end{defn}

By \cite[Remark~3.4~(4)]{DO23}, a map \( f_{\mathrm{orb}} : X_{\mathrm{orb}} \to Y_{\mathrm{orb}} \) is a holomorphic morphism of orbifolds if and only if the morphism \( f : X \to Y \) between the underlying quotient spaces is holomorphic.

%For a compact complex orbifold \( X_{\orb} = \{(U_i, G_i, \pi_i)\}_{i \in I} \) with quotient space \( X \), we have the de Rham isomorphism \( H_{dR}^p(X_{\orb}, \mathbb{R}) \cong H^p(X, \mathbb{R}) \) and the Poincaré duality \( H^p(X, \mathbb{R}) \cong H^{2n-p}(X, \mathbb{R})^{\vee} \) (see \cite{Sat56} and \cite[Proposition 5.20]{GK20}). 
For a complex orbifold \( X_{\orb} = \{(U_i, G_i, \pi_i)\}_{i \in I} \) with quotient space \( X \), we have the de Rham isomorphism \( H_{dR}^p(X_{\orb}, \mathbb{R}) \cong H^p(X, \mathbb{R}) \) by \cite{Sat56}. 
Then, the orbifold Chern classes of vector orbi-bundles can be defined using the curvature of smooth Hermitian metrics, as described in \cite[Subsection 2.2.2]{Kob14}.

\begin{defn}
\label{defn-bundle-chernclass}
Let \( E_{\orb} = \{E_i\}_{i \in I} \) be a vector orbi-bundle over a complex orbifold \( X_{\orb} \), and let \( h_{\orb} \) be a Hermitian metric on \( E_{\orb} \), given by a collection \( \{h_i\}_{i \in I} \) of \( G_i \)-invariant Hermitian metrics on the local bundles \( E_i \), compatible with the orbifold structure. The \emph{orbifold Chern class} \( c_{p}^{\orb}(E_{\orb}) \in H^{2p}(X, \mathbb{R}) \) is then defined via the \( p \)-th orbifold Chern forms \( \Theta_p := \{ \Theta_p(E_i, h_i) \}_{i \in I} \).
\end{defn}

\subsubsection{Bott-Chern class in complex orbifolds}

\begin{defn}[{\cite[Definition 3.1, 3.3]{DO23}, \cite[Definition 2.14]{Ou24}}]
Let \( X_{\orb} = \{(U_i, G_i, \pi_i)\}_{i \in I} \) be a complex orbifold.
\begin{enumerate}[label=$(\arabic*)$]
    \item An \emph{orbifold Hermitian} (resp.~ \emph{K\"ahler}) \emph{form} \( \omega_{\orb} := \{ \omega_i \}_{i \in I} \) is a collection such that each \( \omega_i \) is a \( G_i \)-invariant Hermitian (resp.~ K\"ahler) form on \( U_i \), and the forms are compatible on overlaps.
    
    \item An \emph{orbifold closed \( (1,1) \)-form} (resp.~ \emph{\( (1,1) \)-current}) \( \theta_{\orb} := \{ \theta_i \}_{i \in I} \) consists of \( G_i \)-invariant closed \( (1,1) \)-forms (resp.~ \( (1,1) \)-currents) on \( U_i \) that are compatible on overlaps.
    
    \item Given an orbifold closed \( (1,1) \)-form (or current) \( \theta_{\orb} := \{ \theta_i \}_{i \in I} \) and an orbifold Hermitian form \( \omega_{\orb} := \{ \omega_i \}_{i \in I} \), for any real number \( c \in \mathbb{R} \), we write
    \(
    \theta_{\orb} \geq c \omega_{\orb}
    \)
    if \( \theta_i \geq c \omega_i \) holds on each chart \( U_i \).
    
    \item An \emph{orbifold quasi-psh function} \( \varphi_{\orb} := \{ \varphi_i \}_{i \in I} \) consists of \( G_i \)-invariant quasi-psh functions \( \varphi_i \) on \( U_i \), compatible on overlaps. Given an orbifold closed \( (1,1) \)-form \( \theta_{\orb} := \{ \theta_i \}_{i \in I} \), we say that \( \varphi_{\orb} \) is an \emph{orbifold \( \theta_{\orb} \)-psh function} if \( \theta_i + dd^c \varphi_i \geq 0 \) holds on each chart \( U_i \).
    
    \item An \emph{orbi-divisor} \( D_{\orb} \) is a collection \( \{ D_i \}_{i \in I} \) of reduced \( G_i \)-invariant divisors on \( U_i \), compatible on overlaps.
\end{enumerate}

\end{defn}
For a complex orbifold $X_{\orb}$, we define the \emph{orbifold Bott–Chern cohomology group} \( H^{1,1}_{BC}(X_{\orb}) \) as the quotient of $d$-closed $(1,1)$-orbifold forms modulo the $dd^c$-exact orbifold forms of degree zero. (This definition remains valid if the forms are replaced by currents.)
We now define the notion of positivity for such classes.
\begin{defn}
Let $\omega_{\orb}$ be an orbifold Hermitian form on a complex orbifold $X_\orb$. For a class $\alpha_{\orb} \in H^{1,1}_{BC}(X_{\orb})$, we define:
\begin{enumerate}[label=$(\arabic*)$]
    \item $\alpha_{\orb}$ is called \emph{K\"ahler} if it contains an orbifold K\"ahler form.
    \item $\alpha_{\orb}$ is called \emph{nef} if for every $\varepsilon > 0$, there exists an orbifold closed $(1,1)$-form $\eta_{\varepsilon, \orb} \in \alpha_{\orb}$ such that
    $
    \eta_{\varepsilon, \orb} \geq - \varepsilon \omega_{\orb}.
    $
    \item $\alpha_{\orb}$ is called \emph{big} if it contains an orbifold K\"ahler current $T_{\orb}$, i.e., there exists an orbifold closed $(1,1)$-current such that $T_{\orb} \geq \varepsilon \omega_{\orb}$ for some $\varepsilon > 0$.
    \item $\alpha_{\orb}$ is called \emph{psef} (pseudo-effective) if it contains an orbifold closed positive $(1,1)$-current.
\end{enumerate}
\end{defn}

The following lemma describes the relationship between $H^{1,1}_{BC}(X)$ and $H^{1,1}_{BC}(X_\orb)$.

\begin{lem}[{cf.~\cite[Page~13]{Wu23}, \cite[Definition~2.21]{ZZZ25}}]
\label{lem-orbifold-BC}
Let $X_{\orb} = \{(U_i, G_i, \pi_i)\}_{i \in I}$ be a compact complex orbifold with the quotient space $X$. Assume that $X$ admits a big class. Consider the map
\[
\Phi_X : H^{1,1}_{BC}(X) \to H^{1,1}_{BC}(X_{\orb}), \quad \{ \theta \} \mapsto \left\{ \{ \pi_i^* \theta \}_{i \in I} \right\},
\]
where $\theta$ is a closed positive $(1,1)$-form or current with local potentials. Then the following holds:
\begin{enumerate}[label=$(\arabic*)$]
\item The class $\{ \theta \}$ is big $($resp.~ psef$)$ if and only if $\Phi_X(\{ \theta \})$ is big $($resp.~ psef$)$.
\item The map $\Phi_X : H^{1,1}_{BC}(X) \to H^{1,1}_{BC}(X_{\orb})$ is an isomorphism.
\item We furthermore assume that $X$ admits a K\"ahler class. Then $\{ \theta \}$ is K\"ahler $($resp.~ nef$)$ if and only if $\Phi_X(\{ \theta \})$ is K\"ahler $($resp.~ nef$)$.
\end{enumerate}
\end{lem}

Before proving the lemma, we clarify the meaning of the expression \( \{ \{ \pi_i^* \theta \}_{i \in I} \} \). Take an orbifold chart \( (U_i, G_i, \pi_i) \), where \( \pi_i : U_i \to X_i \) and \( X_i \subset X \). Since \( \theta \) admits local potentials, we can write \( \theta = dd^c \varphi_i \) on \( X_i \). Thus, the pullback is given by
\[
\pi_i^* \theta := dd^c(\varphi_i \circ \pi_i).
\]
Therefore, the collection \( \{ \pi_i^* \theta \}_{i \in I} \) defines an orbifold closed $(1,1)$-form (or current), and hence \( \{ \{ \pi_i^* \theta \}_{i \in I} \} \) defines a class in \( H^{1,1}_{BC}(X_{\orb}) \).

\begin{proof}
We begin with the following claim.

\begin{claim}
\label{claim-orbifold-BC-correspondence}
Under the setting in Lemma~\ref{lem-orbifold-BC}, let $\alpha \in H^{1,1}_{BC}(X)$ and set $\alpha_{\orb} = \Phi_X(\alpha) \in H^{1,1}_{BC}(X_{\orb})$. If $T_{\orb} \in \alpha_{\orb}$ is an orbifold closed positive $(1,1)$-current, then there exists a unique closed positive $(1,1)$-current $T \in \alpha$ with local potentials such that $T_{\orb} =\{ \pi_i^* T \}_{i \in I}$.
\end{claim}

\begin{proof}[Proof of Claim~\ref{claim-orbifold-BC-correspondence}]
Suppose $T_{\orb} = \{ T_i \}_{i \in I}$, where each $T_i = dd^c \widetilde{\psi}_i$ for some $G_i$-invariant psh function $\widetilde{\psi}_i$ on $U_i$. Then $\widetilde{\psi}_i$ descends to an upper semi-continuous function $\psi_i$ on $X_i \subset X$ that is locally bounded above. Since $\pi_i$ is unramified over $X_i \setminus Z_i$ for some analytic subset $Z_i \subsetneq X_i$, it follows that $\psi_i$ is psh on $X_i \setminus Z_i$, and hence on $X_i$ by \cite[Theorem~1.7]{Dem85}. The collection $\{ \psi_i \}$ then defines a closed positive $(1,1)$-current $T$ on $X$ such that $T_{\orb} = \{ \pi_i^* T \}_{i \in I}$.

For uniqueness, suppose there exist two such currents $T, S \in \alpha$ with $\{ \pi_i^* T \}_{i \in I} = \{ \pi_i^* S \}_{i \in I}$. Then $T - S$ vanishes on each $X_i \setminus Z_i$, and hence on $X_i$ by \cite[Theorem~1.7]{Dem85}, so $T = S$.
\end{proof}

We now proceed to the proof of Lemma~\ref{lem-orbifold-BC}.

\medskip
\noindent
(1) The psef case follows directly from Claim~\ref{claim-orbifold-BC-correspondence}. The big case follows immediately, as the big cone is the interior of the psef cone.

\medskip
\noindent
(2) We begin by proving that $\Phi_X$ is injective. 
Suppose $\left\{ \{ \pi_i^* \theta \}_{i \in I} \right\} = 0$ in $H^{1,1}_{BC}(X_{\orb})$. Then there exists a smooth orbifold function $u_{\orb} = \{ u_i \}_{i \in I}$ such that
\(
dd^c(\varphi_i \circ \pi_i - u_i) = 0
\)
on each chart $(U_i, G_i, \pi_i)$ (using the notation introduced before Lemma~\ref{lem-orbifold-BC}).
As in the proof of Claim~\ref{claim-orbifold-BC-correspondence}, the functions $u_i$ descend to a continuous global function $u$ on $X$ satisfying $u_i = u \circ \pi_i$ on $U_i$. Since $\pi_i : U_i \to X_i$ is unramified outside a closed analytic subset $Z_i \subsetneq X_i$, it follows that $dd^c(\varphi_i - u) = 0$ on $X_i \setminus Z_i$. Moreover, as $\varphi_i - u$ is locally bounded near $Z_i$, it follows from \cite[Theorem~1.7]{Dem85} that $\varphi_i - u$ is pluriharmonic on $X_i$. Hence, $\{ \theta \} = \{ dd^c u \} = 0$ in $H^{1,1}_{BC}(X)$, showing that $\Phi_X$ is injective.

To prove surjectivity, take $\alpha_{\orb} \in H^{1,1}_{BC}(X_{\orb})$. Then we can write $\alpha_{\orb} = \beta_{1, \orb} - \beta_{2,\orb}$ for some big classes $\beta_{1, \orb}, \beta_{2, \orb} \in H^{1,1}_{BC}(X_{\orb})$. By (1), there exist big classes $\beta_1, \beta_2 \in H^{1,1}_{BC}(X)$ such that $\Phi_X(\beta_1) = \beta_{1, \orb}$ and $\Phi_X(\beta_2) = \beta_{2, \orb}$. Therefore,
\(
\Phi_X(\beta_1 - \beta_2) = \alpha_{\orb},
\)
proving surjectivity.

\medskip
\noindent

(3) The K\"ahler case follows from \cite[Proposition~4]{Wu23}. Note that in \cite[Proposition~4]{Wu23}, the orbifold structure is assumed to be standard; however, this assumption is not necessary.
%The K\"ahler case follows from \cite[Proposition~4]{Wu23}.\footnote{In \cite[Proposition~4]{Wu23}, the orbifold structure is assumed to be standard; however, this assumption is not necessary.} 
The nef case follows immediately, as the nef cone is the closure of the K\"ahler cone.
\end{proof}

\subsubsection{Slope stability for compact complex orbifolds with nonpluripolar product}
\label{subsubsec-slope-orbifold}
In this subsection, we define the $\langle \alpha^{n-1} \rangle$-slope for orbi-sheaves and establish some fundamental properties, including the existence of Harder–Narasimhan filtrations. Most of the arguments follow those in \cite[Section 5]{Kob14} and \cite[Appendix A]{GKP14} (see also \cite[Subsection 3.2]{DO23} and \cite[Subsection 3.2]{ZZZ25}).

Let $X_{\orb} = \{ (U_i, G_i, \pi_i) \}_{i \in I}$ be a compact complex orbifold, and let $\mathcal{E}_{\mathrm{orb}} := \{ \mathcal{E}_i \}_{i \in I}$ be an orbi-sheaf over $X_{\orb}$. Following \cite[Subsection 3.1]{ZZZ25}, on each chart $U_i$, we define the determinant line bundle $\det(\mathcal{E}_i)$ via a finite resolution of $\mathcal{E}_i$. Since the determinant is independent of the choice of resolution, the collection $\{ \det(\mathcal{E}_i) \}_{i \in I}$ is compatible on overlaps and thus defines a global line bundle. The determinant line bundle of $\mathcal{E}_{\mathrm{orb}}$ is the locally free orbi-bundle of rank one
\[
\det(\mathcal{E}_{\mathrm{orb}}) := \{ \det(\mathcal{E}_i) \}_{i \in I}.
\]
The first Chern class of $\mathcal{E}_{\mathrm{orb}}$ is defined as
\[
c_1^{\orb}(\mathcal{E}_{\mathrm{orb}}) := c_1(\det(\mathcal{E}_{\mathrm{orb}})) \in H^{1,1}_{BC}(X_{\orb}) \subset H^2(X, \mathbb{R}),
\]
where $X$ denotes the quotient space of $X_{\orb}$.

Let $\alpha \in H^{1,1}_{BC}(X)$ be a big class satisfying the vanishing property (cf.~ Definition~\ref{defn-vanishing-exceptional}). Then we can define the intersection number $c_1^{\orb}(\mathcal{E}_{\mathrm{orb}}) \cdot \langle \alpha^{n-1} \rangle$, since $c_1^{\orb}(\mathcal{E}_{\mathrm{orb}}) \in H^2(X, \mathbb{R})$ and $\langle \alpha^{n-1} \rangle \in H_2(X, \mathbb{R})$. 
Similarly, we can define the $\langle \alpha^{n-1} \rangle$-slope $\mu_{\langle \alpha^{n-1} \rangle}(\mathcal{E}_{\orb})$ of a torsion-free sheaf $\mathcal{E}_{\orb}$ in the same way.

\begin{lem}[{cf.~ \cite[Lemma 3.7]{DO23}, \cite[Lemma 3.3]{ZZZ25}}]
\label{lem-torsion}
Let $\mathcal{F}_{\orb}$ be a torsion orbi-sheaf on $X_{\orb}$. Then
\[
c_1^{\orb}(\mathcal{F}_{\orb}) \cdot \langle \alpha^{n-1} \rangle \ge 0.
\]
In particular, if $\mathcal{F}_{\orb} \subset \mathcal{G}_{\orb}$ are torsion-free orbi-sheaves of the same rank, then
\[
c_1^{\orb}(\mathcal{F}_{\orb}) \cdot \langle \alpha^{n-1} \rangle \le c_1^{\orb}(\mathcal{G}_{\orb}) \cdot \langle \alpha^{n-1} \rangle.
\]
\end{lem}

\begin{proof}
By \cite[Lemma 3.3]{ZZZ25}, there exists an effective orbi-divisor $D_{\orb}$ such that
\[
c_1^{\orb}(\mathcal{F}_{\orb}) = \{ [D_{\orb}] \} \in H^{1,1}_{BC}(X_{\orb}).
\]
Hence, $c_1^{\orb}(\mathcal{F}_{\orb})$ is psef, and the inequality follows from Remark \ref{rem-nonopluripolar-fact} (5) and Lemma \ref{lem-orbifold-BC} (1). The second assertion can be shown in the same way as in \cite[Lemma 3.7]{DO23}.
\end{proof}

We continue with a compact complex orbifold $X_\orb$ and its quotient space $X$. 
We recall the argument from \cite[Lemma 3.15]{DO23}: let $q$ be the least common multiple of $|G_i|$, let $W \subset X$ be the open subset where $G_i$ acts freely, and let $E_1, \ldots, E_s$ be the irreducible components of $X \setminus W$ of codimension one. Each $E_l$ defines an orbifold subvariety $E_{l, \orb} := \{\pi_i^{-1}(E_l)\}_{i \in I}$ and its ideal orbi-sheaf $\mathcal{I}_{l,\orb}$. Define
\[
c_1^\orb(E_{l, \orb}) := -c_1^\orb(\mathcal{I}_{l,\orb}).
\]

Let $\mathcal{E}_{\orb}$ be a torsion-free orbi-sheaf of rank $r$. As in \cite[Remark 4.8]{ZZZ25}, it descends to a torsion-free coherent sheaf $\mathcal{E}$ on $X$ via the $G_i$-invariant pushforward $({\pi_i}_* \mathcal{E}_i)^{G_i}$.
Let $\mathcal{G}_\orb := \{ \pi_i^* \mathcal{E}  / \tor\}_{i \in I}$ be the torsion-free orbi-sheaf on $X_{\orb}$. Then the following lemma holds:

\begin{lem}[{cf.~ \cite[Lemma 3.15]{DO23}}]
\label{lem-comparison-slope}
In the above setting, let $\alpha \in H^{1,1}_{BC}(X)$ be a big class satisfying the vanishing property. Then
\[
c_1^{\orb}(\mathcal{E}_{\mathrm{orb}}) \cdot \langle \alpha^{n-1} \rangle = c_1(\mathcal{E}) \cdot \langle \alpha^{n-1} \rangle + \sum\limits_{l=1}^s d_l \cdot c_1^\orb(E_{l, \orb}) \cdot \langle \alpha^{n-1} \rangle,
\]
where $d_l$ is the vanishing order of the natural morphism $\det \mathcal{G}_\orb \rightarrow \det \mathcal{E}_\orb$ along $E_{l, \orb}$, and $0 \leq d_l \leq r(q - 1)$ for all $l$.

In particular, if $X_{\orb}$ is standard, then
$c_1^{\orb}(\mathcal{E}_{\mathrm{orb}}) \cdot \langle \alpha^{n-1} \rangle = c_1(\mathcal{E}) \cdot \langle \alpha^{n-1} \rangle$ holds.
\end{lem}
\begin{proof}
By \cite[Lemma 3.15]{DO23}, we have the following identity in $H^2(X, \mathbb{R})$:
\begin{equation}
\label{eq-DO23-Lemma3.15}
c_1^{\orb}(\mathcal{E}_{\mathrm{orb}}) = c_1(\mathcal{E}) + \sum\limits_{l=1}^s d_l \cdot c_1^\orb(E_{l, \orb}).
\end{equation}
Indeed, by \cite[Lemma 3.15, Step 1]{DO23}, we have
\[
c_1^{\orb}(\mathcal{E}_{\mathrm{orb}}) = c_{1}^{\orb}(\mathcal{G}_\orb) + \sum\limits_{l=1}^s d_l \cdot c_1^\orb(E_{l, \orb}).
\]
On the other hand, by \cite[Remark 3.14 (7)]{DO23}, the determinant $\det \mathcal{E}$ is a $\mathbb{Q}$-line bundle, and by \cite[Lemma 3.15, Step 2]{DO23}, we have
$c_1(\mathcal{E}) = c^{\orb}_1(\mathcal{G}_\orb)$. Therefore, \eqref{eq-DO23-Lemma3.15} follows. The lemma now follows from \eqref{eq-DO23-Lemma3.15} and Lemma~\ref{lem-slope-invariance-VP} (see also the remark below Lemma~\ref{lem-slope-invariance-VP}).

For the last claim, if $X_{\orb}$ is standard, then $X \setminus W$ has codimension at least 2, so $s = 0$, we get the desired result.
\end{proof}

We now define slope stability for Higgs orbi-sheaves. A \textit{Higgs orbi-sheaf} $(\mathcal{E}_{\mathrm{orb}}, \theta_{\mathrm{orb}}) := \{ (\mathcal{E}_i, \theta_i) \}_{i \in I}$ on $X_{\orb}$ consists of a coherent orbi-sheaf $\mathcal{E}_{\mathrm{orb}}$ together with a morphism $\theta_{\mathrm{orb}}: \mathcal{E}_{\mathrm{orb}} \to \mathcal{E}_{\mathrm{orb}} \otimes \Omega_{X_{\orb}}^{1}$ satisfying $\theta_{\mathrm{orb}} \wedge \theta_{\mathrm{orb}} = 0$. We say that $(\mathcal{E}_{\mathrm{orb}}, \theta_{\mathrm{orb}})$ is \emph{torsion-free} (resp.~ \emph{reflexive}, \emph{locally free}) if $\mathcal{E}_{\mathrm{orb}}$ is torsion-free (resp.~ reflexive, locally free). An orbi-subsheaf $\mathcal{F}_{\mathrm{orb}}$ is called \textit{$\theta_{\mathrm{orb}}$-invariant} if $\theta_{\mathrm{orb}}(\mathcal{F}_{\mathrm{orb}}) \subset \mathcal{F}_{\mathrm{orb}} \otimes \Omega_{X_{\orb}}^{1}$.

Let $\alpha \in H^{1,1}_{BC}(X)$ be a big class satisfying the vanishing property. We say that $(\mathcal{E}_{\mathrm{orb}}, \theta_{\mathrm{orb}})$ is \emph{$\langle \alpha^{n-1} \rangle$-stable} (resp.~ \emph{$\langle \alpha^{n-1} \rangle$-semistable}) if for every non-zero $\theta_{\mathrm{orb}}$-invariant orbi-subsheaf $(\mathcal{F}_{\mathrm{orb}}, \theta_{\orb})$ with $\rk \mathcal{F}_{\mathrm{orb}} \neq \rk \mathcal{E}_{\mathrm{orb}}$, we have
\[
\mu_{\langle \alpha^{n-1} \rangle}(\mathcal{F}_{\orb}) < \mu_{\langle \alpha^{n-1} \rangle}(\mathcal{E}_{\orb}) \quad (\text{resp.~ } \mu_{\langle \alpha^{n-1} \rangle}(\mathcal{F}_{\orb}) \le \mu_{\langle \alpha^{n-1} \rangle}(\mathcal{E}_{\orb}) ).
\]
Following the proofs of \cite[Lemma 3.8, 3.9]{DO23} and \cite[Lemma 3.6, 3.7]{ZZZ25}, we obtain:

\begin{lem}\label{lem-orbifold-slope-bound}
For any torsion-free Higgs orbi-sheaf $(\mathcal{E}_\orb, \theta_\orb)$, the quantity
\[
\mu_{\langle \alpha^{n-1} \rangle}(\mathcal{E}_{\mathrm{orb}}, \theta_{\mathrm{orb}}) := \sup \{ \mu_{\langle \alpha^{n-1} \rangle}(\mathcal{F}_{\orb}) \mid \mathcal{F}_{\orb} \text{ is a } \theta_{\mathrm{orb}}\text{-invariant orbi-subsheaf} \}
\]
is finite and is achieved by some $\theta_{\mathrm{orb}}$-invariant orbi-subsheaf.
\end{lem}

\begin{cor}
\label{cor-JN-filtration}
Every torsion-free Higgs orbi-sheaf $(\mathcal{E}_\orb, \theta_\orb)$ admits a Harder–Narasimhan filtration, and every semistable one admits a Jordan–Hölder filtration.
\end{cor}

Finally, we present lemmas to be used in Subsection~\ref{subsec-BG-nonpluripolar}.

\begin{lem}[{cf.~ \cite[Lemma 3.14]{DO23}}]
\label{lem-DO23-lemma3}
Let $X_{\orb} = \{(U_i, G_i, \pi_i) \}_{i \in I}$ be a compact complex orbifold with the quotient space $X$, and let $\alpha \in H^{1,1}_{BC}(X)$ be a big class satisfying the vanishing property. For any torsion-free Higgs orbi-sheaf $(\mathcal{E}_{\orb}, \theta_{\orb})$, the following hold:
\begin{enumerate}[label=$(\arabic*)$]
\item There exists a complex orbifold $Y_{\orb} = \{(V_i, G_i, \rho_i)\}_{i \in I}$ and an orbifold morphism $f_{\orb} : Y_{\orb} \to X_{\orb}$ such that the Higgs orbi-sheaf $(f_{\orb}^* \mathcal{E}_{\orb} / \tor, f_{\orb}^* \theta_{\orb})$ is locally free. $($Here, $f_{\orb}^* \mathcal{E}_{\orb} / \tor$ is locally given by $f_i^* \mathcal{E}_i / \tor$ on each chart $f_i : V_i \to U_i$.$)$
Moreover, $(\mathcal{E}_{\orb}, \theta_{\orb})$ is $\langle \alpha^{n-1} \rangle$-$($semi$)$stable if and only if $(f_{\orb}^* \mathcal{E}_{\orb} / \tor, f_{\orb}^* \theta_{\orb})$ is $\langle f^* \alpha^{n-1} \rangle$-$($semi$)$stable, where $f : Y \to X$ is the morphism induced between the quotient spaces.

\item Assume $X_{\orb}$ is standard. Let $(\mathcal{E}, \theta)$ be the torsion-free Higgs sheaf on $X$ such that $(\mathcal{E}_{\orb}, \theta_{\orb}) = \{ (\pi_i^* \mathcal{E} / \tor, \pi_{i}^{*} \theta) \}_{i \in I}$. %where $\pi_{i}^{*} \theta$ is the induced Higgs field on $\pi_i^* \mathcal{E} / \Tor$. 
Then $(\mathcal{E}, \theta)$ is $\langle \alpha^{n-1} \rangle$-$($semi$)$stable if and only if $(\mathcal{E}_{\orb}, \theta_{\orb})$ is.
\end{enumerate}
\end{lem}

\begin{proof}
The existence of $f_{\orb} : Y_\orb \to X_\orb$ in (1) follows from \cite[Theorem 3.10]{DO23}.
For simplicity, we prove only the "only if" part of (2) in the semistable case, 
as the latter part of (1) and the "if" direction of (2) are standard (cf.~ \cite[Section 4.3]{ZZZ25}), except for the verification of the $\theta$-generically invariance (resp.~ the $\theta_\orb$-invariance).

Assume $(\mathcal{E}, \theta)$ is $\langle \alpha^{n-1} \rangle$-semistable. Let $\mathcal{F}_{\orb}=\{ \mathcal{F}_i\}_{i \in I}$ be a saturated $\theta_{\orb}$-invariant subsheaf of $\mathcal{E}_{\orb}$, and we define the torsion-free subsheaf $\mathcal{F} \subset \mathcal{E}$ induced by $\big( ({\pi_i}_* \mathcal{F}_i)^{G_i} \big)^{\vee\vee} \cap \mathcal{E}$, as in Proposition~\ref{prop-comparison-stability}. Then, by Lemma~\ref{lem-comparison-slope},
\begin{equation}
\label{eq-slope-same}
 c_1(\mathcal{F}) \cdot \langle \alpha^{n-1} \rangle
 = c_1^\orb(\mathcal{F}_{\orb}) \cdot \langle \alpha^{n-1} \rangle.
\end{equation}
Since each $\pi_i$ is quasi-\'etale and $\mathcal{E}$ is locally free in codimension one, it follows that $\mathcal{F}$ is generically $\theta_{\orb}$-invariant in codimension one. Let $\widetilde{\mathcal{F}}$ be the saturation of $\mathcal{F}$ in $\mathcal{E}$; then $\widetilde{\mathcal{F}}$ is generically $\theta$-invariant (see e.g.\ \cite[Lemma 4.16]{ZZZ25}). Using Lemma~\ref{lem-torsion}, we obtain
\[
\mu_{\langle \alpha^{n-1} \rangle}(\mathcal{F}_{\orb}) 
\underset{(\ref{eq-slope-same})}{=}
\mu_{\langle \alpha^{n-1} \rangle}(\mathcal{F})
\underset{(\text{Lem.~\ref{lem-torsion}})}{\le}
\mu_{\langle \alpha^{n-1} \rangle}(\widetilde{\mathcal{F}}) \leq \mu_{\langle \alpha^{n-1} \rangle}(\mathcal{E}).
\]
Therefore, $(\mathcal{E}_{\orb}, \theta_{\orb})$ is $\langle \alpha^{n-1} \rangle$-semistable.
\end{proof}

\begin{lem}
\label{lem-stability-varepsilon}
Let $X_{\orb} = \{(U_i, G_i, \pi_i)\}_{i \in I}$ be a compact K\"ahler  orbifold with quotient space $X$, and let $(\mathcal{E}_{\orb}, \theta_{\orb})$ be a torsion-free Higgs orbi-sheaf. Let $\alpha \in H^{1,1}_{BC}(X)$ be a big class satisfying the vanishing property.

\begin{enumerate}[label=$(\arabic*)$]
\item For any real number $C$, the set
\[A_C:=
\left\{ c_1^\orb(\mathcal{F}_\orb) \cdot \langle \alpha^{n-1} \rangle \;\middle|\; 0 \neq \mathcal{F}_\orb \subsetneq \mathcal{E}_\orb,\; c_1^\orb(\mathcal{F}_\orb) \cdot \langle \alpha^{n-1} \rangle \geq C \right\}
\]
is finite.
\item Let $\{T_\varepsilon\}_{\varepsilon > 0}$ be a family of closed positive $(1,1)$-currents with local potentials such that the difference
\(
\left\{ \langle T_\varepsilon^{n-1} \rangle \right\} - \langle \alpha^{n-1} \rangle
\)
is represented by a positive current and converges to zero as $\varepsilon \to 0$. Then, if $(\mathcal{E}_{\orb}, \theta_{\orb})$ is $\langle \alpha^{n-1} \rangle$-stable, it is $\left\{ \langle T_\varepsilon^{n-1} \rangle \right\}$-stable for sufficiently small $\varepsilon > 0$.
\end{enumerate}
\end{lem}

\begin{proof}
(1). Suppose $\mathcal{H}_\orb \subsetneq \mathcal{E}_\orb$ with $c_1^\orb(\mathcal{H}_\orb) \cdot \langle \alpha^{n-1} \rangle \in A_C$. Let $\mathcal{H}$ (resp.~ $\mathcal{E}$) be the torsion-free sheaf on $X$ induced by $\left({\pi_i}_*\mathcal{H}_i\right)^{G_i}$ (resp.~ $\left({\pi_i}_*\mathcal{E}_i\right)^{G_i}$). By Lemma~\ref{lem-comparison-slope}, we have
\begin{equation}
\label{equ-stablity-varepsilon1}
c_1^\orb(\mathcal{H}_\orb) \cdot \langle \alpha^{n-1} \rangle = c_1(\mathcal{H}) \cdot \langle \alpha^{n-1} \rangle + \sum_{l=1}^{s} d_l \cdot c_1^\orb(E_{l, \orb}) \cdot \langle \alpha^{n-1} \rangle,
\end{equation}
for some integers $0 \le d_l \le r(q-1)$, where $r := \mathrm{rk}(\mathcal{E}_\orb)$ and $q$ is the least common multiple of the orders $|G_i|$. 
Since the values $c_1^\orb(E_{l, \orb}) \cdot \langle \alpha^{n-1} \rangle$ depend only on $X_\orb$, there exists a constant $C'>0$ depending only on $r$ and $X_\orb$ such that
\[
\sum_{l=1}^{s} d_l \cdot c_1^\orb(E_{l, \orb}) \cdot \langle \alpha^{n-1} \rangle \le C'.
\]
Hence $c_1(\mathcal{H}) \cdot \langle \alpha^{n-1} \rangle$ lies in
\[
 \left\{ c_1(\mathcal{F}) \cdot \langle \alpha^{n-1} \rangle \;\middle|\; 0 \neq \mathcal{F} \subsetneq \mathcal{E},\; c_1(\mathcal{F}) \cdot \langle \alpha^{n-1} \rangle \ge C - C' \right\},
\]
which is finite by \cite[Lemma 7.3]{Jin25} (after resolving singularities if necessary). 
Since the correction term in \eqref{equ-stablity-varepsilon1}, namely $\sum_{l=1}^{s} d_l \cdot c_1^\orb(E_{l, \orb}) \cdot \langle \alpha^{n-1} \rangle$, takes only finitely many values, the quantity $c_{1}^{\orb}(\mathcal{H}_{\orb}) \cdot \langle \alpha^{n-1} \rangle$ does as well. Hence we conclude that $A_C$ is finite.

(2). Let $\omega_{\orb}$ be an orbifold K\"ahler form. By (1), there exists a constant $M$ such that
\begin{equation}
\label{eq-(C)-1}
\mu_{\langle \alpha^{n-1} \rangle}(\mathcal{E}_{\orb}) > M \ge \mu_{\langle \alpha^{n-1} \rangle}(\mathcal{F}_{\orb})
\end{equation}
for any $\theta_\orb$-invariant nonzero torsion-free orbi-sheaf $\mathcal{F}_\orb \subsetneq \mathcal{E}_\orb$.
Let $\gamma_\varepsilon := \{ \langle T_\varepsilon^{n-1} \rangle \} - \langle \alpha^{n-1} \rangle$. This is represented by a positive closed $(n-1,n-1)$-current. By \cite[Lemma 5.7.16]{Kob14} and \cite[Lemma 7.4]{Jin25}, there exists a constant $C > 0$ such that for any $0 \neq \mathcal{F}_\orb \subsetneq \mathcal{E}_\orb$,
\begin{equation}
\label{eq-(C)-2}
\mu_{\gamma_\varepsilon }(\mathcal{F}_\orb) \le C \cdot \{ \omega_\orb \} \cdot \gamma_\varepsilon.
\end{equation}
Combining \eqref{eq-(C)-1} and \eqref{eq-(C)-2}, we obtain
\begin{align*}
\mu_{\{ \langle T_\varepsilon^{n-1} \rangle \}}(\mathcal{E}_\orb) - \mu_{\{ \langle T_\varepsilon^{n-1} \rangle \}}(\mathcal{F}_\orb)
&\ge \big( \mu_{\langle \alpha^{n-1} \rangle}(\mathcal{E}_\orb) - M \big)
+ \mu_{\gamma_\varepsilon}(\mathcal{E}_\orb)
- C \cdot \{ \omega_\orb \} \cdot \gamma_\varepsilon.
\end{align*}
Since both $\mu_{\gamma_\varepsilon}(\mathcal{E}_\orb)$ and $\{ \omega_\orb \} \cdot \gamma_\varepsilon$ tend to zero as $\varepsilon \to 0$, it follows that 
\[
\mu_{\{ \langle T_\varepsilon^{n-1} \rangle \}}(\mathcal{E}_\orb) > \mu_{\{ \langle T_\varepsilon^{n-1} \rangle \}}(\mathcal{F}_\orb)
\] for sufficiently small $\varepsilon > 0$.
Therefore, $(\mathcal{E}_\orb, \theta_\orb)$ is $\{\langle T_\varepsilon^{n-1} \rangle\}$-stable for sufficiently small $\varepsilon > 0$.
\end{proof}

\subsection{Intersection number and properties}
%\subsubsection{The intersection number of orbifold chern classes and nonpluripolar products}
\label{subsec-intersection-orbifold-nonpluripolar}
Let $X$ be a compact normal analytic variety in Fujiki's class with quotient singularities in codimension $2$. Let $\mathcal{E}$ be a reflexive sheaf, and let $\alpha_1, \ldots, \alpha_{n-2} \in H^{1,1}_{BC}(X)$ be psef classes. Motivated by \cite[Section 2]{Ou25b}, we define the intersection numbers between $\widehat{c}_2(\mathcal{E})$ and $\langle \alpha_1 \cdots \alpha_{n-2} \rangle$.

%By assumption, there exists a Zariski open subset $X^{\circ} \subset X$ with $\codim(X \setminus X^{\circ}) \geq 3$ such that $X^{\circ}$ admits an orbifold structure. Since $\mathcal{E}$ is reflexive, by choosing $X^{\circ}$ appropriately, the restriction $\mathcal{E}|_{X^{\circ}}$ admits an orbi-bundle structure.

By \cite[Theorem 1.2]{Ou24} and \cite[Theorem 1]{KO25}, there exists a bimeromorphic map $f: Y \to X$, where $Y$ has only quotient singularities and admits a standard orbifold structure $Y_{\orb} = \{ (V_i, G_i, \rho_i) \}_{i \in I}$. Define $\mathcal{F} := f^*\mathcal{E} / \Tor$ and set $\mathcal{F}_i := \rho_i^{*} \mathcal{F} / \tor$. Then the sheaf $\mathcal{F}$ defines a torsion-free orbi-sheaf $\mathcal{F}_{\orb} := \{ \mathcal{F}_i \}_{i \in I}$ on $Y_{\orb}$.

According to \cite[Theorem 3.10]{DO23}, there exists a functorial resolution $p_i: U_i \to V_i$ such that each $U_i$ is smooth and the sheaf $\mathcal{H}_i := p_i^* \mathcal{F}_i / \Tor$ is locally free. Moreover, by functoriality, each $U_i$ admits a $G_i$-action, which gives rise to a (not necessarily standard) orbifold structure $Z_{\orb} := \{ (U_i, G_i, \pi_i) \}_{i \in I}$. Let $Z$ be the quotient space associated with $Z_{\orb}$, and let $p: Z \to Y$ be the morphism induced by the collection $\{ p_i \}_{i \in I}$.
Set $q := f \circ p: Z \to X$. Since $\mathcal{H}_{\orb} := \{ \mathcal{H}_i \}_{i \in I}$ is locally free, it defines an orbi-bundle.
(Note that each $\mathcal{H}_i$ is isomorphic to $\pi_i^{*} q^{*} \mathcal{E} / \Tor$ by \cite[Definition 5 below]{Wu22}.)
%\footnote{Note that $\mathcal{H}_i \cong \pi_i^{*} q^{*} \mathcal{E} / \Tor$ by \cite[Definition 5 below]{Wu22}. We will use this fact in the proof of Proposition \ref{prop-c2-nonpluri-intersection} below.}
Thus  we obtain
$$
c_2^{\orb}(\mathcal{H}_{\orb}) \in H^4(Z, \mathbb{R}).
$$
We now define the intersection number. For any psef classes $\alpha_1, \ldots, \alpha_{n-2} \in H^{1,1}_{BC}(X)$, we define
$$
\widehat{c}_2(\mathcal{E}) \cdot \langle \alpha_1 \cdots \alpha_{n-2} \rangle := c_2^{\orb}(\mathcal{H}_{\orb}) \cdot \langle q^* \alpha_1 \cdots q^* \alpha_{n-2} \rangle.
$$
Note that $\langle q^* \alpha_1 \cdots q^* \alpha_{n-2} \rangle \in H_4(Z, \mathbb{R})$, so the intersection number above is well-defined.

\begin{prop}
\label{prop-c2-nonpluri-intersection}
The intersection number $\widehat{c}_2(\mathcal{E}) \cdot \langle \alpha_1 \cdots \alpha_{n-2} \rangle$ defined above is independent of the choice of $Z$ and $q$.
\end{prop}

Before proving Proposition~\ref{prop-c2-nonpluri-intersection}, we make a few remarks. First, in the same way, one can define $\widehat{c}_1(\mathcal{E})^2 \cdot \langle \alpha_1 \cdots \alpha_{n-2} \rangle$ and $\widehat{c}_1(\mathcal{E})\cdot \widehat{c}_1(\mathcal{E}') \cdot \langle \alpha_1 \cdots \alpha_{n-2} \rangle$ for any reflexive sheaves $\mathcal{E}$ and $\mathcal{E}'$ on $X$. 
%We will later prove the well-definedness of $\widehat{c}_1(\mathcal{F}) \cdot \langle \alpha_1 \cdots \alpha_{n-2} \rangle$.

Second, if all $\alpha_1, \ldots, \alpha_{n-2}$ are nef, then by Lemma~\ref{lem-pluripolar-sing-property}, the intersection number $\widehat{c}_2(\mathcal{E}) \cdot \langle \alpha_1 \cdots \alpha_{n-2} \rangle$ coincides with the usual definition given in \cite[Definition 5.2]{GK20} or \cite[Definition 2.1]{Ou25b} (see also \cite[Section 4]{ZZZ25} for related discussion). Furthermore, if $X$ is smooth in codimension $2$ and $\alpha_1, \ldots, \alpha_{n-2}$ are nef, then the definition of $\widehat{c}_2(\mathcal{E}) \cdot \langle \alpha_1 \cdots \alpha_{n-2} \rangle$ agrees with those in \cite[Definition 4.3]{GKP16} and \cite[Remark 5]{Wu21}.
Indeed, take a strong resolution $\pi \colon \widetilde{X} \to X$ such that $\pi^* \mathcal{E} / \tor$ is locally free. Then, by definition, we have
\[
\widehat{c}_2(\mathcal{E}) \cdot \langle \alpha_1 \cdots \alpha_{n-2} \rangle = c_2(\pi^* \mathcal{E} / \tor) \cdot \langle \pi^* \alpha_1 \cdots \pi^* \alpha_{n-2} \rangle.
\]
The right-hand side is precisely the intersection number defined in \cite{GKP16} and \cite{Wu21}. In fact, the key idea of the proof of Proposition~\ref{prop-c2-nonpluri-intersection} is essentially the same as in the case where $X$ is smooth in codimension two. In that setting, one uses the invariance of $\pi^* \mathcal{E} / \mathrm{Tor}$ under any modification. In our situation, we establish the same invariance on each orbifold chart.

Finally, we introduce the following notation:
\begin{defn}
Let $X$ be a compact normal analytic variety in Fujiki's class with quotient singularities in codimension $2$. We define
\[
\widehat{c}_2(X) := \widehat{c}_2(\Omega_X^{[1]})
\quad \text{and} \quad
\widehat{c}_1(X)^2 := \widehat{c}_1(\Omega_X^{[1]})^2.
\]
For a rank $r$ reflexive sheaf $\mathcal{E}$ on $X$, the \emph{Bogomolov discriminant} is defined by
\[
\widehat{\Delta}(\mathcal{E}) := 2r \widehat{c}_2(\mathcal{E}) - (r-1) \widehat{c}_1(\mathcal{E})^2.
\]
\end{defn}

Now we return to the proof.

\begin{proof}[Proof of Proposition~\ref{prop-c2-nonpluri-intersection}]
We begin with the following setup: Let $q_1 : Z_1 \to X$ and $q_2 : Z_2 \to X$ be proper bimeromorphic maps such that each $Z_l$ has only quotient singularities, and let $Z_{l,\orb}$ denote the associated orbifold structures for $l = 1,2$.
Let $\mathcal{H}_{l, \orb} = \{ \mathcal{H}_{l, i} \}_{i \in I}$ be the orbi-bundle on $Z_{l, \orb}$ induced by $q_l$ and $\mathcal{E}$, constructed as in the setup preceding Proposition~\ref{prop-c2-nonpluri-intersection}. 
Then we have $\mathcal{H}_{l, i} \cong \pi_i^{*} q_l^{*} \mathcal{E} / \Tor$.
%As noted in the footnote of Subsection~\ref{subsec-intersection-orbifold-nonpluripolar}, we have $\mathcal{H}_{l, i} \cong \pi_i^{*} q_l^{*} \mathcal{E} / \Tor$.
Let $\varphi : Z_1 \dashrightarrow Z_2$ be the induced bimeromorphic map, and let $W$ be a resolution of the graph of $\varphi$. We denote by $r_1 : W \to Z_1$ and $r_2 : W \to Z_2$ the induced bimeromorphic maps, so that $q_1 \circ r_1 = q_2 \circ r_2$.

\[
\xymatrix@C=80pt@R=20pt{
 W\ar@{->}[r]^{r_{2}} \ar@{->}[d]_{r_{1}} & Z_{2} \ar@{->}[d]^{q_{2}} \\
 Z_{1} \ar@{->}[r]_{q_{1}} \ar@{-->}[ru]^{\varphi} & X
}
\]

Our goal is to prove the equality
\[
c_2^{\orb}(\mathcal{H}_{1,\orb}) \cdot \langle q_1^* \alpha_1 \cdots q_1^* \alpha_{n-2} \rangle = c_2^{\orb}(\mathcal{H}_{2,\orb}) \cdot \langle q_2^* \alpha_1 \cdots q_2^* \alpha_{n-2} \rangle.
\]
This is equivalent to showing
\begin{equation} \label{eq-orbifold-1}
r_1^* c_2^{\orb}(\mathcal{H}_{1,\orb}) \cdot \sigma = r_2^* c_2^{\orb}(\mathcal{H}_{2,\orb}) \cdot \sigma,
\end{equation}
where $\sigma := \langle r_1^* q_1^* \alpha_1 \cdots r_1^* q_1^* \alpha_{n-2} \rangle = \langle r_2^* q_2^* \alpha_1 \cdots r_2^* q_2^* \alpha_{n-2} \rangle$.

Using refined local charts and a partition of unity, we may reduce to the local case where each $Z_l$ admits an orbifold chart $(U_l, G_l, \pi_l)$ with $Z_l = U_l / G_l$ for $l = 1,2$. In this local setting, the orbi-sheaf $\mathcal{H}_{l, \orb}$ is given by the locally free orbi-sheaf $\mathcal{H}_l$ on $U_l$ such that $\mathcal{H}_l \cong \pi_l^{*} q_l^{*} \mathcal{E} / \Tor$.

Adapting the argument in \cite[Lemma 1.10]{Bla96}, we proceed as follows:
Let $V_l$ be a resolution of the normalization of $U_l \times_{Z_l} W$, and denote by $f_l : V_l \to U_l$ the induced bimeromorphic map. The induced map $g_l : V_l \to W$ is generically finite of degree $|G_l|$ and étale over $r_l^{-1}(Z_{l,\mathrm{reg}})$. Let $\widetilde{V}$ be a resolution of the normalization of $V_1 \times_W V_2$, and let $h_l : \widetilde{V} \to V_l$ be the induced maps. Then $h_1$ (resp.~ $h_2$) is generically finite of degree $|G_2|$ (resp.~ $|G_1|$). We set $\rho_l := f_l \circ h_l : \widetilde{V} \to U_l$.
This gives rise to the diagram:
\begin{equation}
\label{eq-diagram}
\xymatrix@C=80pt@R=30pt{
 U_{1} \ar@{->}[d]_{\pi_1}^{\text{{\tiny \'etale}}} & \ar@{->}[l]_{f_1}^{\text{{\tiny bimero.}}} V_{1} \ar@{->}[d]^{g_1}_{\shortstack{{\tiny gen.\ fin.}\\[-1pt]{\tiny deg.\ $|G_1|$}}}
 & \ar@{->}[l]_{h_1}^{\text{{\tiny gen.\ fin.\ deg.\ $|G_2|$}}}  \widetilde{V} \ar@{->}[r]^{h_2}_{\text{{\tiny gen.\ fin.\ deg.\ $|G_1|$}}} \ar@/_18pt/[ll]_{\rho_1} \ar@/^18pt/[rr]^{\rho_2} & \ar@{->}[r]^{f_2}_{\text{{\tiny bimero.}}} V_{2} \ar@{->}[d]_{g_2}^{\shortstack{{\tiny gen.\ fin.}\\[-1pt]{\tiny deg.\ $|G_2|$}}} & U_2 \ar@{->}[d]_{\pi_2}^{\text{{\tiny \'etale}}} \\
 Z_1 & \ar@{->}[l]_{r_1}^{\text{{\tiny bimero.}}} W & \ar@{=}[l] W \ar@{=}[r] & W \ar@{->}[r]^{r_2}_{\text{{\tiny bimero.}}} & Z_2
}
\end{equation}

We interpret $\sigma \in H^{2n-4}_{dR}(W, \mathbb{R})$ using the de Rham isomorphism \cite{Sat56} together with Poincaré duality \cite{GK20}. 
As a slight abuse of notation, we interpret the pairing $r_1^* c_2^{\orb}(\mathcal{H}_{1,\orb}) \cdot \sigma$ as the integral $\int_W \eta \wedge \eta'$, where we choose smooth forms $\eta$ and $\eta'$ satisfying $r_1^* c_2^{\orb}(\mathcal{H}_{1,\orb}) = \{ \eta \}$ and $\sigma = \{ \eta' \}$. With this notation, we obtain:
%\footnote{This is a slight abuse of notation. More precisely, we choose smooth forms $\eta, \eta'$ such that $r_1^* c_2^{\orb}(\mathcal{H}_{1,\orb}) = \{ \eta \}$ and $\sigma = \{ \eta' \}$, and we interpret the pairing $r_1^* c_2^{\orb}(\mathcal{H}_{1,\orb}) \cdot \sigma$ as $\int_W \eta \wedge \eta'$. This notation is used only for simplicity.}
\begin{equation}
\label{eq-orbifold-3}
r_1^* c_2^{\orb}(\mathcal{H}_{1,\orb}) \cdot \sigma = \frac{1}{|G_1|} g_1^* r_1^* c_2^{\orb}(\mathcal{H}_{1,\orb}) \cdot g_1^* \sigma 
\underset{\text{(\ref{eq-diagram})}}{=}  \frac{1}{|G_1||G_2|} \rho_1^* \pi_1^* c_2^{\orb}(\mathcal{H}_{1,\orb}) \cdot h_1^* g_1^* \sigma.
\end{equation}
A similar identity holds for $r_2^* c_2^{\orb}(\mathcal{H}_{2,\orb}) \cdot \sigma$. 
In our setting, we have
\[
\rho_1 ^* \mathcal{H}_{1} 
\cong (q_1 \circ \pi_1 \circ \rho_1)^* \mathcal{E} / \tor 
\underset{\text{(\ref{eq-diagram})}}{\cong} (q_2 \circ \pi_2 \circ \rho_2)^* \mathcal{E} / \tor 
\cong \rho_2^* \mathcal{H}_{2}.
\]
%which implies $\rho_1^*(\pi_1^* \mathcal{H}_{1} / \tor) \cong \rho_2^*(\pi_2^* \mathcal{H}_{2} / \tor)$, as in Definition-Lemma~\ref{defn-slope}. 
Therefore, we obtain $\rho_1^* \pi_1^* c_2^{\orb}(\mathcal{H}_{1,\orb}) = \rho_2^* \pi_2^* c_2^{\orb}(\mathcal{H}_{2,\orb})$. Since $h_1^* g_1^* = h_2^* g_2^*$, it follows that
\[
\rho_1^* \pi_1^* c_2^{\orb}(\mathcal{H}_{1,\orb}) \cdot h_1^* g_1^* \sigma = \rho_2^* \pi_2^* c_2^{\orb}(\mathcal{H}_{2,\orb}) \cdot h_2^* g_2^* \sigma.
\]
Therefore, by \eqref{eq-orbifold-3}, we obtain the desired identity \eqref{eq-orbifold-1}, which completes the proof.
\end{proof}
\begin{rem}
When \( X \) is smooth and \( \mathcal{E} \) is a reflexive sheaf, the second Chern class \( c_2(\mathcal{E}) \in H^{4}(X, \mathbb{R}) \) is well-defined, and one can consider its intersection number with \( \langle \alpha_{1} \cdots \alpha_{n-2} \rangle \in H_4(X, \mathbb{R})\). However, it is not clear whether this intersection number coincides with the one given in Proposition~\ref{prop-c2-nonpluri-intersection}. This issue arises for the same reason as in Remark~\ref{rem-mixed-slope}. If \( \mathcal{E} \) is locally free, then the two intersection numbers agree.
\end{rem}

At the end of this subsection, we prove several propositions that will be used later.
\begin{prop}
\label{prop-exact-sequence-2}
Let \( X \) be a compact normal analytic variety in Fujiki's class with quotient singularities in codimension 2. 
Let \( \alpha_1, \ldots, \alpha_{n-2} \in H^{1,1}_{BC}(X) \) be psef classes. 
Consider an exact sequence of reflexive sheaves:
%Consider an exact sequence of torsion-free sheaves
\begin{equation}
\label{eq-exact-locally-spplitable}
0 \to \mathcal{E} \to \mathcal{F} \overset{\phi}{\to}  \mathcal{G} \to 0,
\end{equation}
%where \( \mathcal{E} \) and \( \mathcal{F} \) are reflexive.
If the sequence \eqref{eq-exact-locally-spplitable} is locally split $($i.e., it splits over small analytic open sets$)$, then the following equality holds:
\begin{equation*}
\widehat{c}_2(\mathcal{F}) \cdot \langle \alpha_1 \cdots \alpha_{n-2} \rangle 
=
\left( 
\widehat{c}_2(\mathcal{E}) 
+ \widehat{c}_2(\mathcal{G}^{\vee\vee}) 
+ \widehat{c}_1(\mathcal{E}) \cdot \widehat{c}_1(\mathcal{G}^{\vee\vee})
\right) \cdot \langle \alpha_1 \cdots \alpha_{n-2} \rangle.
\end{equation*}
\end{prop}
\begin{proof}
As in Subsection~\ref{subsec-intersection-orbifold-nonpluripolar}, take a bimeromorphic morphism \( q \colon Z \to X \) such that \( Z \) admits an orbifold structure \( Z_{\orb} = \{ (U_i, G_i, \pi_i) \}_{i \in I} \) and an orbi-bundle \( \mathcal{F}_{\orb} = \{ \mathcal{F}_i \}_{i \in I} \) on \( Z_{\orb} \) satisfying
\begin{equation*}
\label{eq-modification-equal}
\widehat{c}_2(\mathcal{F}) \cdot \langle \alpha_1 \cdots \alpha_{n-2} \rangle = c_2^{\orb}(\mathcal{F}_{\orb}) \cdot \langle q^*\alpha_1 \cdots q^*\alpha_{n-2} \rangle.
\end{equation*}
%Similarly, take an orbi-bundles \( \mathcal{G}_{\orb} \) on \( Z_{\orb} \) associated with \( \mathcal{G} \). From the argument of \eqref{eq-2ndchern-exact} in Proposition~\ref{prop-exact-sequence},
Similarly, we can define the orbi-bundles \( \mathcal{E}_{\orb} := \{\mathcal{E}_i\}_{i \in I} \) and \( \mathcal{G}_{\orb} :=\{\mathcal{G}_i\}_{i \in I} \) on \( Z_{\orb} \) from \( \mathcal{E} \) and \( \mathcal{G} \), respectively. Then we obtain the following exact sequence on \( Z_{\orb} \):
\begin{equation*}
\label{eq-IMM24-1}
0 \to \mathcal{E}_{\orb} \to \mathcal{F}_{\orb} \overset{\phi_{\orb}}{\to} \mathcal{G}_{\orb} \to \mathcal{T}_{\orb} \to 0,
\end{equation*}
where \( \mathcal{T}_{\orb} := \Coker(\phi_{\orb}) \) is a torsion orbi-sheaf, and \( \phi_{\orb} \) is induced by the morphism \( \phi \colon \mathcal{F} \to \mathcal{G} \) in \eqref{eq-exact-locally-spplitable}. Thus it is enough to show that \( \mathcal{T}_{\orb}=0 \), and then to show that the morphism \( \phi_{\orb} \colon \mathcal{F}_{\orb} \to \mathcal{G}_{\orb} \) admits a local holomorphic section.

Write \( \phi_{\orb} = \{ \phi_i \}_{i \in I} \), where each \( \phi_i \colon \mathcal{F}_i \to \mathcal{G}_i \) is defined over \( U_i \). Since the sequence \eqref{eq-exact-locally-spplitable} is locally split, we may shrink \( U_i \) if necessary to obtain a Zariski closed subset \( W_i \subset U_i \) such that there exists a morphism \( \eta_i \colon \mathcal{G}_i \to \mathcal{F}_i \) on \( U_i \setminus W_i \) satisfying \( \phi_i \circ \eta_i = \mathrm{id}_{\mathcal{G}_i} \). This local splitting originates from the local splitting of the original morphism \( \phi \colon \mathcal{F} \to \mathcal{G} \) on \( X \). Hence, by the Riemann extension theorem, \( \eta_i \) extends holomorphically over \( U_i \), since it can be identified with a tuple of holomorphic functions that remain bounded near \( W_i \). Therefore, \( \eta_i \) provides a holomorphic section of \( \phi_i \), which proves that \( \phi_{\orb} \) is locally split, as required.
\end{proof}

\begin{lem}
\label{lem-Hodge-index-orb}
Let $X$ be a compact normal analytic variety in Fujiki's class with quotient singularities in codimension $2$. Let $\beta, \alpha_1, \ldots, \alpha_{n-2} \in H^{1,1}_{\mathrm{BC}}(X)$ be psef classes. If
\(
\langle \beta^2 \cdot \alpha_1 \cdots \alpha_{n-2} \rangle > 0,
\)
then for any reflexive sheaf $\mathcal{F}$, we have
\begin{equation*} \label{equa-hodge1}
    (\widehat{c}_1(\mathcal{F})^2 \cdot \langle \alpha_1 \cdots \alpha_{n-2} \rangle) \cdot (\langle \beta^2 \cdot \alpha_1 \cdots \alpha_{n-2} \rangle) \leq (\widehat{c}_1(\mathcal{F}) \cdot \langle \beta \cdot \alpha_1 \cdots \alpha_{n-2} \rangle)^2.
\end{equation*}
\end{lem}

\begin{proof}
We use the same notation as in the proof of Proposition \ref{prop-exact-sequence-2}. Moreover, we may assume that
\[
%%\widehat{c}_1(\mathcal{E})^2 \cdot(\langle \alpha_1 \cdots \alpha_{n-2} \rangle  = c_1^{\orb}(E_{\orb})^2 \cdot
%\langle q^{*}\alpha_1 \cdots q^{*}\alpha_{n-2} \rangle)
%\quad \text{and} \quad
\widehat{c}_1(\mathcal{F}) \cdot 
\langle \beta \cdot  \alpha_1 \cdots \alpha_{n-2} \rangle 
= c_1^{\orb}(F_{\orb}) \cdot\langle q^{*} \beta \cdot q^* \alpha_1 \cdots q^*\alpha_{n-2} \rangle.
\]
Therefore, the result follows by applying Lemma~\ref{lem-Hodge-index} to $\gamma := c_1^{\orb}(F_{\orb}) \in H^{1,1}_{\mathrm{BC}}(Z)$ on $Z$.
\end{proof}

\subsection{Proof of Theorem \ref{thm-BGinequality-main}}
\label{subsec-BG-nonpluripolar}
In this subsection, we prove Theorem~\ref{thm-BGinequality-main}. The proof reduces to the orbifold case, so we will establish the corresponding results for orbifolds in what follows.

\begin{lem}[{cf.~\cite[Corollary~14.13]{Dem12}, \cite[Theorem~6]{Wu23}}]
\label{lem-demailly-app-orbifold}
Let $X_{\orb} = \{ (U_i, G_i, \pi_i) \}_{i \in I}$ be a compact K\"ahler orbifold with the quotient space $X$.
%Let $X$ be an $n$-dimensional normal analytic K\"ahler variety with quotient singularities. 
Let $\gamma$ be a smooth $(1,1)$-form with local potentials on $X$, and let $T$ be a closed $(1,1)$ current with local potentials on $X$ such that $T \ge \gamma$.

Then there exists a sequence of closed $(1,1)$ currents $T_m$ on $X$ in the cohomology class $\{T\}$, each with local potentials given by the logarithm of a sum of squares of holomorphic functions, and a sequence $\varepsilon_m > 0$ with $\varepsilon_m \to 0$, such that the following properties hold:
\begin{enumerate}[label=$(\arabic*)$]
    \item $T_m$ is less singular than $T$ and converges weakly to $T$.
    \item $T_m \ge \gamma - \varepsilon_m \omega$ for some K\"ahler form $\omega$.
    \item For any prime divisor $E$ on $X$, 
    $$
    \nu(T, E) - \frac{n}{m} \le \nu(T_m, E) \le \nu(T, E).
    $$
\end{enumerate}
\end{lem}

\begin{proof}
By replacing $T$ with $T - \gamma$, we may assume $\gamma = 0$. Let $\theta$ be a smooth $(1,1)$-form with local potentials such that $\{ \theta \} = \{ T \} \in H^{1,1}_{BC}(X)$, and let $\varphi$ be a $\theta$-psh function with $T = \theta + dd^c \varphi$. 
%Let $X_{\orb} = \{ (U_i, G_i, \pi_i) \}_{i \in I}$ be an orbifold structure on $X$ with local charts $\pi_i : U_i \to X_i$, and 
Let $\Phi_X : H^{1,1}_{BC}(X) \to H^{1,1}_{BC}(X_{\orb})$ be the natural isomorphism as in Lemma~\ref{lem-orbifold-BC}. 
Choose K\"ahler forms $\omega$ on $X$ and $\omega_{\orb}$ on $X_{\orb}$ such that $\Phi_X(\{\omega\}) = \{\omega_{\orb}\}$.

Define $\varphi_{\orb} := \{ \pi_i^* \varphi \}_{i \in I}$ and $\theta_{\orb} := \{ \pi_i^* \theta \}_{i \in I}$. Since $dd^c \varphi_{\orb} \ge -\theta_{\orb}$, the orbifold version of Demailly's approximation theorem \cite[Theorem~6]{Wu23} implies that there exist a sequence of orbifold quasi-psh functions $\varphi_{m,\orb}$ and a sequence $\varepsilon_m > 0$ with $\varepsilon_m \to 0$ such that:
\begin{itemize}
    \item $\varphi_{m,\orb}$ has the same singularities as the logarithm of a sum of squares of $G_i$-invariant holomorphic functions on $U_i$, after possibly shrinking the orbifold charts.
    \item $\varphi_{\orb} \le \varphi_{m,\orb}$ and $\varphi_{m,\orb}$ converges to $\varphi_{\orb}$ pointwise and in $L^1(X_{\orb})$.
    \item $dd^c \varphi_{m,\orb} \ge - \theta_{\orb} - \varepsilon_m \omega_{\orb}$.
    \item $\nu(\varphi_{\orb}, x) - \frac{n}{m} \le \nu(\varphi_{m,\orb}, x) \le \nu(\varphi_{\orb}, x)$ for any $x \in X_{\orb}$.
\end{itemize}

As in Claim~\ref{claim-orbifold-BC-correspondence} (using the fact that $\varphi_{m,\orb}$ is locally bounded above on each $U_i$ and applying \cite[Theorem 1.7]{Dem85}), the psh function $\varphi_{m,\orb}$ induces a $(\theta + \varepsilon_m \omega)$-psh function $\varphi_m$ on $X$. We then define a closed positive $(1,1)$ current with local potentials by
$$
S_m := dd^c \varphi_m + \theta + \varepsilon_m \omega.
$$
By construction, we have $S_m \in \{ \theta + \varepsilon_m \omega \}$. As the $G_i$-invariant holomorphic functions on $U_i$ descend to holomorphic functions on $X_i$, the local potentials of $S_m$ have the same singularities as the logarithm of a sum of squares of holomorphic functions on $X$. Moreover, $S_m$ converges weakly to $T = \theta + dd^c \varphi$, and for any prime divisor $E$ on $X$, we have $\nu(T, E) - \frac{n}{m} \le \nu(S_m, E) \le \nu(T, E)$. Setting $T_m := S_m - \varepsilon_m \omega$, we find that $T_m$ is less singular than $T$, and satisfies all required properties.
\end{proof}

\begin{thm}
\label{thm-Higgs-BG-movable}
Let $X_{\mathrm{orb}} = \{(U_i, G_i, \pi_i)\}_{i \in I}$ be a compact K\"ahler orbifold with the quotient space $X$, and let $(E_{\orb}, \theta_{\orb})$ be a Higgs orbi-bundle of rank $r$ on $X_{\orb}$. Let $\alpha \in H^{1,1}_{BC}(X)$ be a big class on $X$ satisfying the vanishing property $($see Definition~\ref{defn-vanishing-exceptional}$)$.
If $(E_{\orb}, \theta_{\orb})$ is $\langle \alpha^{n-1} \rangle$-semistable, then the following Bogomolov--Gieseker inequality holds:
$$
\left( 2r \, c_{2}^{\orb}(E_\orb) - (r-1) \, c_{1}^{\orb}(E_\orb)^2 \right) \cdot \langle \alpha^{n-2} \rangle \geq 0.
$$
\end{thm}

In what follows, to simplify notation, we write
\[
\widehat{\Delta}_{\orb}(E_{\orb}) := 2r \, c_{2}^{\orb}(E_\orb) - (r-1) \, c_{1}^{\orb}(E_\orb)^2 \in H^{4}(X, \mathbb{R})
\]
for an orbi-bundle \( E_\orb \) of rank \( r \).

\begin{proof}
Note that $X$ is $\mathbb{Q}$-factorial and has rational singularities by \cite[Proposition 5.15]{KM98} and \cite[Proposition]{Bla96}. By replacing $\alpha$ with its positive part $P(\alpha)$, which is modified nef by \cite[Lemma~2.6]{DH23}, and using Proposition~\ref{prop-nonpluripolar-modifiednef}, we may assume that $\alpha$ is big and modified nef.
%Let $(\mathcal{E}, \theta)$ be the reflexive Higgs sheaf on $X$ induced by $(E_{\orb}, \theta_{\orb})$ as in \cite[Remark~4.8]{ZZZ25}. We may assume that $(\mathcal{E}, \theta)$ is $\langle \alpha^{n-1} \rangle$-semistable.

\begin{case}[stable case]
First, we consider the case where $(E_\orb, \theta_\orb)$ is $\langle \alpha^{n-1} \rangle$-stable. Let $\omega$ be a K\"ahler form on $X$. We construct a family of K\"ahler currents $\{ T_\varepsilon \}_{\varepsilon > 0}$ with analytic singularities such that:

\begin{enumerate}[label=(\Alph*)]
    \item $T_\varepsilon \in \alpha + \varepsilon \{\omega\}$ and $\nu(T_\varepsilon, E) = 0$ for any prime divisor $E$ on $X$.
    \item For each $k = 1, \ldots, n$, the class $\{ \langle T_\varepsilon^k \rangle \} - \langle \alpha^k \rangle$ is represented by a closed positive current and converges to $0$ as $\varepsilon \to 0$.
\end{enumerate}

Let $T_{\min}$ be the closed positive $(1,1)$ current with minimal singularities in $\alpha$. By applying Demailly’s approximation theorem in Lemma~\ref{lem-demailly-app-orbifold}, for any $\varepsilon > 0$, we obtain a current $S_\varepsilon$ with analytic singularities in the class $\alpha$ such that $S_\varepsilon \ge -\frac{\varepsilon}{2} \omega$. We then define
$$
T_\varepsilon := S_\varepsilon + \varepsilon \omega \in \alpha + \varepsilon \{\omega\},
$$
which is a K\"ahler current. We now verify properties (A) and (B):

\medskip
(A) Since $\alpha$ is big and modified nef, we have $\nu(\alpha, E) = \nu(T_{\min}, E) = 0$ for any prime divisor $E$ (see Subsection~\ref{subsubsec-DZD-nonpluripolar}). By property (3) in Lemma~\ref{lem-demailly-app-orbifold}, it follows that $\nu(T_\varepsilon, E) = 0$ for any prime divisor $E$.

(B) By Lemma~\ref{lem-demailly-app-orbifold}, the current $S_\varepsilon$ is less singular than $T_{\min}$, so $T_\varepsilon$ is less singular than $T_{\min} + \varepsilon \omega$ in the class $\alpha + \varepsilon \{\omega\}$. Hence we have
$$
\{ \langle T_\varepsilon^k \rangle \}
\underset{\text{(Rem.~\ref{rem-nonopluripolar-fact} (1))}}{\geq} 
\{ \langle (T_{\min} + \varepsilon \omega)^k \rangle \}
\underset{\text{(Rem.~\ref{rem-nonopluripolar-fact} (2))}}{\geq} 
\{ \langle T_{\min}^k \rangle \} + (\varepsilon \{\omega\})^k 
\underset{\text{(Lem.~\ref{lem-nonpluripolar-minimlal-singular})}}{\geq} 
\langle \alpha^k \rangle.
$$
On the other hand, since $\langle (\alpha + \varepsilon \{\omega\})^k \rangle \ge \{ \langle T_\varepsilon^k \rangle \}$ and $\langle (\alpha + \varepsilon \{\omega\})^k \rangle \to \langle \alpha^k \rangle$ as $\varepsilon \to 0$ by Remark~\ref{rem-nonopluripolar-fact}~(1) and (2), we conclude that
$$
\{ \langle T_\varepsilon^k \rangle \} \to \langle \alpha^k \rangle
\quad \text{as $\varepsilon \to 0$}. 
$$

%This shows that $T_\varepsilon$ is the desired sequence of K\"ahler currents satisfying the required properties.

\medskip
By Lemma~\ref{lem-stability-varepsilon}~(2), there exists a constant $\varepsilon_0 > 0$ such that for every $0 < \varepsilon < \varepsilon_0$, the Higgs sheaf $(E_\orb, \theta_\orb)$ is $\{ \langle T_{\varepsilon}^{n-1} \rangle \}$-stable. Fix such $\varepsilon > 0$. We now prove the following claim. 

\begin{claim}
\label{claim-Kahler-orbiform}
There exist a compact K\"ahler orbifold $Y_{\orb}$, an orbifold morphism $f_{\orb} : Y_{\orb} \to X_{\orb}$, and a semipositive orbifold form $\tau_{\orb}$ on $Y_{\orb}$ such that:
\begin{enumerate}[label=(\alph*)]
    \item The Higgs orbi-bundle $(f_{\orb}^* E_{\orb}, f_{\orb}^* \theta_{\orb})$ on $Y_{\orb}$ is $\{ \tau_{\orb} \}^{n-1}$-stable, and
    $$
    \widehat{\Delta}_{\orb}(f_{\orb}^* E_{\orb}) \cdot \{ \tau_{\orb} \}^{n-2} = \widehat{\Delta}_{\orb}(E_{\orb}) \cdot \{ \langle T_{\varepsilon}^{n-2} \rangle \}.
    $$
    \item The following Bogomolov--Gieseker inequality holds:
    $$
    \widehat{\Delta}_{\orb}(f_{\orb}^* E_{\orb}) \cdot \{ \tau_{\orb} \}^{n-2} \ge 0.
    $$
\end{enumerate}
\end{claim}

\begin{proof}[Proof of Claim \ref{claim-Kahler-orbiform}]
Define the current $T_i := \pi_i^* T_{\varepsilon}$ on $U_i$, and let $T_{\varepsilon, \mathrm{orb}} := \{ T_i \}_{i \in I}$ be the associated orbifold current on $X_{\orb}$. By the construction of $T_{\varepsilon}$ in Lemma~\ref{lem-demailly-app-orbifold}, the local potentials of $T_i$ are given by logarithms of sums of $G_i$-invariant holomorphic functions on $U_i$. Then, by applying the argument in \cite[Corollary~14.13]{Dem12}, we obtain $G_i$-equivariant resolutions $f_i : V_i \to U_i$ such that
$$
f_i^* T_i = \tau_i + [D_i],
$$
where $\tau_i$ is a smooth semipositive $(1,1)$-form and $D_i$ is an $f_i$-exceptional divisor on $V_i$. (The fact that each $D_i$ is exceptional follows from condition~(A). The $G_i$-equivariance of $f_i$ follows from the existence of functorial resolutions; see also \cite[Theorem~3.10]{DO23}. Note that each $f_i$ can be obtained as a sequence of blowups.)

The collection of spaces $\{ V_i \}_{i \in I}$ and morphisms $\{ f_i \}_{i \in I}$ defines an orbifold $Y_{\orb} = \{ (V_i, G_i, \rho_i) \}_{i \in I}$ and an orbifold morphism
$$
f_{\orb} : Y_{\orb} \to X_{\orb}.
$$
Let $Y$ be the quotient space of $Y_{\orb}$, and let $f : Y \to X$ be the induced morphism.

Now define the semipositive orbifold form $\tau_{\orb} := \{ \tau_i \}_{i \in I}$ and the orbi-divisor $D_{\orb} := \{ D_i \}_{i \in I}$. (Since each $f_i$ is a composition of blowups, the family $\{ D_i \}_{i \in I}$ defines a well-defined orbi-divisor by \cite[Subsection~3.3]{DO23}.)
Then we have
$$
f^*_{\orb} T_{\varepsilon, \orb} = [D_{\orb}] + \tau_{\orb}.
$$
The orbi-divisor $D_{\orb}$ induces a divisor $D$ on $Y$. Moreover, there exists a $(1,1)$-form $\tau$ on $Y$, induced by $\tau_{\orb}$, with locally continuous psh potentials such that
\begin{equation}
\label{eq-gamma-nonpluripolar-0}
f^* T_{\varepsilon} = [D] + \tau.
\end{equation}
By the construction (see also Lemma~\ref{lem-orbifold-BC}), we have $\tau_{\orb} = \{ \rho_i^* \tau \}_{i \in I}$. Let $\beta := \{ \tau \} \in H^{1,1}_{BC}(Y)$.

We now show that for each $k = 1, \ldots, n$:
\begin{equation}
\label{eq-gamma-nonpluripolar}
\beta^k = \{ \langle (f^* T_{\varepsilon})^k \rangle \}.
\end{equation}
(Note that this equality is understood via Poincaré duality, identifying $H^{2k}(Y, \mathbb{R})$ with $H_{2n - 2k}(Y, \mathbb{R})$.)
Since $T_{\varepsilon}$ has singularities in codimension at least $2$ by Condition~(A), it has small unbounded locus. Hence, so does $f^* T_{\varepsilon}$. Thus, by Lemma~\ref{lem-nonpluripolar-unbounded} and \eqref{eq-gamma-nonpluripolar-0}, we have
\begin{equation}
\label{eq-gamma-nonpluripolar-2}
\langle (f^* T_{\varepsilon})^k \rangle = \langle \tau^k \rangle.
\end{equation}

Moreover, since $\tau_{\orb}$ is semipositive, the class $\{ \tau_\orb\}$ is nef. Therefore, by Lemma~\ref{lem-orbifold-BC} (3), the class $\beta = \{ \tau \} \in H^{1,1}_{BC}(Y)$ is also nef, and hence
\begin{equation}
\label{eq-gamma-nonpluripolar-3}
\beta^k = \langle \beta^k \rangle.
\end{equation}

Next, we verify that $\beta$ is a big class. Since $T_{\varepsilon}$ is a K\"ahler current, we have $\tau \ge f^* \omega$ for some K\"ahler form $\omega$ on $X$. By \cite[Subsection~3.3]{DO23}, there exists an effective $f$-exceptional $\mathbb{Q}$-divisor $E$ on $Y$ such that the class $\{ f^* \omega - [E] \}$ is K\"ahler. Hence, we can write:
$$
\beta = \{ \tau \} = \{ \tau - f^* \omega \} + \{ [E] \} + \{ f^* \omega - [E] \},
$$
which is the sum of a psef class and a K\"ahler class, and thus is big.
Since the local potential of $\tau$ is continuous, $\tau$ is the current with minimal singularities in $\beta$. Therefore, by Lemma~\ref{lem-nonpluripolar-minimlal-singular}, we obtain:
\begin{equation}
\label{eq-gamma-nonpluripolar-4}
\langle \beta^k \rangle = \{ \langle \tau^k \rangle \}.
\end{equation}
Combining \eqref{eq-gamma-nonpluripolar-2} - \eqref{eq-gamma-nonpluripolar-4}, we conclude the desired identity \eqref{eq-gamma-nonpluripolar}.

We are now ready to verify conditions~(a) and (b) in Claim~\ref{claim-Kahler-orbiform}.

\medskip
(a) By Lemma~\ref{lem-DO23-lemma3}~(2) and Lemma~\ref{lem-stability-varepsilon}~(2), the Higgs orbi-bundle $(f^*_{\orb} E_\orb, f^*_{\orb} \theta_\orb)$ is $\langle (f^* T_{\varepsilon})^{n-1} \rangle$-stable. 
%Since $\beta = \{ \tau \}$, $\tau_{\orb} = \{ \rho_i^* \tau \}_{i \in I}$, and 
By \eqref{eq-gamma-nonpluripolar}, we conclude that $(f^*_{\orb} E_\orb, f^*_{\orb} \theta_\orb)$ is $\{ \tau_{\orb} \}^{n-1}$-stable.
Applying Proposition~\ref{prop-c2-nonpluri-intersection} together with \eqref{eq-gamma-nonpluripolar}, we obtain:
$$
\widehat{\Delta}_\orb(E_\orb) \cdot \{ \langle T_{\varepsilon}^{n-2} \rangle \}
\underset{\text{(Prop.~\ref{prop-c2-nonpluri-intersection})}}{=}
\widehat{\Delta}_\orb(f^*_{\orb} E_\orb) \cdot \{ \langle f^* T_{\varepsilon}^{n-2} \rangle \}
\underset{\eqref{eq-gamma-nonpluripolar}}{=}
\widehat{\Delta}_\orb(f^*_{\orb} E_\orb) \cdot \beta^{n-2}.
$$
Moreover, since $\beta^{n-2}$ coincides with $\{ \tau_{\orb} \}^{n-2}$ in $H^{2n - 4}(Y, \mathbb{R})$, we have
%since $\tau_{\orb} = \{ \rho_i^* \tau \}_{i \in I}$, the class $\beta^{n-2}$ coincides with $\{ \tau_{\orb} \}^{n-2}$ in $H^{2n - 4}(Y, \mathbb{R})$. Therefore,
$$
\widehat{\Delta}_\orb(f^*_{\orb} E_\orb) \cdot \beta^{n-2}
= \widehat{\Delta}_\orb(f^*_{\orb} E_\orb) \cdot \{ \tau_{\orb} \}^{n-2}.
$$

(b) Let $\omega_{Y,\orb}$ be an orbifold K\"ahler form on $Y_\orb$. By Lemma~\ref{lem-stability-varepsilon}~(2) and condition~(a), the Higgs orbi-bundle $(f^*_{\orb} E_\orb, f^*_{\orb} \theta_\orb)$ is $\{ \tau_{\orb} + \delta \omega_{Y,\orb} \}^{n-1}$-stable for sufficiently small $\delta > 0$. Since $\{ \tau_{\orb} + \delta \omega_{Y,\orb} \}$ is a K\"ahler class, the Bogomolov--Gieseker inequality from \cite[Corollary~1.3]{ZZZ25} implies:
$$
\widehat{\Delta}_{\orb}(f^*_{\orb} E_\orb) \cdot \{ \tau_{\orb} + \delta \omega_{Y,\orb} \}^{n-2} \ge 0.
$$
Letting $\delta \to 0$ completes the proof.
\end{proof}

\medskip
We now complete the proof of the Bogomolov--Gieseker inequality for a stable Higgs bundle. From Claim~\ref{claim-Kahler-orbiform}, we obtain
$$
\widehat{\Delta}_\orb(E_\orb) \cdot \{ \langle T_{\varepsilon}^{n-2} \rangle \}
=
\widehat{\Delta}_\orb(f^*_{\orb} E_\orb) \cdot \{ \tau_{\orb} \}^{n-2} \ge 0.
$$
Therefore, by property (B), we conclude the Bogomolov--Gieseker inequality for the $\langle \alpha^{n-1} \rangle$-stable Higgs sheaf:
\begin{equation*}
\widehat{\Delta}_\orb(E_\orb) \cdot \langle \alpha^{n-2} \rangle
=\lim_{\varepsilon \to 0} \widehat{\Delta}_\orb(E_\orb) \cdot \{ \langle T_{\varepsilon}^{n-2} \rangle \}
\ge 0.
\end{equation*}
\end{case}

\begin{case}[Semistable case]
Finally, we consider the case where $(E_\orb, \theta_\orb)$ is $\langle \alpha^{n-1} \rangle$-semistable. We apply the argument of \cite[Proposition~5.1]{Chen22} in the orbifold setting, referring to \cite[Lemma 3.13]{DO23}.
By Corollary~\ref{cor-JN-filtration}, we may take a Jordan–Hölder filtration of $(E_{\orb}, \theta_{\orb})$ with respect to $\langle \alpha^{n-1} \rangle$:
$$
0 = \mathcal{E}_{0, \orb} \subset \mathcal{E}_{1, \orb} \subset \cdots \subset \mathcal{E}_{l, \orb} = E_{\orb},
$$
where each successive quotient $\mathcal{G}_{i, \orb} := \mathcal{E}_{i, \orb} / \mathcal{E}_{i-1, \orb}$ is a $\langle \alpha^{n-1} \rangle$-stable torsion-free orbi-sheaf of rank $r_i$, equipped with a Higgs field $\theta$, and satisfies $\mu_{\langle \alpha^{n-1} \rangle}(\mathcal{G}_{i, \orb}) = \mu_{\langle \alpha^{n-1} \rangle}(E_{\orb})$.
%(We note that the existence of such a Jordan–Hölder filtration follows by an argument similar to that in \cite[Lemma~3.7]{ZZZ25}.)
By Lemma~\ref{lem-DO23-lemma3}~(1), there exists an orbifold morphism $f_{\orb} : Y_{\orb} \to X_{\orb}$ from a compact K\"ahler orbifold $Y_{\orb}$ such that the pullback $f_{\orb}^* E_{\orb}$ admits a filtration
$$
0 = \mathcal{E}'_{0, \orb} \subset \mathcal{E}'_{1, \orb} \subset \cdots \subset \mathcal{E}'_{l, \orb} = f_{\orb}^* E_{\orb}
$$
on $Y_{\orb}$, where each successive quotient $\mathcal{G}'_{i, \orb} := \mathcal{E}'_{i, \orb} / \mathcal{E}'_{i-1, \orb}$ is a $\langle f^* \alpha^{n-1} \rangle$-stable orbi-bundle satisfying
\begin{equation}
\label{eq-semistable-slope}
\mu_{\langle f^* \alpha^{n-1} \rangle}(\mathcal{G}'_{i, \orb}) = \mu_{\langle f^* \alpha^{n-1} \rangle}(f^*_{\orb} E_{\orb}).
\end{equation}
Here, $f : Y \to X$ denotes the morphism between the quotient spaces induced by $f_{\orb} : Y_{\orb} \to X_{\orb}$.
%In particular, we may consider the Chern classes of each $\mathcal{G}'_{i, \orb}$ (see \cite{Wu23}).
As in \cite[Proposition~5.1]{Chen22} (cf.~\cite[Chapter~1, Subsection~6.c]{Nak04}, \cite[Equation (3.6.1)]{Lan04}), we obtain
\begin{align}
\begin{split}
\label{eq-Langer-orbifold}
\frac{\widehat{\Delta}_{\orb}(f^*_{\orb} E_{\orb})}{r} \cdot \langle f^* \alpha^{n-2} \rangle
&= \sum_{i=1}^{l} \frac{\widehat{\Delta}_{\orb}(\mathcal{G}'_{i,\orb})}{r_i} \cdot \langle f^* \alpha^{n-2} \rangle \\
&\quad - \frac{1}{r} \sum_{1 \le i < j \le l} r_i r_j \left( \frac{c_1^{\orb}(\mathcal{G}'_{i,\orb})}{r_i} - \frac{c_1^{\orb}(\mathcal{G}'_{j,\orb})}{r_j} \right)^2 \cdot \langle f^* \alpha^{n-2} \rangle.
\end{split}
\end{align}
%(Here, $\widehat{\Delta}_{\orb}(\mathcal{E}) := 2r\, c_2^{\orb}(\mathcal{E}) - (r-1) c_1^{\orb}(\mathcal{E})^2$.)
 %(This identity follows by the same argument as in the proof of Lemma~\ref{lem-Langer-ineq}.)
Since each $\mathcal{G}'_{i, \orb}$ is $\langle f^* \alpha^{n-1} \rangle$-stable, we have $\widehat{\Delta}_{\orb}(\mathcal{G}'_{i, \orb}) \cdot \langle f^* \alpha^{n-2} \rangle \ge 0$ by the result of case 1.
Furthermore, by the Hodge index theorem (Lemma~\ref{lem-Hodge-index}), we have
\begin{align}
\begin{split}
\left( \frac{c_1^{\orb}(\mathcal{G}'_{i,\orb})}{r_i} - \frac{c_1^{\orb}(\mathcal{G}'_{j,\orb})}{r_j} \right)^2 \cdot \langle f^* \alpha^{n-2} \rangle
\le \frac{\left( \mu_{\langle f^* \alpha^{n-1} \rangle}(\mathcal{G}'_{i,\orb}) - \mu_{\langle f^* \alpha^{n-1} \rangle}(\mathcal{G}'_{j,\orb}) \right)^2}
{\langle f^* \alpha^n \rangle}
 \underset{(\ref{eq-semistable-slope})}{=} 0.
\end{split}
\end{align}

Thus, combining Proposition~\ref{prop-c2-nonpluri-intersection} with \eqref{eq-Langer-orbifold}, we conclude:
$$
\widehat{\Delta}_{\orb}(E_{\orb}) \cdot \langle \alpha^{n-2} \rangle
\underset{\text{(Prop.~\ref{prop-c2-nonpluri-intersection})}}{=}
\widehat{\Delta}_{\orb}(f^*_{\orb} E_{\orb}) \cdot \langle f^* \alpha^{n-2} \rangle
\underset{\eqref{eq-Langer-orbifold}}{\ge}
0.
$$
\end{case}

\end{proof}

\begin{ex}
\label{exa-Keum}
In Theorem \ref{thm-Higgs-BG-movable}, we consider a current $T_{\varepsilon}$ instead of the class $\alpha + \varepsilon \{ \omega \}$, because in general, for a modification $\pi \colon \widetilde{X} \to X$, it may happen that $\langle \pi^* \alpha \rangle \ne \langle \pi^* \alpha - \{ E \} \rangle$ for a $\pi$-exceptional divisor $E$. We illustrate this with two examples.

First, let $\pi \colon Y \to \mathbb{C}\mathbb{P}^2$ be the blow-up at a point, with exceptional divisor $E$, and let $\alpha$ be a nonzero K\"ahler class on $\mathbb{C}\mathbb{P}^2$. Then the class $\pi^* \alpha - \{ E \}$ is a K\"ahler class on $Y$. Hence, it is clear that $\langle \pi^* \alpha \rangle \ne \langle \pi^* \alpha - \{ E \} \rangle$.

Another example, originally due to Keum \cite{Keum08}, involves a fake projective plane $S$ with an automorphism group $G$. Keum showed that if the order of $G$ is $7$, then the minimal resolution of $S/G$ is not of general type. Set $X := S/G$, and let $\pi \colon Y \to X$ be the minimal resolution. By \cite[Lemma~4.5]{Keum08}, there exists a $\mathbb{Q}$-effective exceptional divisor $E$ such that
$$
K_Y + E = \pi^* K_X.
$$
If we had $\langle \pi^* c_1(K_X) \rangle = \langle \pi^* c_1(K_X) - \{ E \} \rangle$, then it would follow that $\langle K_Y \rangle = \pi^* K_X$. Since $K_X$ is nef and big, this would imply that $K_Y$ is big, contradicting the fact that $Y$ is not of general type.
Keum's example also demonstrates that even if the canonical divisor $K_X$ is big, the canonical divisor $K_{\widetilde{X}}$ on a resolution $\widetilde{X} \to X$ is not necessarily big. Therefore, assumption (1) in Theorem~\ref{thm-main-big-canonical} does not imply assumption (2).
\end{ex}

Now we prove Theorem~\ref{thm-BGinequality-main} by establishing the following more general result.

\begin{thm}[$\supset$ Theorem~\ref{thm-BGinequality-main}]
\label{thm-BGinequality-nonpluripolar}
Let $X$ be a compact normal analytic variety in Fujiki's class, with quotient singularities in codimension $2$.  
Let $\mathcal{E}$ be a rank $r$ reflexive sheaf, and let $\alpha \in H^{1,1}_{BC}(X)$ be a big class.  
Assume that one of the following holds:
\begin{enumerate}[label=$(\arabic*)$]
    \item $\mathcal{E}$ is $\langle \alpha^{n-1} \rangle$-semistable.
    \item $X$ has rational singularities, and there exists a Higgs field $\theta$ such that the Higgs sheaf $(\mathcal{E}, \theta)$ is $\langle \alpha^{n-1} \rangle$-semistable.
\end{enumerate}

If $\alpha$ satisfies the vanishing property $($see Definition~\ref{defn-vanishing-exceptional}$)$, then the Bogomolov--Gieseker inequality holds:
$$ 
\widehat{\Delta}(\mathcal{E}) \cdot \langle \alpha^{n-2} \rangle 
= \left(2r \, \widehat{c}_2(\mathcal{E}) - (r-1)\, \widehat{c}_1(\mathcal{E})^2 \right) \cdot \langle \alpha^{n-2} \rangle 
\ge 0.
$$
\end{thm}

In particular, when $X$ is Moishezon, the assumption on the \textit{vanishing property} is not needed, thanks to Lemma~\ref{lem-VP}. Hence, Theorem~\ref{thm-BGinequality-nonpluripolar} implies Theorem~\ref{thm-BGinequality-main}.

\begin{proof}
We consider case~(1).  
As in Subsection~\ref{subsec-intersection-orbifold-nonpluripolar}, take an orbifold modification \( f: Y \to X \) such that \( Y \) admits a standard orbifold structure \( Y_{\orb} \). Set \( \mathcal{F} := f^*\mathcal{E} / \Tor \) on \( Y \), which induces an orbi-sheaf \( \mathcal{F}_{\orb} \) on \( Y_{\orb} \).  
By taking a functorial resolution on each orbifold chart, we obtain a bimeromorphic morphism \( p: Z \to Y \) such that \( Z \) admits a (not necessarily standard) orbifold structure \( Z_{\orb} \) and the pullback of \( \mathcal{F}_\orb \) by $p$ induces an orbi-bundle \( \mathcal{H}_{\orb} \) on \( Z_{\orb} \). Up to a modification, we may assume that $Z_\orb$ is K\"ahler. Set \( q := f \circ p : Z \to X \).

Since \( \mathcal{E} \) is \( \langle \alpha^{n-1} \rangle \)-semistable, it follows from Proposition~\ref{prop-comparison-stability} that \( \mathcal{F} \) is \( \langle f^*\alpha^{n-1} \rangle \)-semistable. By Lemma~\ref{lem-DO23-lemma3}~(2), \( \mathcal{F}_{\orb} \) is also \( \langle f^*\alpha^{n-1} \rangle \)-semistable, and by Lemma~\ref{lem-DO23-lemma3}~(1), so is \( \mathcal{H}_{\orb} \) with respect to \( \langle q^*\alpha^{n-1} \rangle \).

By the definition of the intersection number in Proposition~\ref{prop-c2-nonpluri-intersection}, we obtain
\[
\widehat{\Delta}(\mathcal{E}) \cdot \langle \alpha^{n-2} \rangle
=
\widehat{\Delta}_{\orb}(\mathcal{H}_{\orb}) \cdot \langle q^*\alpha^{n-2} \rangle.
\]
Therefore, the conclusion follows from Theorem~\ref{thm-Higgs-BG-movable}, since the orbi-bundle \( \mathcal{H}_{\orb} \) is \( \langle q^*\alpha^{n-1} \rangle \)-semistable.

The argument for case~(2) is similar.  
Note that the assumption that \( X \) has rational singularities is necessary to ensure that the pullback of a Higgs sheaf remains a Higgs sheaf.
\end{proof}

\section{Miyaoka--Yau inequality}
\label{sec-MY-inequality}
\subsection{Proof of Theorem \ref{thm-main-big-canonical}}

Before giving the proof, we recall a result from \cite[Theorem 4.9]{Jin25} concerning stability. To do so, we first introduce some terminology.

Let $f: X \dashrightarrow Y$ be a bimeromorphic map between compact normal analytic varieties. Let $p: Z \rightarrow X$ and $q: Z \rightarrow Y$ be resolutions of $f$:
$$
\xymatrix@C=40pt@R=30pt{
& Z \ar[ld]_p \ar[rd]^q & \\
X \ar@{-->}[rr]^f && Y
}
$$
We say that $f$ is a \emph{bimeromorphic contraction} if every $p$-exceptional divisor is also $q$-exceptional (cf.~ \cite[Definition 1.1]{Hu-Keel00}). This is equivalent to saying that the inverse map $f^{-1}: Y \dashrightarrow X$ has no exceptional divisors.

Let $\alpha_X$ and $\alpha_Y$ be big classes on $X$ and $Y$, respectively, such that $\alpha_Y = f_* \alpha_X$.
Motivated by \cite[Definition A.10]{DHY23}, we say that $f$ is \emph{$\alpha_X$-negative} if there exists an effective $q$-exceptional divisor $E$ such that
$$
p^*\alpha_X = q^*\alpha_Y + \{E\},
$$
and the support of $p_*E$ is contained in the support of the $f$-exceptional divisors.
Under this setting, the following result holds:

\begin{thm}\cite[Theorem 4.9]{Jin25}
\label{thm-jinnnouchi}
Under the setting above, assume that $f$ is an $\alpha_X$-negative bimeromorphic contraction. Let $\mathcal{E}_X$ and $\mathcal{E}_Y$ be torsion-free sheaves on $X$ and $Y$, respectively. Suppose that $\mathcal{E}_X$ is isomorphic to $f^*\mathcal{E}_Y$ outside the $f$-exceptional locus on $X$.
Then, $\mathcal{E}_X$ is $\langle \alpha_X^{n-1} \rangle$-semistable if and only if $\mathcal{E}_Y$ is $\langle \alpha_Y^{n-1} \rangle$-semistable. 
\end{thm}

We now turn to the proof of Theorem \ref{thm-main-big-canonical}.

\begin{proof}[Proof of Theorem \ref{thm-main-big-canonical}]
We first show that under either assumption (1) or (2), the reflexive cotangent sheaf $\Omega_X^{[1]}$ is $\langle c_1(K_X)^{n-1} \rangle$-semistable. This follows from the same argument as in \cite[Example 4.11]{Jin25}.
Indeed, under assumption (1), $X$ admits a canonical model $X_{\mathrm{can}}$ by \cite[Theorem 1.2]{BCHM10}. According to \cite[Theorem A]{Gue16}, the reflexive cotangent sheaf $\Omega_{X_{\mathrm{can}}}^{[1]}$ is $c_1(K_{X_{\mathrm{can}}})^{n-1}$-semistable. Let $f: X \dashrightarrow X_{\mathrm{can}}$ be the birational map induced by the $K_X$-MMP. Since $f$ is a $c_1(K_X)$-negative bimeromorphic contraction, Theorem~\ref{thm-jinnnouchi} implies that $\Omega_X^{[1]}$ is $\langle c_1(K_X)^{n-1} \rangle$-semistable.
The same conclusion holds under assumption (2), by applying the $K_{\widetilde{X}}$-MMP to a projective manifold $\widetilde{X}$. (Note that under assumption (2), $X$ is Moishezon).

Now, define the reflexive sheaf $\mathcal{E} := \Omega_X^{[1]} \oplus \mathcal{O}_X$, and consider the Higgs field $\theta$ given by
$$
\begin{array}{cccc}
\theta \colon & \mathcal{E} := \Omega_X^{[1]} \oplus \mathcal{O}_X & \rightarrow & \mathcal{E} \otimes \Omega_X^{[1]}= (\Omega_X^{[1]} \oplus \mathcal{O}_X) \otimes \Omega_X^{[1]}\\
& (a, b) & \mapsto & (0, a).
\end{array}
$$
By Remark~\ref{rem-nonopluripolar-fact} (3) and (4), we have
$$
\mu_{\langle c_1(K_X)^{n-1} \rangle}(\Omega_X^{[1]}) = \frac{1}{n} c_1(K_X) \cdot \langle c_1(K_X)^{n-1} \rangle 
\underset{\text{(Rem. \ref{rem-nonopluripolar-fact} (4))}}{=} \frac{1}{n} \langle c_1(K_X)^n \rangle 
\underset{\text{(Rem. \ref{rem-nonopluripolar-fact} (3))}}{>} 0.
$$
As shown in the proof of \cite[Proposition 2.8]{IMM24}, the Higgs sheaf $(\mathcal{E}, \theta)$ is $\langle c_1(K_X)^{n-1} \rangle$-stable (note that Lemma~\ref{lem-slope-same} is required for the argument to hold).
Therefore, applying the Bogomolov--Gieseker inequality in Theorem~\ref{thm-BGinequality-nonpluripolar} to $(\mathcal{E}, \theta)$, we obtain
$$
\left(2(n+1)\widehat{c}_2(X) - n \widehat{c}_1(X)^2\right) \cdot \langle c_1(K_X)^{n-2} \rangle 
\underset{\text{(Prop. \ref{prop-exact-sequence-2})}}{=} 
\widehat{\Delta}(\mathcal{E})\cdot \langle c_1(K_X)^{n-2} \rangle 
\ge 0.
$$
We recall that any klt variety admits quotient singularities in codimension two, according to \cite[Lemma 5.8]{GK20}.
\end{proof}

\subsection{Proof of Theorem \ref{thm-main-big-anticanonical}}

Next, we consider the Miyaoka–Yau inequality in the case where $-K_X$ is big. To this end, we first recall the notion of the \emph{canonical extension sheaf} (cf.~ \cite{Tian92}, ~\cite[Section 4]{GKP22}, \cite[Section 3]{DGP24}).
Let $X$ be a normal analytic variety, and assume that $K_X$ is $\Q$-Cartier. Then, by \cite[Section 4]{GKP22}, we have
$c_1(-K_X) \in H^{1}(X, \Omega_{X}^{1})$.
Since $H^1(X, \Omega_{X}^{1}) = \operatorname{Ext}^1(\mathcal{O}_{X}, \Omega_{X}^{1})$, this gives rise to an extension of $\Omega_{X}^{1}$ by $\mathcal{O}_{X}$:
$$
0 \to \Omega_{X}^{1} \to \mathcal{W}_{X} \to \mathcal{O}_{X} \to 0.
$$
Taking the dual of this sequence, we obtain
$$
0 \to \mathcal{O}_{X} \to \mathcal{E}_{X} \to \mathcal{T}_{X} \to 0.
$$
The sheaf $\mathcal{E}_{X}$ is called the \emph{canonical extension sheaf}. According to \cite[Section 4]{GKP22}, this sequence is locally split.

We now recall the definition of K-semistability via the delta-invariant $\delta(X)$ (cf.~\cite{FO18}, \cite{BJ20} for the case where $-K_X$ is ample, and \cite{DZ22}, \cite{Xu23} for the case where $-K_X$ is merely big).
Let $X$ be a projective klt variety with $-K_X$ big. For any birational morphism $\pi : Y \to X$ and any prime divisor $E$ on $Y$, we define the discrepancy by
$$
A_{X}(E) := 1 + \operatorname{ord}_{E}(K_{Y/X}),
$$
and the $S$-invariant by
$$
S_{X}(E) := \frac{1}{\operatorname{vol}(-K_X)} \int_0^{\infty} \operatorname{vol}(-\pi^*K_X - tE) \, dt.
$$
The \emph{delta-invariant} $\delta(X)$ is then defined as
$$
\delta(X) := \inf_E \frac{A_{X}(E)}{S_{X}(E)},
$$
where the infimum is taken over all prime divisors $E$ over $X$.
We say that $X$ is \emph{K-semistable} if $\delta(X) \ge 1$.

Now, we turn to the proof of Theorem \ref{thm-main-big-anticanonical}.

\begin{proof}[Proof of Theorem \ref{thm-main-big-anticanonical}]
We consider only case (1), as case (2) is analogous to the proof of Theorem~\ref{thm-main-big-canonical}.
Since $X$ is K-semistable, it follows from \cite[Theorems 1.1, 1.2, and Corollary 3.5]{Xu23} that the section ring of $-K_X$
$$
R(X, -K_X) := \bigoplus_{m \in r \mathbb{N}} H^0(X, -mK_X),
$$
is finitely generated and $Z := \operatorname{Proj}(R(X, -K_X))$ is a K-semistable klt Fano variety, where $r$ denotes the Cartier index of $-K_X$.
Furthermore, according to \cite[Remark 2.4 and Theorem 3.3]{DGP24}, the canonical extension sheaf $\mathcal{E}_Z$ is $c_1(-K_Z)^{n-1}$-semistable.

Let $f \colon X \dashrightarrow Z$ be the natural rational map, and take a resolution $W$ with morphisms $p \colon W \to X$ and $q \colon W \to Z$ resolving $f$:
$$
\xymatrix@C=40pt@R=30pt{
& W \ar[ld]_p \ar[rd]^q & \\
X \ar@{-->}[rr]^f && Z
}
$$
Since the section ring $R(X, -K_X)$ is finitely generated, $f$ is a birational contraction by \cite[Lemma 1.6]{Hu-Keel00}. Hence, there exists an $f$-exceptional effective divisor $D$ such that
$$
-K_X = p_* q^*(-K_Z) + D.
$$
By the negativity lemma \cite[Lemma 3.39]{KM98}, there exists an effective divisor $B$ on $W$ such that
$$
p^*(-K_X) = q^*(-K_Z) + B.
$$
Therefore, we have an isomorphism $p^* \mathcal{E}_X \cong q^* \mathcal{E}_Z$ outside the support of $B$.
Moreover, the map $f$ is $c_1(-K_X)$-negative. Indeed, since the birational map $f \colon X \dashrightarrow Z$ is a contraction, its resolution $q \colon W \to Z$ is also a contraction. Thus,
\[
q_* p^*(-K_X) = q_* q^*(-K_Z) = -K_Z,
\]
which implies that $q_* B = 0$. Therefore, we have $p_* B = D$, and it follows that $p_* B$ is $f$-exceptional. Hence, $f$ is $c_1(-K_X)$-negative.

By applying Theorem~\ref{thm-jinnnouchi}, it follows that the canonical extension sheaf $\mathcal{E}_X$ is $\langle c_1(-K_X)^{n-1} \rangle$-semistable.
Finally, applying the Bogomolov--Gieseker inequality from Theorem~\ref{thm-BGinequality-nonpluripolar} to $\mathcal{E}_X$, we obtain
\begin{equation*}
\left(2(n+1)\widehat{c}_2(X) - n\widehat{c}_1(X)^2\right) \cdot \langle c_1(-K_X)^{n-2} \rangle 
\underset{\text{(Prop. \ref{prop-exact-sequence-2})}}{=}
\widehat{\Delta}(\mathcal{E}_X) \cdot \langle c_1(-K_X)^{n-2} \rangle 
\geq 0.
\end{equation*}
\end{proof}

\begin{rem}
After this paper was submitted, the structure of \(X\) in the equality cases of the Miyaoka--Yau inequalities in Theorems \ref{thm-main-big-canonical} and \ref{thm-main-big-anticanonical} was clarified in \cite{Jin25-2} and \cite{ZZZ26}. According to \cite{ZZZ26}, if equality holds in Theorem \ref{thm-main-big-canonical}, then the canonical model is quasi-covered by the unit ball. According to \cite{Jin25-2} and \cite{ZZZ26}, if equality holds in Theorem \ref{thm-main-big-anticanonical}, then the anticanonical model is quasi-covered by $\mathbb{C}\mathbb{P}^n$.
\end{rem}

At the end of this paper, we also note that the Miyaoka--Yau inequality without the non-pluripolar product does not hold in general when \(K_X\) is merely big.

\begin{lem}
\label{lem-counterexample-MY}
Let $X$ be a minimal projective manifold of general type with $n=\dim X\geq 3$. Then the following assertions hold.
\begin{enumerate}[label=$(\arabic*)$]
    \item If $n$ is odd, then, for $r\gg 0$, the blow-up $X_r$ of $X$ at $r$ distinct points satisfies
    \[
    (2(n+1)c_2(X_r)-nc_1(X_r)^2)\cdot K_{X_r}^{n-2}<0.
    \]
    \item If $n$ is even, then there exists a smooth curve $C\subset X$ such that the blow-up $X_C$ of $X$ along $C$ satisfies
    \[
    (2(n+1)c_2(X_C)-nc_1(X_C)^2)\cdot K_{X_C}^{n-2}<0.
    \]
\end{enumerate}
\end{lem}

This lemma follows from the following Proposition \ref{prop-blowup-chern}.

\begin{prop}\label{prop-blowup-chern}
Let $\pi:\widetilde{X}\rightarrow X$ be the blow-up of a projective manifold $X$ along a smooth center $Y$.
\begin{enumerate}[label=$(\alph*)$]
    \item If $Y$ is a point, then
    \begin{equation}\label{equa-point}
        \begin{split}
            &(2(n+1)c_2(\widetilde{X})-nc_1(\widetilde{X})^2)\cdot K_{\widetilde{X}}^{n-2}-(2(n+1)c_2(X)-nc_1(X)^2)\cdot K_X^{n-2}\\
            &=(-1)^{n}4n(n-1)^{n-2}.
        \end{split}
    \end{equation}
    \item If $Y$ is a smooth curve of genus $g$, then
    \begin{equation}\label{equa-curve}
        \begin{split}
            &(2(n+1)c_2(\widetilde{X})-nc_1(\widetilde{X})^2)\cdot K_{\widetilde{X}}^{n-2}-(2(n+1)c_2(X)-nc_1(X)^2)\cdot K_X^{n-2}\\
            &=(-1)^{n-1}(n-2)^{n-2}(6K_X\cdot[Y]+3(n-2)(2g-2)).
        \end{split}
    \end{equation}
\end{enumerate}
\end{prop}

\begin{proof}[Proof of Lemma \ref{lem-counterexample-MY}]
Assertion (1) follows immediately from \eqref{equa-point}. We now explain (2). Since $K_X$ is nef and big, $mK_X$ is base point free for $m\gg 1$. We can choose general divisors $H_1,\cdots,H_{n-1}\in|mK_X|$ for $m\gg1$ such that $C_m:=H_1\cap\cdots\cap H_{n-1}$ is smooth. Then, by the adjunction formula,
    $$K_{C_m}=(K_X+H_1+\cdots +H_{n-1})|_{C_m}=(1+(n-1)m)K_X|_{C_m},$$
    and hence $\deg K_{C_m}=(1+(n-1)m)m^{n-1}K_X^{n}.$
    Therefore
\[
\begin{aligned}
	6K_X\cdot [C_m]+3(n-2)(2g(C_m)-2)
	=\Bigl(6+3(n-2)(1+(n-1)m)\Bigr)m^{n-1}K_X^n  
	\to +\infty
\end{aligned}
\]
as \(m\to +\infty\), since \(K_X^n>0\).
    Since $n$ is even, \eqref{equa-curve} gives the desired result for $m \gg 1$.
\end{proof}

\begin{proof}[Proof of Proposition \ref{prop-blowup-chern}]
We recall the results of \cite[Subsection 15.4]{Ful93}.
Let \(Y \subset X\) be a submanifold of codimension \(d\), and let \(E=\pi^{-1}(Y)\) be the exceptional divisor of \(\pi\).
The blow-up diagram is as follows:
\[
\begin{tikzcd}
   E \arrow[r,"j"] \arrow[d,"g"]   & \widetilde{X} \arrow[d,"\pi"]   \\
   Y \arrow[r,"i"]  & X
\end{tikzcd}
\]
Let $N$ be the rank $d$ normal bundle of $Y$ in $X$, and identify $E$ with $\mathbb{P}(N)$. Then $N_{E}\widetilde{X}$ is $\mathcal{O}_N(-1)$. Let $\mathcal{O}(1)=\mathcal{O}_N(1)$ and $\xi=c_1(\mathcal{O}(1))$. Recall from \cite[Theorem 15.4]{Ful93} that
\begin{equation}
    \label{equa-diff-chern}
    c(\widetilde{X})-\pi^*c(X)=j_*(g^*c(Y)\cdot\alpha),
\end{equation}
where 
\begin{equation}
    \label{equa-defn-alpha}\alpha=\frac{1}{\xi}\left[\sum\limits_{m=0}^d g^*c_{d-m}(N)-(1-\xi)\sum\limits_{m=0}^d(1+\xi)^mg^*c_{d-m}(N)\right].
\end{equation}
Since $N_{E}\widetilde{X}=\mO_N(-1)$ and
    $j^*[E]=c_1(N_{E}\widetilde{X})=-\xi$, we have
    \begin{equation}
    \label{equa-jm}
    j_*(\xi^{m-1})=(-1)^{m-1}j_*(j^*[E]^{m-1})=(-1)^{m-1}[E]^m.
    \end{equation}

First, we prove \eqref{equa-point}. Note that $N$ is a vector space of rank $n$ over $Y$, and hence $c(N)=c(Y)=1$. Thus $K_{\widetilde{X}} = \pi^{*}K_{X} + (n-1)[E]$ and 
    \begin{equation}
    \label{equa-2ndchern-pt}
        \begin{split}
        c_2(\widetilde{X})-\pi^*c_2(X)
         \underset{\text{\eqref{equa-diff-chern} and \eqref{equa-defn-alpha}}}{=}
         \frac{n(n-3)}{2}j_{*}\xi
         \underset{\text{\eqref{equa-jm}}}{=}
        \frac{n(n-3)}{2}[E]^2.
        \end{split}
    \end{equation}
    Since $\pi^*K_X|_E=0$ and $[E]^{n}=(-1)^{n-1}$, we obtain
    \begin{equation*}\label{equa-example-1}
        \begin{split}
             c_2(\widetilde{X})\cdot K_{\widetilde{X}}^{n-2}-c_2(X)\cdot K_X^{n-2}
             =
             (c_2(\widetilde{X})-\pi^{*}c_2(X))\cdot K_{\widetilde{X}}^{n-2}
             \underset{\eqref{equa-2ndchern-pt}}{=}(-1)^{n-1}(n-1)^{n-2} \frac{n(n-3)}{2}.
        \end{split}
    \end{equation*}
    Therefore, by a direct calculation, \eqref{equa-point} follows from $K_{\widetilde{X}}^n-K_X^n=(-1)^{n-1}(n-1)^n$.

  Second, we prove \eqref{equa-curve}. Fix $p \in Y$ and set $F=[g^{-1}(p)]$. Since $N$ is the normal bundle of $Y$ in $X$ and $c_1(N) \in H^2(Y,\mathbb{Z}) \cong \mathbb{Z}$, we have
\begin{equation}\label{equa-degreeN}
j_*(g^*c_1(N))=\deg_{Y}(N)[F]=\left(2g-2-K_X\cdot [Y]\right)[F]
\in H^4(\widetilde{X},\mathbb{Z}).
\end{equation}
Since $c(N)=1+c_1(N)$ and $c(Y)=1+c_1(Y)$, the same argument as in the first case gives
    \begin{equation}\label{equa-example-3}
        \begin{split}
        &c_2(\widetilde{X})-\pi^*c_2(X) \\
    & \underalign{\text{\eqref{equa-diff-chern} and \eqref{equa-defn-alpha} }}{=}
    -\frac{(n-1)(n-4)}{2}j_*(\xi)-(n-3)j_*(g^*c_1(N))-(n-2)j_*(g^*c_1(Y))\\
        &\underalign{\text{\eqref{equa-jm} and \eqref{equa-degreeN}}}{=}\frac{(n-1)(n-4)}{2}[E]^2+\big((n-3)K_X\cdot [Y]+2g-2\big)[F].
        \end{split}
    \end{equation}
    We compute $c_2(\widetilde{X})\cdot K_{\widetilde{X}}^{n-2}-c_2(X)\cdot K_X^{n-2}$. Since $\pi_*(E^m)=i_*g_*(-\xi)^{m-1}=0$ for all $m\leq n-2$, because $\dim E-\dim Y=n-2$, we have
    \begin{equation}
    \label{equa-chern-equiv}
    \pi^*c_2(X)\cdot K_{\widetilde{X}}^{n-2}=c_2(X)\cdot K_X^{n-2}.
    \end{equation}
    Therefore it remains to compute $(c_2(\widetilde{X})-\pi^*c_2(X))\cdot K_{\widetilde{X}}^{n-2}$. For this, it is enough to compute $[E]^2\cdot K_{\widetilde{X}}^{n-2}$ and $[F]\cdot K_{\widetilde{X}}^{n-2}$. Since $Y$ is a curve, we note that
    \begin{equation}\label{equa-example-vanish}
\pi^*K_X^{m}\cdot[E]^k=j^*\pi^*K_X^{m}\cdot(-1)^{k-1}\xi^{k-1}=0\ \text{ for any $m\geq 2$.}
    \end{equation}
    Moreover, since $g : E=\mathbb{P}(N) \to Y$ is a projective bundle with fiber $F \cong \mathbb{P}^{n-2}$, we have
    \begin{equation}\label{equa-example-projection}
    g_{*}\xi^{n-2}= 1 \quad \text{and} \quad
    g_{*}\xi^{n-1}=-\deg_Y(N)=-(2g-2)+K_X\cdot [Y].
    \end{equation}
    Thus we obtain
    \begin{equation}\label{equa-example-5}
        \begin{split}
            [E]^2\cdot K_{\widetilde{X}}^{n-2}
            &\underalign{}{=}j^*[E]\cdot j^*K_{\widetilde{X}}^{n-2}
            = j^*[E]\cdot j^{*}\big(\pi^{*}(K_X)+(n-2)E\big)^{n-2}\\
            &\underalign{}{=}(-\xi)\cdot\big(g^*(K_X|_Y)-(n-2)\xi\big)^{n-2}\\
            &\underalign{\text{\eqref{equa-example-vanish}}}{=}(-1)^{n-2}(n-2)^{n-2}\left(g_{*}\xi^{n-2}\cdot (K_X \cdot [Y])- g_*\xi^{n-1}
            \right)\\
            &\underalign{\eqref{equa-example-projection}}{=}(-1)^{n-2}(n-2)^{n-2}(2g-2).
        \end{split}
    \end{equation}
    Since $\pi^*K_X|_F$ is trivial and $E|_{F}\cong \mathcal{O}_{\mathbb{P}^{n-2}}(-1)$, we have
    \begin{equation}\label{equa-example-6}
            [F]\cdot K_{\widetilde{X}}^{n-2}=(\pi^*K_X|_F+(n-2)E|_F)^{n-2}
            =(-1)^{n-2}(n-2)^{n-2}.
    \end{equation}
    Therefore,
    \begin{equation}\label{equa-example-7}
    \begin{split}
        &c_2(\widetilde{X})\cdot K_{\widetilde{X}}^{n-2}-c_2(X)\cdot K_X^{n-2} 
        \underalign{\eqref{equa-chern-equiv}}{=} (c_2(\widetilde{X})-\pi^*c_2(X))\cdot K_{\widetilde{X}}^{n-2}\\
        &\underalign{\text{\eqref{equa-example-5} and  \eqref{equa-example-6}}}{=}(-1)^{n-2}(n-2)^{n-2}(n-3)\big(K_X\cdot[Y]+(n-2)(g-1)\big).
    \end{split}
    \end{equation}
    By the same calculation as in \eqref{equa-example-1}, we also obtain
    \begin{equation}\label{equa-example-4}
        \begin{split}
            K_{\widetilde{X}}^n-K_X^n
            &=(-1)^{n-2}(n-2)^{n-2}(2K_X\cdot[Y]+(n-2)(2g-2)).
        \end{split}
    \end{equation}
    Combining \eqref{equa-example-7} with \eqref{equa-example-4}, we complete the proof.
\end{proof}

\bibliographystyle{alpha}
\bibliography{ref_MY.bib}
\end{document}